\documentclass[final]{amsart}
\usepackage[alphabetic, nobysame]{amsrefs}
\usepackage{hyperref}
\usepackage[T1,T2A]{fontenc}
\usepackage[utf8]{inputenc}
\usepackage[english]{babel}
\usepackage{amsmath,amssymb}
\usepackage{enumerate}
\usepackage{xcolor}
\allowdisplaybreaks

\makeatletter
\@namedef{subjclassname@2020}{%
	\textup{2020} Mathematics Subject Classification}
\makeatother
\newcommand{\NN}{\mathbb{N}}
\newcommand{\ZZ}{\mathbb{Z}}
\newcommand{\QQ}{\mathbb{Q}}
\newcommand{\RR}{\mathbb{R}}

\newcommand{\CC}{\mathbb{C}}

\DeclareMathOperator{\Id}{Id}

\DeclareMathOperator{\Diam}{Diam}
\DeclareMathOperator{\Aff}{Aff}
\DeclareMathOperator{\Span}{Span}
\DeclareMathOperator{\Conv}{Conv}

\newcommand{\Kwapien}{Kwapie\'{n}}
\newcommand{\avint}{\mathop{\,\rlap{-}\!\!\int}\nolimits}

\newtheorem{Thm}{Theorem}[section]

\newtheorem{defi}[Thm]{Definition}
\newtheorem{lem}[Thm]{Lemma}
\newtheorem{prop}[Thm]{Proposition}

\numberwithin{equation}{section}

\begin{document}
\title{A multidimensional solution to additive homological equations}
\author[A. Ber]{Aleksei Ber}
\address{
Department of Mathematics \\National University of Uzbekistan\\
	Vuzgorodok, 100174\\ Tashkent, Uzbekistan\\
\emph{E-mail~:} {\tt ber@ucd.uz}
}
\author[M.J. Borst]{Matthijs Borst}
\address{Delft Institute of Applied Mathematics \\Delft University of Technology\\
	P.O. Box 5031,
	2600 GA\\ Delft, The Netherlands\\
\emph{E-mail~:} {\tt  m.j.borst@outlook.com   }
}
\author[S.J. Borst]{Sander Borst}
\address{Centrum Wiskunde \& Informatica\\
    P.O. Box 94079,
    1090 GB\\ Amsterdam, The Netherlands\\
    \emph{E-mail~:} {\tt  sander.borst@cwi.nl  }
}
\thanks{This project has received funding from the European Research Council (ERC) under the European Union's Horizon 2020 research and innovation programme (grant agreement QIP--805241)}
\author[F. Sukochev]{Fedor Sukochev}
\address{School of Mathematics and Statistics, University of New South Wales, Kensington, 2052, NSW, Australia  
\emph{E-mail~:} {\tt f.sukochev@unsw.edu.au}
}

\begin{abstract}
	In this paper we prove that for a finite-dimensional real normed space $V$, every bounded mean zero function $f\in L_\infty([0,1];V)$ can be written in the form
	$f = g\circ T - g$ for some $g\in L_\infty([0,1];V)$ and some ergodic invertible measure preserving transformation $T$ of $[0,1]$.
	Our method moreover allows us to choose $g$,  for any given  $\varepsilon>0$, to be such that $\|g\|_\infty\leq (S_V+\varepsilon)\|f\|_\infty$, where $S_V$ is the Steinitz constant corresponding to $V$.
\end{abstract}
	\maketitle		
\section{Introduction}
    
	Given a bounded function $f$ on the unit interval, with mean zero, can we find a measure preserving transformation $T$, and a bounded function $g$, such that
    \begin{align}\label{ahe}
		f= g\circ T -g,
    \end{align}
	with equality holding almost everywhere? We call this the homological equation, and while it has been extensively studied in the scalar-valued setting, little is known about the homological equation for vector-valued functions. These problems are what we consider here.

   	We shall always assume that the interval $[0,1]$ is equipped with the standard Lebesgue measure $\lambda$. 
    The equation \eqref{ahe}, also known as the coboundary equation, was studied by Anosov for a fixed operator  $T$ in \cite{Anosov}, where it was demonstrated that such an equation with continuous or even analytic left hand side on the torus may have a measurable but not integrable  solution. This study arose because of a comment made by Kolmogorov in \cite{Kolmogorov}.
    We remark that by \cite[Theorem 1]{Anosov}, if $f$ is integrable and has a measurable solution $g$ for some $T$, then $f$ must be mean zero.
	Building upon this direction, it is interesting to note that Bourgain \cite{Bourgain} considered
	a closely related variant of the problem, showing that for a compact abelian group
	with finitely many components, any mean zero function $f\in L^p(G)$, for
	$p\in(1,\infty)$, admits a decomposition
	\[
		f=\sum\limits_{j=1}^J(f_j-\tau(a_j)f_j),
	\]
	for $f_j\in L^p(G)$, $a_j\in G$, and the standard translation operator $\tau$.
	Moreover, Bourgain was able to prove the sharpness of this result, providing bounds
	on the index $J$.

	Browder \cite[Theorem 2]{Browder} also examined when the real-valued homological equation has a solution $g\in L_\infty[0,1]$, for a given function $f\in L_\infty[0,1]$ and a transformation $T$. It was shown that it is necessary and sufficient for $\|\sum_{j=0}^{k}f\circ T^j\|_\infty$ to be uniformly bounded over all $k\ge 1$.
    
	In \cite{AdamsPublished} it was shown that for every real-valued mean-zero $f\in L_\infty[0,1]$, there is an ergodic transformation $T$, such that (\ref{ahe}) admits a solution $g\in L_\infty[0,1]$. In \cite{Adams2} the result was strengthened, to show that for any real-valued mean zero $f\in L_p[0,1]$, there exists a solution solution $g\in L_{p-1}[0,1]$ for some ergodic $T$.

	The following result from \cite[Theorem 0.1]{bbs_full_2019} shows for real-valued mean zero $f\in L_\infty[0,1]$ that we can choose $g$ such that $\|g\|_\infty \leq (1+\varepsilon)\|f\|_\infty$.
   
    \begin{Thm}\label{theorem:previos-result}
Let $f\in L_\infty[0,1]$ be a real-valued mean zero function. For any $\varepsilon>0$ there exists a mod $0$ measure preserving transformation $T$ of $[0,1]$ and a function $g\in L_\infty[0,1]$ with   $\|g\|_\infty \leq (1+\varepsilon)\|f\|_\infty$ so that
       $f = g\circ T - g$.
    \end{Thm}

Throughout, a mod $0$ measure preserving transformation is defined as follows.

\begin{defi}[{\cite{bbs_full_2019}}]
	Given two measure spaces $(\Omega, \mathcal A,\mu)$, and
	$(\Omega^\prime,\mathcal A^\prime,\mu^\prime)$,
a \textit{mod $0$ measure preserving transformation} is a bijection
$T:\Omega\setminus N\to\Omega^\prime\setminus N^\prime$, for null sets
$N\in\mathcal A$, $N^\prime\in\mathcal A^\prime$, such that both $T$ and
$T^{-1}$ are measurable, and $\mu^\prime(T(A))=\mu(A)$, for all
$A\in\mathcal A$, with $A\subseteq \Omega\setminus N$.
\end{defi}

The result of Theorem \ref{theorem:previos-result} was announced earlier in \cite{Kwapien}, however the proof there held only for $f\in C[0,1]$. It provides an upper bound on $\|g\|_\infty$, which is important for certain applications in the theory of symmetric functionals (see e.g. \cite{FK}) and singular traces (see e.g. \cite{SingularTraces}). Unlike \cite{AdamsPublished}, Theorem \ref{theorem:previos-result} states nothing about the ergodicity of $T$.

A question that arises naturally is:{\it\ does the result of Theorem \ref{theorem:previos-result} hold for complex-valued mean zero functions?}
This question may be equivalently restated for mean zero functions taking values in  $\RR^2$ and, generalizing even further, for mean zero $\RR^d$-valued functions  for an arbitrary positive integer $d$. Another question is whether the transformation $T$ can be chosen to be ergodic.
In this paper, we answer these questions affirmatively, by proving the following result:
    \begin{Thm}\label{theorem:main-result}
		Let $f\in L_\infty([0,1];V)$ be a $V$-valued mean zero function, for any fixed
		real normed space $V$. For any $\varepsilon>0$ there exists an ergodic mod $0$ measure preserving invertible transformation $T$ of $[0,1]$ and a function $g\in L_\infty([0,1];V)$ with $\|g\|_\infty\leq (S_V+\varepsilon)\|f\|_{\infty}$ (here $S_V$ is the Steinitz constant corresponding to $V$) such that $f= g\circ T -g$.
    \end{Thm}
This theorem holds for all measure spaces which are mod 0 isomorphic
to the interval $[0, 1]$ with respect to the Lebesgue measure.
We note that if we would fix some basis for $V$ and apply Theorem \ref{theorem:previos-result} to the component functions of $f$, this would only yield that $f_i=g_i\circ T_i-g_i$ for $i=1,\ldots, \dim(V)$. In this case we could have $T_i\neq T_j$ for $i\neq j$, so the result does not follow from the earlier theorem, but really is more general.
Additionally we also show that the resulting transformation is ergodic.

Given previous research into the matter, it is not at all clear how to prove the statement of Theorem \ref{theorem:main-result}, as proof techniques from \cite{bbs_full_2019} and \cite{AdamsPublished} can not just be extended to the case of complex-valued functions or more general $\RR^d$-valued functions.
The proof of \cite{AdamsPublished} for real-valued functions is split in a proof for step-functions and a proof for functions that take infinitely many values. Therefore, it seems that it is not possible to extend their method to $\RR^d$-valued functions.  Furthermore, it is not at all clear whether their techniques would extend
to $\RR^d$-valued functions, even if we impose additional conditions on the function $f$.
Although, as already noted, the proof of \cite{bbs_full_2019} for real-valued functions is also not sufficient,
we will show that the proof can be extended to a certain class of $\RR^d$-valued functions, but this is not at all trivial.
We combine this partial result with a new technique to obtain the main result in full. 

The constant $S_V$ mentioned in the theorem is the Steinitz constant corresponding to the space $V$, which arises from Steinitz's rearrangement lemma \cite[Lemma 2.1.3]{kadets_series_1997}. It is defined to be the smallest value such that for all finite collections of vectors $v_1,\ldots, v_n$ in $V$ with sum $\sum_{i=1}^{n}v_i=0$, there exists a permutation $\pi$ such that $\|\sum_{j=1}^kv_{\pi(i)}\|\leq S_V\max_i\|v_i\|$ for all $k=1,\ldots, n$  \cite{kadets_series_1997}. To show that the Steinitz constant and the rearrangement lemma are closely related to the additive homological equation, we will give an equivalent definition. Let $\Omega_n$ be a finite set of $n$ elements, equipped with  the counting measure. Then $S_V$ can be also defined to be the smallest value such that for $n\geq 1$, and for all mean zero $f\in L_\infty(\Omega_n, V)$ there exists an (ergodic) measure preserving transformation $T$ of $\Omega_n$ and a set of positive measure $X\subseteq \Omega_n$, such that $\|\sum_{j=0}^{k}f\circ T^k\|_{L_\infty(X;V)}\leq S_V\|f\|_\infty$ for all $k=1,2,...$. As a consequence of Theorem \ref{theorem:main-result} we have the following result, which can be seen as a natural continuous analogue to Steinitz's rearrangement lemma.

\begin{Thm}\label{theorem:bound-partial-sums}
	Let $f\in L_\infty([0,1];V)$ be a $V$-valued mean zero function for any real normed space $V$, and $\varepsilon>0$. There exists an ergodic mod $0$ measure preserving transformation $T$ of $[0,1]$ and a set $X\subset [0,1]$ of positive measure such that for $k=1,2,...$ we have
	$\|\sum_{j=0}^{k}f\circ T^j\|_{L_\infty(X;V)}\leq (S_V+\varepsilon)\|f\|_\infty$.
\end{Thm}

We remark that the equivalent formulation from \cite[Theorem 2]{Browder} holds for $f\in L_\infty([0,1];V)$ as well. This is to say that
for $f\in L_\infty([0,1];V)$ and measure preserving $T$ we can find a solution $g\in L_\infty([0,1];V)$ if and only if $\|\sum_{j=0}^{k}f\circ T^j\|_\infty$ is uniformly bounded for $k\ge 1$.

 An immediate corollary of our main result is the following extension of \Kwapien's Theorem  \ref{theorem:previos-result} to the case of  complex-valued mean zero functions.
     \begin{Thm}\label{theorem:complex}
        Let $f\in L_\infty[0,1]$ be a complex-valued mean zero function. For any $\varepsilon>0$ there exists an ergodic mod $0$ measure preserving invertible transformation $T$ and a function $g\in L_\infty[0,1]$ with $\|g\|_\infty\leq (\dfrac{\sqrt{5}}{2}+\varepsilon)\|f\|_{\infty}$ such that $f= g\circ T -g$.
    \end{Thm}

	We now give an overview of our method of proof of Theorem \ref{theorem:main-result}, and outline the structure of the paper.
	The proof of the main theorem follows three key steps.

	In the next section, we will establish the basic facts, definitions, and
	notations used throughout the paper. We recall the definition of the
	Steinitz constant $S_V$, and its fundamental properties, and introduce
	affinely homogeneous and affinely partially homogeneous functions.

	We then can start working on the first key step towards the proof of
	Theorem~\ref{theorem:main-result}.
	The following lemma is fundamental to \Kwapien's proof
	\cite{Kwapien}, and is no less fundamental in our work here.

\begin{lem}\label{Kwapien Lemma}
	If $(a_{i,j})_{n\times m}$ is a matrix with real entries such that $|a_{i,j}|\leq C,\ i=1,\dots,n,\ j=1,\dots,m$ and $\sum_{j=1}^{m}a_{i,j}=0$ for $i=1,\dots,n$, then there are permutations $\sigma_1,\dots,\sigma_n$ of the integers $\{1,\dots,m\}$ such that 
	$$|\sum_{i=1}^{k}a_{i,\sigma_i(j)}|\leq 2C,\ k=1,\dots,n,\ j=1,\dots,m.$$
\end{lem}
We generalize this result in Theorem \ref{ala_kwapen_lem} to the case when the real entries $a_{i,j}$ are replaced with vectors from $V$. 
Our extension of \Kwapien's lemma is the main result in Section 3, Theorem~\ref{Kwapien Lemma}. This theorem is then used in Section \ref{section:continuous-functions} to solve the equation for continuous functions on Cantor sets, see Theorem \ref{t_cantor}.\\

In Section 4, we show that the functions we consider may be decomposed into affinely
partially homogeneous functions.
In Section 5, we prove several ``Shrinking lemmas'', which provide refinements of
Lusin's theorem, and which are necessary in order to prove the main result when
restricted to affinely homogeneous functions.

Indeed, this is the next key step in the proof, and the focus of Sections 6 and 7.
In Section 6, we prove that the main result holds for continuous mean-zero functions
over the Cantor set (Theorem~6.1).
Using the result for continuous functions on the Cantor set, we then solve in Section \ref{section:affinely-homogeneous}, the equation for the subclass of $L_\infty([0,1];V)$ that we call affinely homogeneous functions. These are functions that may be understood to be ``very non-constant'', with respect to linear affine subspaces. For such functions we prove, using tools we develop in Section \ref{section:shrinking}, that we can construct subsets of the domain that are of positive measure, homeomorphic to the Cantor set, and such that the restriction of $f$ to this subset is mean zero and continuous. We can then apply the result for continuous functions to solve the equation for this class of functions. We note that the transformation constructed here is not ergodic.\\

Finally, in Section~8, we complete the proof of Theorems~1.3 and 1.4, building upon
the prior results.
However, in order to prove these main results, we need different tools since the method for affinely homogeneous functions can not be used in general, and also since we want $T$ to be ergodic. Our proof for general functions does however use the results that we developed for affinely homogeneous functions. Indeed, in Lemma \ref{lemma:countable-partition-of-set} we use results from Section \ref{section:decomposing-function} and Section \ref{section:affinely-homogeneous} to construct a partition of the domain, and a measure preserving transformation satisfying certain properties.
In the final part of the proof of the theorem we apply this lemma inductively to obtain transformations $T^{(1)},T^{(2)},...$. Using these transformations we construct an ergodic transformation $T$ and a function $g$ that solve the equation.\\

\subsection{Novelty and necessity of Affinely Homogeneous Function Techniques}

We feel compelled to emphasize that, although the constructions for affinely homogeneous functions and continuity on Cantor sets bear some analogy to \cite{bbs_full_2019}, the proof  for general functions is totally different. Indeed, the proof of \cite{bbs_full_2019} is based on splitting the case of general $f$ into (roughly speaking) two cases, when $f$ is continuous and when $f$ is simple. A quick analysis shows that such splitting is impossible when we deal with $\RR^d$-valued functions. This fact has necessitated a new approach which is most visible in the proof of Theorem \ref{theorem:main-result} given in Section 8, and in the preceding Lemma \ref{lemma:countable-partition-of-set}. The proof of Section 8 moreover has some connections to the construction done in \cite{AdamsPublished}, though the two methods are different. 
Let us briefly discuss why these earlier techniques are not amenable to the general result.

As can be seen in \cite{AdamsPublished} and \cite{bbs_full_2019}, when solving the equation for real-valued functions, problems arise when dealing with step-functions. In \cite{AdamsPublished} this is bypassed by restricting to the case that $f$ takes infinitely many values, and using a different method for step-functions. In \cite{bbs_full_2019} the proof divides the domain into parts on which $f$ is mean zero and either behaves `non-constant' or is a step-function with two steps, and solving the equation separately on these domains. In our current work for $\RR^d$-valued functions the problem with step-functions gets more difficult, as it becomes a problem with affine subspaces. We solve the equation by extending methods from \cite{bbs_full_2019} and using a new technique.\\

It would seem that techniques from \cite{AdamsPublished} for real-valued functions simply cannot be extended to $\RR^d$-valued functions, even when we would impose extra conditions on the function $f$. It is also not clear whether techniques from \cite{bbs_full_2019} can be extended to $\RR^d$-valued functions, however in sections \ref{section:optimization-result}, \ref{section:shrinking}, \ref{section:continuous-functions} and \ref{section:affinely-homogeneous} we showed that this is possible for $\RR^d$-valued functions that are affinely homogeneous. 
Indeed fix $\alpha \in (0,1)$ and consider the mean zero function $f\in L_\infty([0,1];\RR^2)$ given by
$$f=(f_1,f_2),\ f_1=(1-\alpha)\chi_{[0,\alpha]}-\alpha\chi_{(\alpha,1]},\ f_2(t)=t-\dfrac{1}{2}.$$

A solution $g,T$ solving the equation $f =g\circ T - g$, would directly provide us with a solution for the first coordinate function $f_1$.
Now the function $f$ takes infinitely many values but some sort of extension of the method of \cite{AdamsPublished} can not work as it can not deal with the step-function $f_1$.
An extension of the method of \cite{bbs_full_2019} also fails as the Cantor set construction can not be carried out for $f$ when $\alpha$ is irrational.  This is the reason for the need of a new approach.\\

The new approach is carried out in Section \ref{section:proof-main-theorem}
and uses our construction for  affinely homogeneous functions.
There is a connection between our approach in Section \ref{section:proof-main-theorem} and the method in \cite{AdamsPublished}, though the methods are different.
The connection exists between the partition $\{A_{i,j}:i\geq 1, 1\leq j\leq q_i\}$ that we construct in Lemma \ref{lemma:countable-partition-of-set} and between the collection of disjoint sets
$\{I_{i,j}:1\leq j\leq w, 1\leq i\leq h_j\}$ that was developed in \cite{AdamsPublished}, Lemma 12.4. The sets $\{I_{i,j}\}$ were referred to as a  $W-TUB(\varepsilon,M,h,w)$ and come together with a certain transformation $\tau$ mapping $I_{i,j}$ to $I_{i+1,j}$ for $i=1,...,h_j-1$. This is different to the case at hand, in that the collection $\{I_{i,j}\}$ is finite, while $\{A_{i,j}\}$ is countably infinite. Also, to construct the collection $\{I_{i,j}\}$, the function $f$ has to take infinitely many values, while the partition $\{A_{i,j}\}$ can always be constructed. Furthermore, the sets $\{I_{i,j}\}$ do not partition the entire interval like the sets $\{A_{i,j}\}$, though the function $f$ is still mean zero on their union. Some bounds that hold for the $W-TUB$ are
$\left|\sum_{i=0}^{k}f(\tau^i(x))\right| \leq M\|f\|_\infty$ for $x\in I_{1,j}, k<h_j$ and 
$\left|\sum_{i=0}^{h_j-1}f(\tau^i(x))\right|<\varepsilon$ for 
$x\in I_{1,j}$. These conditions are similar to the conditions we show in Lemma \ref{lemma:countable-partition-of-set}. 

\subsection{Failure of Theorems~1.3 and 1.4 for Infinite Dimensional Vector Spaces}

It is far from apparent that the results of Theorems~1.3 and 1.4 should not carry through
to infinite dimensional vector spaces.
While the constant $S_V$ is monotonically increasing with the dimension of $V$, this only forms an
upper bound, and so the statement is not immediately sufficient to disprove any infinite
dimensional extension.
However, a straightforward construction shows us that even in the most simple case of
a separable, infinite-dimensional Hilbert space, the main theorems fail.

In the space $\RR^d$, $d>1$ equipped with Euclidean norm, consider vertices of the regular  $d-1$-dimensional simplex centered at zero: $$x_k=(a_{k,i}),\ k=1,\dots,d,\ a_{k,i}=\delta_{ki}-\dfrac{1}{d}.$$ We have $\|x_k\|^2=\dfrac{d-1}{d^2}+\dfrac{(d-1)^2}{d^2}=\dfrac{d-1}{d}<1,\ \sum_{k=1}^{d}x_k=0$.
We shall assume that $d$ is even and that $y$ is given by the sum of $\dfrac{d}{2}$ such vertices (with possible repetition!).
 Let us estimate the norm $\|y\|$ from below.
 At least $\dfrac{d}{2}$ of components of $y$ are equal to $\dfrac{d}{2}\cdot(-\dfrac{1}{d})=-\dfrac{1}{2}$.
 Hence, $\|y\|^2\geq \dfrac{d}{2}\cdot\dfrac{1}{4}=\dfrac{d}{8}$, that is $$\|y\|\geq \sqrt{\dfrac{d}{8}}.$$For every $n\ge 1$ we  shall set $d_n=2^n$ and choose $$r_n>0, $$ satisfying the condition $$\sum_{n=1}^\infty r_n^2\leq 1,\ 2^{\frac{n-3}{2}}r_n\longrightarrow\infty.$$Further, for every $n> 1$ we denote $x^n_1,\dots,x^n_{d_n}$ elements of $\RR^{d_n}$,  defined as above but multiplied by a constant factor depending on $n$ only so that we have  $$\|x^n_k\|=r_n.$$
     Finally, we define the space $V=\oplus_{n=1}^\infty \RR^{d_n}$ as a Hilbertian sum.
 We now set$$f_n:[0,1]\longrightarrow \RR^{d_n}:\ f_n([\dfrac{i-1}{{d_n}},\dfrac{i}{{d_n}}))=x_i^n,\ i=1,\dots,{d_n},\ f_n(1)=x_{d_n}^n$$ We have $$f_n\in L_\infty([0,1],\RR^{d_n}),\ \|f_n\|_\infty\leq r_n, \int f_n d\lambda=0.$$
Setting $f=\oplus_{n=1}^\infty f_n$., we obtain $$f\in L_\infty([0,1],V),\ \|f\|_\infty\leq 1,\ \int fd\lambda=0.$$ Assume that $T$ is a measure preserving transformation of $[0,1]$,  such that $$\sup_k\|\sum_{i=0}^k f\circ T^i\|_\infty=C<\infty.$$
 Then $$\sup_k\|\sum_{i=0}^k f_n\circ T^i\|_\infty\leq C\ \forall n.$$
 We shall show that this is not so and hence we obtain a contradiction with the assumption about existence of such $T$.
 For almost every $t\in [0,1]$ we have that the element$$\sum_{i=0}^{d_n/2-1} f_n\circ T^i(t)$$ coincides with the sum of $\dfrac{d_n}{2}$ elements from the set $$\{x^n_1,\dots,x^n_{d_n}\},$$that is$$t\in[0,1]\Rightarrow T^i(t)\in[0,1]\Rightarrow f_n\circ T^i(t)\in \{x^n_1,\dots,x^n_{d_n}\}),$$  whose norm cannot be less than $$\sqrt{\dfrac{d_n}{8}}\cdot r_n=2^{\frac{n-3}{2}}r_n.$$ 
 Therefore, $C\geq 2^{\frac{n-3}{2}}r_n\longrightarrow\infty$,  the required contradiction.
 Hence, $\sup_k\|\sum_{i=0}^k f\circ T^i\|_\infty=\infty$.
 Therefore there is no function $g\in L_\infty([0,1],V)$,  satisfying $f=g\circ T-g$.
 Indeed, otherwise $$\sup_k\|\sum_{i=0}^k f\circ T^i\|_\infty={\sup_k}\|g\circ T^{k+1}-g\|_\infty\leq 2\|g   \|_\infty.$$

	\section{Preliminaries}
 \subsection{Three fundamental theorems}   The following version of Lusin's Theorem is stated in  \cite[Theorem 2.2.10]{Bogachev1}.
	
	\begin{Thm}[Lusin's theorem]\label{prelim:lusin}
		Let $D\subseteq [0,1]$ be Borel-measurable and let $f: D \to \RR$ be Borel-measurable. If $\varepsilon>0$, then there is a compact subset $K\subseteq A$ such that $\lambda(A \setminus K) < \varepsilon$ and such that the restriction of $f$ to $K$ is continuous.		
	\end{Thm}
	
	The following fundamental fact is obtained by combining Theorems 9.3.4 and 9.5.1 from \cite{Bogachev2}.
	\begin{Thm}\label{prelim:isomorphism}
		Let $A,B\subseteq [0,1]$ be some subsets of equal positive measure, then there exists a mod $0$ measure preserving transformation $T$ between $A$ and $B$.
	\end{Thm}

We shall also crucially use Lyapunov's theorem \cite[Theorem 2.c.9]{LT1}.

\begin{Thm} \label{t_ljap} 
Let $\{\mu_i\}_{i=1}^d$ be a set of finite (not necessarily positive) non-atomic measures on the measurable space $(\Omega,\Sigma)$. Then a set 
$$\{(\mu_1(X),\dots,\mu_d(X)):\ X\in\Sigma\}$$ 
is convex and compact in $\mathbb{R}^d$. 
\end{Thm}

\subsection{The space $L_\infty(D;V$)}
Throughout, $(V,\|\cdot\|)$ will denote a finite-dimensional normed vector space over $\RR$.
Let $D$ be a Lebesgue measurable subset of $[0,1]$ equipped with Lebesgue measure $\lambda$, and let $f:D\longrightarrow V$ be a measurable mapping.
 A vector $r\in V$ is said to be an \textit{essential value} of the function $f$, if $\lambda(f^{-1}(U))>0$ for an arbitrary neighbourhood $U$ of the vector $r$. The symbol $\sigma(f)$ stands for the set of all essential values of $f$ (the usage of this symbol is justified by the fact that for a function  $f\in L_\infty[0,1]$ the set of all its essential values coincides with the spectrum of the element $f$ in the $C^*$-algebra $L_\infty[0,1]$).
    
By $L_\infty(D;V)$ we denote the linear space of all measurable mappings  $f:D\longrightarrow V$, for which the set $\sigma(f)$ is bounded. As usual, we will identify any two mappings if they are equal almost everywhere (that is the space $L_\infty(D;V)$ consists of classes of measurable mappings equal almost everywhere).
   
We will say that a function $f\in L_\infty(D;V)$ is \textit{simple} if $f=\sum_{i=1}^{\infty}r_i\chi_{X_i}$, where $r_i\in V,\ i=1,2,\dots,\ \{X_i\}_{i=1}^{\infty}$ is a splitting of $D$ into measurable subsets. 
 
Define a norm on $L_\infty(D;V)$, by setting for $f\in L_\infty(D;V)$  $$\|f\|_\infty=\sup\{\|r\|:\ r\in \sigma(f)\}.$$
 
For every $f\in L_\infty(D;V)$ the integral $\int f d\lambda \in V$ is defined in a standard way. If $\int f d\lambda=0$, then the function $f$ is said to be  \textit{mean zero}.
 
 We shall frequently use the notation  
 $$\avint_{X}f d\lambda = \frac{\int_{X}f d\lambda}{\lambda(X)},$$ i.e. $\avint_{X}f d\lambda$ is the mean value of  $f$ on the set $X$.
 Furthermore, we shall sometimes use Euclidian norm, in which case we denote $(\cdot,\cdot)$ for the Euclidean inner products.
 
 \subsection{Affinely homogeneous functions}
 
For an arbitrary set $X\subset V$, the symbol ${\rm Aff}(X)$  denotes the affine  subspace in  $V$ generated by $X$, that is 
$$
{\rm Aff}(X) = \{\sum_{i=1}^{k}a_ix_i: x_i\in X, a_i\in \RR,\ \sum_{i=1}^{k}a_i=1\}.
$$ 
Recall that any affine subspace in $V$ may be viewed as the set $\{x+V_0\}$ where $x$ is some point in $V$ and $V_0$ is a linear subspace in $V$. The dimension of such affine subspace is defined to be the dimension of the subspace $V_0$. In particular, every point in $V$ is an affine subspace of dimension $0$.
 
 We will say that a function $f\in L_\infty(D;V)$ is \textit{affinely homogeneous} if for every proper affine subspace $W\subsetneq {\rm Aff}(\sigma(f))$ we have $\lambda(f^{-1}(W)) =0$. This is to say that every subset of positive measure has a full-dimensional image.
 
 We note moreover that a real-valued function is affinely homogeneous if and only if it is either constant, or satisfies $\lambda(f^{-1}(\{y\}))= 0$ for all $y\in \RR$.
 
It is easy to see that any affinely homogeneous simple function is constant. Indeed, if a simple function has two distinct essential values, say $a$ and $b$, then $\lambda(f^{-1}(a))>0$ and $\lambda(f^{-1}(b))>0$ and since $a\subsetneq {\rm Aff}(\sigma(f))$ and $b\subsetneq {\rm Aff}(\sigma(f))$ are proper affine subspaces of ${\rm Aff}(\sigma(f))$ we arrive at a contradiction.

More generally, for any affinely homogeneous function $f$ it holds that we have
$\Aff(\sigma(f|_{A}))=\Aff(\sigma(f))$ for every subset $A\subseteq D$ of positive measure.

We will say that a function $f\in L_\infty(D;V)$ is said to be \textit{affinely partially homogeneous}, if $D$ can be split into at most $d+1$ measurable subsets, where $d=\dim(V)$, such that (the reduction of) $f$ is affinely homogeneous on each of this subsets. For example, a function $f = (1-a)\chi_{[0,a)} -a\chi_{[a,1]}\in L_\infty([0,1];\RR)$ is affinely partially homogeneous for any $a\in (0,1)$.

\subsection{The Steinitz constant}
For every given finite-dimensional normed space $V$ (over $\RR$) there exists a smallest number $S_V$ (called the Steinitz constant), such that for every collection $r_1,\cdots,r_n\in V,\ \sum_{i=1}^n r_i=0$, the following inequalities hold
    $$\|\sum_{i=1}^{k}r_{\pi(i)}\|\leq S_V\max\{\|r_i\|:\ i=1,\dots,n\},\ k=1,\dots,n,$$ for some permutation $\pi$ of the set  $\{1,\dots,n\}$ \cite{Steinitz_1913}. This constant, generally speaking, does not only depends on the dimension of $V$, but also on the norm.
    
    It is shown in \cite{grinberg_value_1980} that $S_V\leq \dim(V)$ (see detailed proof in \cite[Lemma 2.1.3]{kadets_series_1997}). Trivially, we have $S_{\RR}=1$.
 	In \cite[Remark 3]{banaszczyk1987} it is stated that \lq\lq Applying the same method as in the proof of Lemma 2, one can show
    that the Steinitz constant of an $n$-dimensional space is not greater than $n-1+\dfrac{1}{n}$\rq\rq, however, this assertion is not supplied with a proof.
    If we equip $\RR^d$ with Euclidian norm then it holds that $S_{\RR^d}\geq\frac{\sqrt{d+3}}{2}$ \cite{grinberg_value_1980}, $S_{\RR^2} = S_{\CC}=\frac{\sqrt{5}}{2}$ \cite[Theorem 2]{banaszczyk1987},\cite{banaszczyk1990}. For other estimates for $S_{\RR^d}$ for Euclidean norm when $d>2$ see \cite[Remark 8, Added in proof]{banaszczyk1990_2}.
    
     Let us explain the appearance of Steinitz constant, by proving that the main result holds for mean zero functions $f\in L_\infty(\Omega_n;V)$ for any finite measure space $\Omega_n$ equipped with counting measure. Indeed, as $\sum_{i=1}^{n}f(i)=0$ it follows from the definition of the Steinitz constant that there exists a permutation $\pi$ of $\{1,...,n\}$ s.t.
    $$\|\sum_{i=1}^m f(\pi(i))\| \leq S_V\|f\|_\infty,\ m=1,\dots,n.$$
    We can then define a cyclic permutation $\sigma$ of $\Omega_n$ as $\sigma(\pi(j)) = \pi(j+1)$ for $j=1,...,n-1$ and $\sigma(\pi(n)) = \pi(1)$. We then put $g(\pi(k)) = \sum_{i=1}^{k-1} f(\pi(i))$ for $k=2,...,n$ and $g(\pi(1)) = 0$. Then $g\circ \sigma - g = f$ and $\|g\|_\infty\leq S_V\|f\|_{\infty}$, which shows the result.	
	It can be seen that this proof method can also be applied to simple functions $f\in L_\infty([0,1];V)$ of the form
  	$$f=\sum_{k=1}^n r_k\chi_{I_k};\ I_k=[\dfrac{k-1}{n},\dfrac{k}{n}),\ r_k\in V,\ k=1,\dots,d;\ \sum_{k=1}^n r_k=0.$$  
  	as they can be identified with a mean zero function $\widetilde{f}$ in $L_\infty(\Omega_n;V)$ given by $\widetilde{f}(k) = r_k$. Solving the equation for this function and consequently defining the transformation $T$ to map $T(I_{k}) = I_{\sigma(k)}$, and defining the simple function $g$ by setting  $g|_{I_{\pi(k)}}=\sum_{i=1}^{k-1} r_{\pi(i)}$ for $k=2,...,n$ and $g|_{I_{\pi(1)}} = 0$ gives us the result.

  \section{A Multidimensional Version of \Kwapien's Lemma}
\label{section:optimization-result}
The main result of this section is Theorem \ref{ala_kwapen_lem}. Its proof is based on the following known results. The notation $\Conv(X)$ stands for the convex hull of a set $X\subset V$.
  
  \begin{Thm}
      \label{BG_T3}  \cite[Theorem 3]{baranyCombinatorialQuestionsFinitedimensional1981}
      Let $V$ be a $d$-dimensional real normed space, with the unit ball $B^d $, let $C_i\subset B^d$ and let  $0\in \Conv(C_i),\ i=1,2,\dots\ .$ 
      Under these assumptions, there exist elements $c_i\in C_i,\ i=1,2,\dots$, such that
      $$\left\|\sum_{i=1}^pc_i\right\|\leq 2d,\ p=1,2,\dots\ .$$
  \end{Thm}
  
  \begin{Thm}
      \label{GS_T1}  \cite[Theorem 1]{grinberg_value_1980}, \cite[Lemma 2.1.3]{kadets_series_1997}. Let $V$ be a $d$-dimensional  real normed space, $\|x_i\|\leq 1,\ i=1,\dots,n$ and $x_1+\dots+x_n=x$.
      Then there exists such enumeration $\pi$, that for all natural indices $k\leq n$ we have
      $$\left\|\sum_{i=1}^k x_{\pi(i)}-\frac{k-d}{n}x\right\|\leq d.$$ 
  \end{Thm}
  
  Now we are well equipped to prove the following lemma.
  \begin{lem}
      \label{l1}
      Let $V$ be a $d$-dimensional real normed space. Let $\{a_{i,j}\}_{i,j=1}^{n,m}$ be vectors in $V$ with $\|a_{i,j}\|\leq 1,\ i=1,\dots,n,\ j=1,\dots,m$,
      $$\sum_{j=1}^m a_{i,j}=0,\ i=1,\dots,n.$$
      Let $p\leq m$ be a natural number.
      Then the set $\{1,\dots,m\}$ contains such subsets $I_1,\dots,I_n$, that
      $$|I_1|=\dots=|I_n|=p$$
      and
      $$\|\sum_{i=1}^k \sum_{j\in I_i} a_{i,j}\|\leq 4d^2,\ k=1,\dots,n.$$
  \end{lem}
  \begin{proof}
For every fixed $i=1,\dots,n$, we have $\sum_{j=1}^m a_{ij}=0$ by the assumption. By Theorem \ref{GS_T1} , replacing the collection of vectors $x_1,\dots,x_n$ with the collection $a_{i,1},\dots,a_{i,m}$, we infer the existence of a permutation $\pi$ of the set $\{1,2,\dots,m\}$  such that $\left\|\sum_{j=1}^k a_{i,\pi(j)}\right\|\leq d,\ k=1,\dots,m$. Relabelling   vectors $a_{ij},\ j=1,\dots,m$, we may assume without loss of generality that for every  $i=1,\dots,n$ we have
      $$\|\sum_{j=1}^l a_{i,j}\|\leq d,\ l=1,\dots,m.$$
      
Let $m_1$ be the least common multiple of the numbers $m$ and 
$p$, and let $m_2=m_1/p$.  Let us consider the mapping $\alpha$ from $\{1,\dots,m_1\}$ onto $\{1,\dots,m\}$, which maps a number 
$j$ to the  remainder of the division on $m$, provided that $j$ is not a scalar multiple of $m$, and into $m$ otherwise.   
      
We now replace the matrix $\{a_{ij}\}_{i,j=1}^{n,m}$ with the 
matrix $\{a'_{i,j}\}_{i,j=1}^{n,m_1}$, where  $a'_{i,j}=a_{i,\alpha(j)}$. In other words, any column of the matrix $\{a_{i,j}\}_{j=1}^{m}$ is repeated $m_1/m$ times.
      
Observe that the matrix $$\{a'_{i,j}\}_{i,j=1}^{n,m_1}$$ continues to satisfy the same assumptions as the original matrix $\{a_{i,j}\}_{i,j=1}^{n,m}$.
      
We now set $b_{i,j}=\sum_{r=(j-1)p+1}^{jp}a'_{i,r},\ j=1,\dots,m_2,\ i=1,\dots,n$. Let us show that
      $$\|b_{i,j}\|\leq 2d$$ for all $i,j$.
      
 If the sequence $\alpha((j-1)p+1),\alpha((j-1)p+2),\dots,\alpha(jp)$ increases, then we have
      $$\|b_{i,j}\|=\|\sum_{r=1}^{\alpha(jp)}a_{ir}-\sum_{r=1}^{\alpha((j-1)p)}a_{ir}\|\leq 2d.$$
      Otherwise, $m\in \{\alpha((j-1)p+1),\alpha((j-1)p+2),\dots,\alpha(jp)\}$. That is $\{\alpha((j-1)p+1),\alpha((j-1)p+2),\dots,\alpha(jp)\}$ consists of two sets:
      $\{m-k+1,m-k+2,\dots,m\}$ and $\{1,2,\dots,p-k\}$ whose intersection is empty.
      
Then, we obtain $$\|b_{i,j}\|=\|\sum_{r=m-k+1}^{m}a_{i,r}+\sum_{r=1}^{p-k}a_{i,r}\|=\|-\sum_{r=1}^{m-k}a_{i,r}+\sum_{r=1}^{p-k}a_{i,r}\|\leq 2d.$$
      
Also, for all $i\in \{1,\dots,n\}$, $\sum_{j=1}^{m_2} b_{i,j}=\sum_{r=1}^{m_1}a'_{i,r}=\frac{m_1}{m}\sum_{j=1}^m a_{i,j}=0$, so $0\in \Conv\{b_{i,j}: j\in \{1,\dots,m_2\}\} $ for all $i\in \{1,\dots,n\}$.
      
Theorem    \ref{BG_T3} yields that there exists such indices $j_i$, that  
      $$\|\sum_{i=1}^k b_{i,j_i}\|\leq 4d^2$$
for all $k=1,\dots,n$.  Since $b_{i,j_i}=\sum_{r=(j_i-1)p+1}^{j_ip}a'_{i,r}=\sum_{r=(j_i-1)p+1}^{j_ip}a_{i,\alpha(r)}=\sum_{j\in I_i} a_{i,j}$, where $I_i=\alpha(\{(j_i-1)p+1,(j_i-1)p+2,\dots,j_ip\})$, the estimate above yields the assertion and completes the proof.
  \end{proof}
  
  We now use the result of Lemma \ref{l1} to obtain a similar result for non-mean-zero vectors.
  \begin{lem}
      \label{lem:splitting_lem}
       Let $V$ be a $d$-dimensional real normed space. Let $(v_{i,j})_{1\leq i\leq n, 1\leq j\leq m}$ be vectors in $V$ with $\|v_{i,j}\|\leq 1$ and $p\in \{1,\dots,m\}$, let $x_k=\sum_{i=1}^k\sum_{j=1}^m\frac{1}{m}v_{i,j}$ we can find sets $I_i\subseteq \{1,\dots,m\}$ for $i=1,\ldots, n$  such that
      \begin{align*}
      |I_k|                                                              & =p               & \forall k\in\{1,\dots,n\},  \\
      \left|\left|\sum_{i=1}^k\sum_{j\in I_i}v_{i,j} - px_k\right|\right| & \leq 8d^2       & \forall k\in \{1,\dots,n\}.
      \end{align*}
  \end{lem}
  \begin{proof}
      Define $(v_{i,j}')_{1\leq i\leq n, 1\leq j\leq m}$ by setting $v_{i,j}'=\frac{1}{2}v_{i,j}-\frac{1}{2m}\sum_{k=1}^mv_{i,k}$. Note that we now have $\sum_{j=1}^mv_{i,j}'=0$ for all $i\in \{1,\dots,n\}$ and that $\|v_{i,j}'\|\leq \frac{1}{2}\|v_{i,j}\|+\frac{1}{2m}\sum_{k=1}^m\|v_{i,k}\|\leq \frac{1}{2} + \frac{m}{2m} = 1$.
      Using Lemma \ref{l1} we can find sets $I_i$ for $i=1,\ldots, n$ such that
      \begin{align*}
      |I_k|                                                              & =p               & \forall k\in \{1,\dots,n\},  \\
      \left|\left|\sum_{i=1}^k\sum_{j\in I_i}v_{i,j}'\right|\right| & \leq 4d^2             & \forall k\in \{1,\dots,n\}.
      \end{align*}
Appealing to the preceding inequality, we obtain
      \begin{align*}
      &\left|\left|\sum_{i=1}^k\sum_{j\in I_i}v_{i,j}-px_k\right|\right| =
      \left|\left|\sum_{i=1}^k\sum_{j\in I_i}v_{i,j}-\sum_{i=1}^k\frac{|I_k|}{m}\sum_{j=1}^mv_{i,j}\right|\right|\\
      &= 
      \left|\left|\sum_{i=1}^k\sum_{j\in I_i}\left(v_{i,j}-\frac{1}{m}\sum_{t=1}^mv_{i,t}\right)\right|\right|\\
      &=2\left|\left|\sum_{i=1}^k\sum_{j\in I_i}v_{i,j}'\right|\right|   \leq 2\cdot 4d^2  =8d^2
      \end{align*}
      for all $k\in \{1,\dots,n\}$.
  \end{proof}
  
  Finally, we are prepared to generalize Lemma \ref{Kwapien Lemma} for entries in $V$.
  
  \begin{Thm}
      \label{ala_kwapen_lem}
      Let $V$ be a $d$-dimensional real normed space.
Let $(v_{i,j})_{1\leq i\leq n, 1\leq j\leq m}$ be vectors in $V$ with $\|v_{i,j}\|\leq 1$ and  $x_k= \frac{1}{m} \sum_{i=1}^k \sum_{j=1}^m v_{i,j}$ for all $k\in \{1,\dots,n\}$, there exist permutations $(\pi_i)_{1\leq i \leq n}$ of $\{1,\dots,m\}$ with $\|\sum_{i=1}^kv_{i,\pi_{i}(j)}-x_k\| \leq \frac{8d^2}{\log 1.5}$ for all $k$ and all $j$.
      \label{thm:result}
  \end{Thm}
  \begin{proof}
 We will show that we can construct suitable permutations by partitioning the input vectors into two almost equally sized sets using Lemma \ref{lem:splitting_lem} and then recursively constructing suitable permutations for both parts of the partition. We then combine these two permutations into one permutation and we show that this permutation satisfies the required properties.
 
When $m=1$ the assertion follows trivially, as in this case $\sum_{i=1}^kv_{i,1}-x_k=0$ for all $k$.
 
Let $m=2$. Then $(v_{i,1}-\dfrac{v_{i,1}+v_{i,2}}{2})+(v_{i,2}-\dfrac{v_{i,1}+v_{i,2}}{2})=0$ for all $i=1,2\dots,n$. It follows from Theorem \ref{BG_T3} that there exists indices $j_i\in\{1,2\}$, such that $\|\sum_{i=1}^k(v_{i,j_i}-\dfrac{v_{i,1}+v_{i,2}}{2})\|\leq 4d$. Let us set $\pi_i(1)=j_i,\pi_i(2)=3-j_i$. Then $\|\sum_{i=1}^kv_{i,\pi_i(1)}-x_k\|\leq 4d$ for all $k$ and $v_{i,\pi_i(2)}-\dfrac{v_{i,1}+v_{i,2}}{2}=-(v_{i,\pi_i(1)}-\dfrac{v_{i,1}+v_{i,2}}{2})$, therefore $\|\sum_{i=1}^kv_{i,\pi_i(2)}-x_k\|\leq 4d$ for all $k$. Since we know that $\int_{0}^{\log_{1.5} (2)-1}\left(\frac{2}{3}\right)^{x}dx=\dfrac{1}{4(\log3-\log2)}>\dfrac{1}{2}$, we obtain $4d<8d^2\int_{0}^{\log_{1.5} (2)-1}\left(\frac{2}{3}\right)^{x}dx$, and so $$\|\sum_{i=1}^kv_{i,\pi_{i}(j)}-x_k\| \leq 8d^2\int_{0}^{\log_{1.5} (2)-1}\left(\frac{2}{3}\right)^{x}dx.$$

Next, we will prove via induction on $m$ that for a given set of input vectors $(v_{i,j})_{1\leq i\leq n, 1\leq j\leq m}$ with $\|v_{i,j}\|\leq 1$, there exist permutations $(\pi_i)_{1\leq i \leq n}$ of $\{1,\dots,m\}$ with
      \begin{align*}
      \|\sum_{i=1}^kv_{i,\pi_{i}(j)}-x_k\| \leq 8d^2\int_{0}^{\log_{1.5} (m)-1}\left(\frac{2}{3}\right)^{x}dx\text{.
      }
      \end{align*}
The assertion of Theorem \ref{ala_kwapen_lem} would then follow by replacing the integral above with the integral from 0 to $\infty$, which has value $1/\log 1.5$. 

For $m=2$ the inequality was established above. For $m> 2$, assume that the statement holds up to $m-1$ inclusive. By Lemma \ref{lem:splitting_lem}, there exist sets $(I_i)_{1\leq i\leq n}$ in $\{1,\dots,m\}$ such that for all $i\in \{1,\dots,n\}$,  $|I_{i}|=p:=\lceil \frac{m}{2} \rceil$  that satisfy the assertion from this lemma. For all $k\in \{1,\dots,n\}$, let $\delta_k := \sum_{i=1}^k \sum_{j\in I_i}v_{i,j}$.
      Lemma \ref{lem:splitting_lem} now implies that
      \begin{align}
      \|\delta_k - px_k\| \leq  8d^2\text{.}\label{delta_k_leq_px_k}
      \end{align}
      Let $\delta_k' := \sum_{i=1}^k \sum_{j\in \{1,\dots,m\}\setminus I_i}v_{i,j}$. We claim that $\delta_k+\delta_k'=mx_k$. Indeed, 
$$\delta_k+\delta_k'=\sum_{i=1}^k \sum_{j\in I_i}v_{i,j}+\sum_{i=1}^k \sum_{j\in \{1,\dots,m\}\setminus I_i}v_{i,j}=\sum_{i=1}^k \sum_{j=1}^mv_{i,j}=mx_k,$$ and so
      \begin{align}
      \|\delta_k' - (m-p)x_k\|=\|mx_k- \delta_k-(m-p)x_k\|=\|\delta_k -px_k\| \leq  8d^2\text{.
      }\label{delta_prime_k_leq_mpx_k}
      \end{align}
      
      For each $i\in\{1,\dots,n\}$, let $\pi'_{i}$ be a permutation of $\{1,\dots,m\}$ that maps the set $\{1,\ldots, p\}$ to $I_i$.
      Now define $(v^{(1)}_{i,j})_{1\leq i\leq n,1\leq j\leq p}$ by setting $v^{(1)}_{i,j} = v_{i,\pi'_i(j)}$.
      Using our induction hypothesis we can find permutations $\pi^{(1)}_i$ of $\{1,\ldots,p\}$ such that for all $k\in \{1,\dots,n\}$ and all $j\in \{1,\ldots, p\}$,
      \begin{align*}
      \left|\left|\sum_{i=1}^kv^ {(1)}_{i,\pi_i^ {(1)}(j)}-\frac{1}{p}\delta_k\right|\right|
      \leq 8d^2\int_{0}^{\log_{1.5} (p)-1}\left(\frac{2}{3}\right)^{x}dx\text{.
      }
      \end{align*}
      Similarly we define $(v^{(2)}_{i,j})_{1\leq i\leq n,1\leq j\leq m-p}$ by setting $v^{(2)}_{i,j} = v_{i,\pi'_i(j+p)}$ and using the induction hypothesis we can find permutations $\pi^{(2)}_i$ of $\{1,\dots,m-p\}$ such that for all $k\in \{1,\dots,n\}$ and all $j\in\{1,\dots,m-p\}$,
      \begin{align*}
      \left|\left|\sum_{i=1}^kv^ {(2)}_{i,\pi_i^ {(2)}(j)}-\frac{1}{m-p}\delta_k'\right|\right|\leq 8d^2\int_{0}^{\log_{1.5} (m-p)-1}\left(\frac{2}{3}\right)^{x}dx\text{.
      }
      \end{align*}
      Now we define
      \begin{align*}
      \pi_i(j)=
      \begin{cases}
      \pi'_i(\pi^{(1)}_i(j))     & j\leq p \\
      \pi'_i(\pi^{(2)}_i(j-p)+p) & j> p
      \end{cases}
      \text{.
      }
      \end{align*}
      Define
      \begin{align*}
      p_j=
      \begin{cases}
      p   & j\leq p \\
      m-p & j> p
      \end{cases}
      \text{.
      }
      \end{align*}
      and
      \begin{align*}
      \Delta_i(j)=
      \begin{cases}
      \frac{1}{p_j}\delta_i  & j\leq p \\
      \frac{1}{p_j}\delta_i' & j> p
      \end{cases}
      \text{.}
      \end{align*}
Considering two cases, when $j\leq p$  and when $j> p$ and applying  (\ref{delta_k_leq_px_k}) and  (\ref{delta_prime_k_leq_mpx_k}), respectively,
we obtain $\|\Delta_i(j)-x_i \|\leq 8d^2/p_j$ for all $i\in \{1,\dots,n\}$ and all $j\in \{1,\dots,m\}$.
      
      For $j\in \{1,\ldots, p\}$ we have for all $k\in\{1,\dots,n\}$ that
      \begin{align*}
      \left|\left|\sum_{i=1}^kv_{i,\pi_i(j)}-\Delta_k(j)\right|\right| & =\left|\left|\sum_{i=1}^kv^ {(1)}_{i,\pi_i^ {(1)}(j)}-\frac{1}{p}\delta_k\right|\right| \\
      & \leq 8d^2\int_{0}^{\log_{1.5} (p)-1}\left(\frac{2}{3}\right)^{x}dx
      \\&= 8d^2\int_{0}^{\log_{1.5} (p_j)-1}\left(\frac{2}{3}\right)^{x}dx\text{.
      }
      \end{align*}
      
      Similarly for $j\in \{p+1,\ldots, m\}$ we have that
      \begin{align*}
      \left|\left|\sum_{i=1}^kv_{i,\pi_i(j)}-\Delta_k(j)\right|\right| & =\left|\left|\sum_{i=1}^kv^ {(2)}_{i,\pi_i^ {(2)}(j-p)}-\frac{1}{p}\delta_k' \right|\right| \\
      & \leq 8d^2\int_{0}^{\log_{1.5} (m-p)-1}\left(\frac{2}{3}\right)^{x}dx
      \\&= 8d^2\int_{0}^{\log_{1.5} (p_j)-1}\left(\frac{2}{3}\right)^{x}dx\text{.
      }
      \end{align*}
      Combining these yields that for all $k\in \{1,\dots,n\}$ and all $j\in \{1,\dots,m\}$
      \begin{align*}
      & \left|\left|\sum_{i=1}^kv_{i,\pi_i(j)}- x_k\right|\right|
      \leq \left|\left|\sum_{i=1}^kv_{i,\pi_i(j)}- \Delta_k(j)\right|\right|+\left|\left|\Delta_k(j)-x_k\right|\right| \\
      & \leq 8d^2 \int_0^{\log_{1.5}( p_j)-1}\left(\frac{2}{3}\right)^{x}dx+8d^2\frac{1}{p_j}
      \\&= 8d^2\int_0^{\log_{1.5} (p_j)-1}\left(\frac{2}{3}\right)^{x}dx+8d^2\cdot \left(\frac{2}{3}\right)^{\log_{1.5}p_j}\\&
      \leq 8d^2 \int_0^{\log_{1.5}( p_j)}\left(\frac{2}{3}\right)^{x}dx                                              \\
      & \leq 8d^2 \int_0^{\log_{1.5}( m)-1}\left(\frac{2}{3}\right)^{x}dx
      \end{align*}
      Note that the last inequality follows from the fact that $m/p_j\geq3/2$, so $\log_{1.5}(m)-\log_{1.5}(p_j)=\log_{1.5}(m/p_j)\geq 1$.
      
  \end{proof}
    
   	\section{Decomposition for Bounded Functions into Affinely Partially Homogeneous Functions}\label{section:decomposing-function}
   
  We will need several well-known results due to Caratheodory. The first lemma below can be found in \cite[Theorem 8.11]{Simon}.
  \begin{lem}
      \label{l_kara} Let $B\subset \RR^d,\ d<\infty$. Then any element $\xi\in \Conv(B)$ can be decomposed as a convex span of at most $d+1$ elements from $B$.
\end{lem}
For the following two results we refer to
  \cite[Corollary IV.1.13]{GTM} and \cite[Corollary IV.3.11]{GTM} respectively. 
  \begin{Thm}
      \label{t_kara} A convex span of the closure of a bounded subset in $\mathbb{R}^d,\ d<\infty$ coincides with the closure of the convex span of this subset.
  \end{Thm}
  
  \begin{Thm}
      \label{intersection_halfspaces} The closed convex hull of a set $A\subseteq \RR^d$ equals the intersection of all closed half-spaces containing it.
  \end{Thm}
   
   We begin with the following general (and probably well-known)  result.
   
   \begin{prop}
       \label{p1}
       Let $\{\xi_i\}_{i\in I}\subset\mathbb{R}^d,\ d<\infty,\ \{\alpha_i\}_{i\in I}\subset \mathbb{R}_+\setminus\{0\}$, $card(I)\leq \aleph_0$,
       $0<\|\xi_i\|\leq 1$ (here, $\|\cdot\|$ is Euclidian norm), $\sum_i \alpha_i \leq 1,\ \sum_i \alpha_i\xi_i=0$.
       Then there exist indices $i_1,\dots,i_m\in I,\ 1 \leq m\leq d+1,$ and scalars $0<\beta_k\leq\alpha_{i_k},\ k=1,\dots,m$ such that $\sum_{k=1}^m\beta_k\xi_{i_k}=0$.
   \end{prop}
   \begin{proof}
   Without loss of generality, we may assume that 
$$\dim(\Span\{\xi_i:\ i\in I\})=d,\quad{\rm and}\quad\sum_i \alpha_i=1.$$
       
       Let $B=\{\xi_i\}_{i\in I},\ C=\Conv(B)$. By Theorem \ref{t_kara}, $\overline{C}=\Conv(\overline{B})$.
       Therefore, $0\in \overline{C}=\Conv(\overline{B})$.
       
       For any set $X\subset \mathbb{R}^d$ its \emph{support function} $h_X$ is defined by 
       $$h_X(\eta)=\sup\{(\eta,\xi):\ \xi\in X\}.$$
       
       Let $\mathbb{S}^{d-1}=\{\eta\in\mathbb{R}^d:\ \|\eta\|=1\}$, that is $\mathbb{S}^{d-1}$ is a sphere  in $\mathbb{R}^d$  centered at zero with radius 1.
       
       Since for $\eta \in \mathbb{S}^{d-1}$ the closed half-space $H_\eta:=\{\xi: (\eta,\xi)\leq h_{X}(\eta)\}$ contains $X$, and since every closed half-space $H$ that contains $X$ is contained in 
       $H_\eta$ for some $\eta\in \mathbb{S}^{d-1}$, 
       it follows from \ref{intersection_halfspaces} that $$\overline{\Conv(X)}=\bigcap_{\eta\in\mathbb{S}^{d-1}}\{\xi:\ (\eta,\xi)\leq h_X (\eta)\}.$$
       
       We show that $\overline{C}$ contains a ball with radius $r_0>0$ centered at $0$. Indeed, the function $h_{\overline{C}}$ is continuous on the unit sphere $\mathbb{S}^{d-1}$. Since $\mathbb{S}^{d-1}$ is compact, it follows that there exists a point $\eta_0\in\mathbb{S}^{d-1}$, at which $h_{\overline{C}}$ reaches the minimum. Assume that $h_{\overline{C}}(\eta_0)\leq 0$. Then, $(\eta_0,\xi)\leq 0$ for any $\xi\in B$. From the equality $\sum_i \alpha_i\xi_i=0$ it follows that $\sum_i \alpha_i(\eta,\xi_i)=0$, and, thus, $(\eta_0,\xi)=0$ for any $\xi\in B$. This contradicts the fact that $\dim(\Span(B))=d$. Therefore, $r_0:=h_{\overline{C}}(\eta_0)>0$.
       Since $\overline{C}=\bigcap_{\eta\in\mathbb{S}^{d-1}}\{\xi:\ (\eta,\xi)\leq h_{\overline{C}}(\eta)\}$, it follows that $\overline{C}$ contains a ball with a radius $r_0$ centered at zero $0$.
       
       Since $\overline{B}$ is compact, then there exists $n\in\mathbb{N}$ such that $B_n:=\{\xi_i\}_{i=1}^n$  is a $r_0/3$-net in $\overline{B}$.
       
       Let $\eta\in \mathbb{S}^{d-1}$. There exists a vector 
       $\xi\in \overline{B}$ such that 
$$
(\eta,\xi)=h_{\overline{B}}(\eta)=h_{\overline{C}}(\eta)\geq r_0.
$$ 
Let now $\xi'\in B_n$ be such that 
$\|\xi-\xi'\|<r_0/3$. We have 
$$|(\eta,\xi')-(\eta,\xi)|\leq \|\xi-\xi'\|<r_0/3,\ (\eta,\xi)\geq r_0.$$ 
It follows that
$|(\eta,\xi')|=(\eta,\xi')$ and therefore 
$$(\eta,\xi')\geq (\eta,\xi)-|(\eta,\xi')-(\eta,\xi)|>r_0-r_0/3>r_0/2.$$
 
Thus, $h_{B_n}(\eta)\geq r_0/2$. Therefore,  $\Conv(B_n)$ contains the ball with radius $r_0/2$ centered at $0$. In particular, the point $0$ is a convex combination of the vectors $\{\xi_i\}_{i=1}^n$.
       
       By Lemma \ref{l_kara}, there exist $\xi_{i_1},\dots,\xi_{i_m}\in B_n,\ m\leq d+1$ such that $0=\sum_{k=1}^m\beta'_k\xi_{i_k},\ \beta'_k\in \mathbb{R}_+,\ \sum_{k=1}^m\beta'_k=1$.
       Finally, setting 
       $$\beta_k=\beta'_k\gamma,\ \gamma:=\min\{\alpha_{i_k}:\ k=1,\dots,m\},$$
we complete the proof.   
   \end{proof}

  	In the following lemma we partition the domain of a function $f$, so that on each partition subset $P$, the function $f|_P$ is affinely homogeneous.
  
  \begin{lem} \label{lemma:partition}
      Let $f\in L_\infty(D;\RR^d)$. Then there exists a finite or countable partition $\{P_i\}_{i\in I}$ of $D$ of measurable subsets of non-zero measure, so that every $f|_{P_i}$ is affinely homogeneous.
  \end{lem}
  \begin{proof}
      Consider the collection $\mathcal{A}$ of all families $\{D_i\}_{i\in I}$ of disjoint measurable subsets of $D$, of positive measure, for which $f|_{D_i}$ is affinely homogeneous. We order this collection by inclusion. Then, by Zorn's lemma we can find a maximal element $\{P_i\}_{i\in I} \in \mathcal{A}$. We show that this is a partition.
      Let $X=D\setminus \bigcup_{i\in I} P_i$. Suppose that $\lambda(X)>0$. Since the set $\{0,1,\dots,d\}$ is finite, it follows that there exists a minimal $k$ such that there exists an affine linear subspace $W\subseteq {\rm Aff}(\sigma(f|_X))$, $\dim(W)=k$ and $\lambda(f^{-1}(W)\cap X)>0$. Setting $P_0=f^{-1}(W)\cap X$, we obtain that $f|_{P_0}$ is affinely homogeneous. However, this contradicts with the maximality of $\{P_i\}_{i\in I}$. We conclude that $\lambda(X) = 0$, hence $\{P_i\}_{i\in I}$ is a partition.
  \end{proof}

   \begin{Thm}
       \label{t-reduction-affinely-partially-homogeneous}
       Let $f\in L_\infty([0,1];\mathbb{R}^d),\ \int f d\lambda=0$. Then there exists a finite or countable partition of the interval $[0,1]$ into measurable subsets $X_1,X_2,\dots$ such that
       
       (i). $\int_{X_n}f d\lambda=0,\ n=1,2,\dots$;
       
       (ii). For any $n=1,2,\dots$, the function $f|_{X_n}$ is affinely partially homogeneous.
   \end{Thm}
   \begin{proof}
       Let $\{X_i\}_{i\in I}$ be a maximal collection of disjoint subsets of $[0,1]$ of positive measure, satisfying (i) and (ii). Such collection exists by application of Zorn's lemma. Let $D = [0,1]\setminus \bigcup_{i\in I}X_i$. We show that $\lambda(D) = 0$. Namely, suppose that $\lambda(D)> 0$. 
       Then let $\{D_i\}_{i\in I}$ be a decomposition established in Lemma \ref{lemma:partition}. Then
       $$
       0=\sum_{i\in I}\lambda(D_i)\avint_{D_i}fd\lambda.
       $$
       Now, by Proposition \ref{p1} we have that for some $1\leq m\leq d+1$ we can find $i_1,...i_m \in I$ and $0<\lambda_j\leq \lambda(D_{i_j})$ with  $$0=\sum_{j=1}^m\lambda_j\avint_{D_{i_j}}fd\lambda$$
       Set $\lambda_j' = \frac{\lambda_j}{\lambda(D_{i_j})}$ so that 
       $0 = \sum_{j=1}^m\lambda_j'\int_{D_{i_j}}fd\lambda$. Now,  we can define non-atomic measures $\{\mu_i\}_{i=1}^d$ as $\mu_i(E) = \int_{E}f_id\lambda$ for every Lebesgue measurable set $E\subset [0,1]$    
       and apply Theorem \ref{t_ljap}, so that we obtain measurable subsets $D_{i_j}'\subset D_{i_j}$ of non-zero measure with
       $\int_{D_{i_j}'}fd\lambda = \lambda_j'\int_{D_{i_j}}fd\lambda$. Now set $X = \bigcup_{j=1}^m D_{i_j}'$ so that $\int_{X}fd\lambda = 0$. Furthermore, by the properties of $D_{i_j}$ and the fact that $m\leq d+1$ we have that $f|_{X}$ is affinely partially homogeneous. However, then the collection $\{X_i\}_{i\in I}\cup \{X\}$ would contradict the maximality of $\{X_i\}_{i\in I}$. We thus conclude that $\lambda(D) = 0$, and hence $\{X_i\}_{i\in I}$ partitions $[0,1]$.
   \end{proof}    
    
\section{Shrinking Lemmas}\label{section:shrinking}

\subsection{Obtaining positive constants} We will need the following lemma for proving Lemma \ref{lemma:integral-on-subset}. We shall prove this lemma for general mean zero integrable functions.

In the following lemma, $\|\cdot\|$ will be used for the Euclidian norm on $\RR^k$ and $(\cdot,\cdot)$ for the inner product. Likewise $\|\cdot\|_1$ on $L_1(D;\RR^k)$ is defined using the Euclidian norm $\|\cdot\|$.
For $v\in \RR^d$ and $f\in L_1(D;\RR^d)$ we moreover denote 
$(v,f)$ for the function $t\mapsto (v,f(t))$, i.e. the composition of $f$ with the inner product. 
We will furthermore simply write $|f|$ to denote the function 
$t\mapsto \|f(t)\|$.
\begin{lem}\label{lemma:obtaining-positive-consants}
	Let $D\subseteq [0,1]$ be of positive measure and let 
	$f\in L_1(D;\RR^d)$ be satisfying 
	$\int_{D}fd\lambda = 0$. We can find 
	$\alpha,\beta_{\min},\beta_{\max},\tau>0$ s.t. for all non-zero 
	$v\in \Span(\sigma(f))$ we have 
	$\lambda(\{\frac{(v,f)}{\|v\|\cdot |f|} > \alpha\}\cap \{\beta_{\min} <|f|<\beta_{\max}\})>\tau$
\end{lem}
\begin{proof}
	We will prove the lemma by induction to the dimension $d$. 
	The statement holds trivially for $d=0$ since then there are no non-zero vectors. 
	Let us fix $d\geq 1$ and assume that we have already proven it for $0\leq j\leq  d-1$. Let $f\in L_1([0,1];\RR^d)$ be mean zero. Suppose first that $\Span(\sigma(f)) \not=\RR^d$. 
	By choosing an orthonormal basis for $\Span(\sigma(f))$, we can consider $f$ as a mean zero function in $L_1(D;\RR^k)$
	where
	$$k = \dim \Span(\sigma(f)).$$
	By the induction hypothesis, we then obtain values $\alpha,\beta_{\min},\beta_{\max},\tau>0$ so that for every $v\in \Span(\sigma(f))$ the stated property holds. The statement then also holds with the same constants when we consider $f$ again as function in $L_1([0,1];\RR^d)$ and this proves the statement for this case.
	
	We can thus assume that $\Span(\sigma(f)) = \RR^d$.
	Now, with this assumption we have that for any non-zero $v\in \RR^d$ it holds that $(v,f) \not= 0$.
	
	We will now work toward defining the scalars $\alpha,\beta_{\min},\beta_{\max},\tau$.
	
	We set 
	$$
	D_0 := D\setminus \{f=0\}.
	$$
	Then we have $\lambda(D_0)>0$ and $\int_{D_0}fd\lambda = 0$ and for non-zero $v\in \RR^d$ we have $$(v,f|_{D_0}) \not= 0.$$		
	
	Let $\mathbb{S}^{d-1}$ denote the $d-1$-dimensional unit sphere. 
	For $v\in \mathbb{S}^{d-1}$ define a bounded function $h_v: D_0\to \RR$ as $h_v = \frac{(v,f)}{|f|}$. By Cauchy-Schwartz we have for $v,w\in \mathbb{S}^{d-1}$ that $|h_v^+ - h_w^+| = \frac{1}{|f|}| (v,f)^+ - (w,f)^+| \leq \frac{1}{|f|} |(v-w,f)|\leq \|v-w\|$. 
	Hence the map $v\mapsto \|h_v^+\|_{L_\infty(D_0)}$ is continuous.
	Define 
	$$
	\alpha = \frac{1}{2}\min_{v\in \mathbb{S}^{d-1}}\|h_v^+\|_{L_\infty(D_0)}
	$$ which is possible due to compactness of $\mathbb{S}^{d-1}$. We claim that the inequality 
	$$\|h_v^+\|_{L_\infty(D_0)}>0$$ holds for every
	$v\in \mathbb{S}^{d-1}$. Indeed, suppose for a moment that this is not so. Then $(v, f|_{D_0}) \leq 0$ almost everywhere. 
	Now, since $\int_{D_0}fd\lambda = 0$ we have that 
	$$\int_{D_0}(v,f) d\lambda = (v,\int_{D_0}fd\lambda) =0$$ and this implies that $(v,f|_{D_0}) = 0$ a.e., which is a contradiction with the inequality $(v,f|_{D_0}) \not= 0$ above. Thus, we must have $\|h_v^+\|_{L_\infty(D_0)}>0$. This immediately implies that $\alpha>0$. \\
	
	Now for  $v\in \mathbb{S}^{d-1}$ define 
	\begin{align*}
	\tau_{v} = \lambda(\{h_{v}>\alpha\}), &\quad& \tau = \frac{1}{2}\inf_{v\in \mathbb{S}^{d-1}} \tau_{v}
	\end{align*}
	We note for $v\in \mathbb{S}^{d-1}$ that $\tau_v>0$  as $\|h_{v}^+\|_{\L_\infty(D_0)}>\alpha$.
	We show that also $\tau>0$.
	
	Suppose that $(v_n)$ is a sequence in $\mathbb{S}^{d-1}$ such that $\tau_{v_n}\to 0$. By compactness of $\mathbb{S}^{d-1}$ we can assume that $v_n$ converges to some $v\in \mathbb{S}^{d-1}$. Choose $\varepsilon>0$.
	Since 
	\begin{align*}
	\{h_v>\alpha+\frac{1}{j}\} &\text{ increases to } \{h_v>\alpha\} \text{ as } j\to \infty
	\end{align*} and since $D_0$ has finite measure, we can choose $\delta>0$ small s.t
	\begin{align*}
	\lambda(\{h_v>\alpha\}\setminus\{h_v>\alpha+\delta\}) &< \varepsilon
	\end{align*}
	Now, since $h_{v_n}\to h_v$ in $L_\infty(D_0)$ by Cauchy-Schwarz, we can find $N$ s.t. for $n\geq N$ we have 
	\begin{align*}
	\|h_v - h_{v_n}\|_\infty&<\delta
	\end{align*}
	Now for $n\geq N$ we have
	\begin{align*}
	\tau_v - \tau_{v_n} &\leq \lambda((\{h_v>\alpha\} 
	\setminus(\{h_{v_n}>\alpha\}))\\
	&\leq\lambda(\{h_v>\alpha\}\setminus\{h_v>\alpha+\delta\}) + \lambda(\{h_v>\alpha+\delta\}\setminus\{h_{v_n}>\alpha\})\\
	&<\varepsilon + 0 =\varepsilon
	\end{align*}
	As $\tau_{v_n}\to 0$, it follows that $\tau_v\leq \varepsilon$. Since $\varepsilon$ was arbitrary this means that $\tau_v = 0$ which is a contradiction. Hence, we infer that such a sequence $(v_n)$ does not exist. Therefore, we have $\tau>0$.\\
	
	We now choose $\beta_{\min}>0$ small and $\beta_{\max}>0$ large such that $$\lambda(D_0\setminus \{\beta_{\min}<|f|<\beta_{\max}\})<\frac{1}{2}\tau$$
	Note that this is possible as $\lambda(D_0\cap \{f=0\})=0$ and as $D_0$ has finite measure.\\
	
	Now for non-zero $v \in \RR^d = \Span(\sigma(f))$ we find
	\begin{align*}
	\lambda(\{\frac{(v,f)}{\|v\| \cdot|f|} >\alpha\} &\cap \{\beta_{\min}<|f|<\beta_{\max}\}) =\\
	&=\lambda(\{h_{\frac{v}{\|v\|}} > \alpha\} \cap \{\beta_{\min}<|f|<\beta_{\max}\}) \\
	&\geq \lambda(\{h_{\frac{v}{\|v\|}}> \alpha\}) - \lambda(D_0\setminus \{\beta_{\min}<|f|<\beta_{\max}\})\\
	&\geq \tau_{\frac{v}{\|v\|}} - \frac{1}{2}\tau\\
	&\geq 2\tau - \frac{1}{2}\tau \\
	&> \tau
	\end{align*} 
	which proves the statement.
\end{proof}

\subsection{Changing the mean zero condition for subsets}
We will now deal with a domain $D\subseteq [0,1]$ with positive measure, and a mean zero function $f\in L_1(D;V)$. The following lemma allows us to obtain a slightly smaller, compact subset $E\subseteq D$ for which $f|_{E}$ is continuous and mean zero. The lemma can also be used to obtain, for vectors $u$ in a certain neighborhood of $0$, a subset $E\subseteq D$ so that $\int_{E}(f+u)d\lambda = 0$. This last result will be needed for proving Lemma \ref{lemma:rational-splitting}.\\
\begin{lem}\label{lemma:integral-on-subset}
	Let $V$ be a finite-dimensional normed real space. Let $D\subseteq [0,1]$ be of positive measure and let $f\in L_1(D;V)$ be mean zero. Then, for $\varepsilon>0$ there is a scalar $\delta>0$ such that for every  measurable subset $D'\subseteq D$ with $\lambda(D\setminus D')<\delta$ and every vector $u\in \Span(\sigma(f))$ with $\|u\|\leq \delta$ we can find a compact subset $E\subseteq D'\cap (\inf D,\sup D)$ with $\lambda(D\setminus E)<\varepsilon$, such that $\int_{E}(f+u)d\lambda = 0$, and moreover such that $f|_{E}$ is continuous.
	\begin{proof}
		As norms on finite-dimensional vector spaces are equivalent, we can w.l.o.g. assume that $V$ is the vector space $\RR^d$ with Euclidian norm.
		
		Let $D,f$ and $\varepsilon$ be as stated. We apply Lemma \ref{lemma:obtaining-positive-consants} to $D$ and $f$ and select positive constants $\alpha,\beta_{\min},\beta_{\max}$ and $\tau$ from that lemma. In particular, we have $\beta_{\min}\in (0,\beta_{\max})$ and $\alpha\in (0,1)$, hence we can set $\gamma:= \sqrt{1-\frac{\alpha^2 \beta_{\min}}{4\beta_{\max}}}\in (0,1)$ and $\rho := \frac{\alpha}{2\beta_{\max}(1-\gamma)}>0$.
		
		We introduce a continuous, non-decreasing function $I_{\sup}:[0,\lambda(D)]\to [0,\|f\|_1]$ as
		\begin{align*}
		I_{\sup}(s) = \sup_{U\subseteq D, \lambda(U)=s}\int_{U}|f|d\lambda
		\end{align*}
		Further, we set 
		\begin{align}\label{eq:definition:delta-prime}
		\delta' &:= \frac{1}{2}\min \{\frac{\tau}{4(1 + \rho)},\frac{\tau\beta_{\max}}{2},\frac{\tau}{2\rho}, \frac{\varepsilon}{2(1+\rho)},\beta_{\max},\frac{\alpha \beta_{\min}}{8}\}>0\\
		\label{eq:definition:delta} \delta &:=\frac{1}{2}\min\{\delta',I_{\sup}^{-1}(\delta')\}>0
		\end{align}

		Now, choose a measurable subset $D'\subseteq D$ with $\lambda(D\setminus D')<\delta$ and choose $u\in \Span(\sigma(f))$ with $\|u\|\leq \delta$. For $i=1,...d$, let $f_i$ denote the coordinate functions of $f$ w.r.t. the standard basis.  By Theorem \ref{prelim:lusin}, there exist compact subsets $K_i\subseteq D'\cap (\inf D,\sup D)$ with $\lambda(D'\setminus K_i)<\frac{\delta}{d}$ and such that $f_i|_{K_i}$ is continuous, $i=1,...d$. Now set $K := \bigcap_{i=1}^d K_i$ so that $f|_{K}$ is continuous and moreover bounded as $K$ is compact.
		We obtain that 
		$$\lambda(D\setminus K)= \lambda(D\setminus D') + \lambda(D'\setminus K)  \leq \delta + d\cdot\frac{\delta}{d} = 2\delta.
		$$
		We now set $E_0:=K$ and 
		$$
		v_0 := \int_{E_0}(f +u)d\lambda.
		$$ 
		Since $f$ is mean zero on $D$ we have that 
		\begin{align*}\|v_0\| &\leq \|\int_{E_0}ud\lambda\| + \|\int_{D\setminus E_0}fd\lambda \|\\ & \leq \lambda(E_0)\|u\| + I_{\sup}(\lambda(D\setminus E_0))\\ &\leq \delta + I_{\sup}(2\delta)\\
		&\leq 2\delta'
		\end{align*} 
		
		We will now inductively define compact sets $(E_j)_{j\geq 1}$ and mutually disjoint sets $(A_j)_{j\geq 1}$ in $D$ and define the vectors $v_j := \int_{E_j}(f+u)d\lambda$ such that for $j\geq 1$ the following holds: 
		\begin{enumerate}
			\item $E_j = E_{j-1}\setminus A_{j}$
			\item $A_j \subseteq E_{j-1} \cap
			\{\beta_{\min}<|f|<\beta_{\max}  \}\cap \{\frac{ (v_{j-1},f)}{\|v_{j-1}\|\cdot |f|} > \alpha\}$
			\item $\lambda(A_j)= \frac{\alpha \|v_{j-1}\|}{2\beta_{\max}} $
			\item $\|v_j\| \leq \gamma^j\|v_0\|$
			\item $\lambda(E_0\setminus E_j) \leq \rho \|v_0\|$
		\end{enumerate}
		Assume $E_l$ and $v_l$ are defined for $l< j$ and $A_{l}$ are defined for $0<l< j$ so that they satisfy the above. We construct $E_j, A_j$ and $v_j$ and show that the stated properties hold. Suppose first that $v_{j-1} = 0$. We then define $A_j :=\emptyset$ and $E_j := E_{j-1}$ so that $v_j =v_{j-1}=0$ and $\lambda(E_0\setminus E_j) = \lambda(E_0\setminus E_{j-1})\leq \rho {\|v_0\|}$. Then all conditions are satisfied and we are done.	
		We can thus assume that $v_{j-1}\not=0$.	
		Then, since $v_{j-1}\in \Span(\sigma(f))$ is non-zero, we have that 
		\begin{align*}
		\lambda(E_{j-1}&\cap \{\beta_{\min}<|f|<\beta_{\max}\}\cap \{\frac{ (v_{j-1},f)}{\|v_{j-1}\|\cdot |f|} > \alpha\})\\
		&>\lambda(\{\beta_{\min}<|f|<\beta_{\max}\}\cap \{\frac{ (v_{j-1},f)}{\|v_{j-1}\|\cdot |f|} > \alpha\}) - \lambda(D\setminus E_{j-1})\\
		&\geq  \tau - \lambda(D\setminus E_0) - \lambda(E_0\setminus E_{j-1})\\
		&\geq \tau - 2\delta -  \rho\|v_0\|\\
		&\geq \tau - (2+2\rho)\delta'\geq \frac{1}{2}\tau
		\end{align*}
		
		Now for $r\in [0,1]$, we set $$B_r = (0,r)\cap E_{j-1} \cap
		\{\beta_{\min}<|f|<\beta_{\max}\}\cap \{\frac{ (v_{j-1},f)}{\|v_{j-1}\|\cdot |f|} > \alpha\}.
		$$ 
		Since $f|_{K}$ is continuous, it follows that $E_{j-1} \cap \{\beta_{\min}<|f|<\beta_{\max}  \}$ is open in $E_{j-1}$. Furthermore, on this set we have that $f$ stays away from zero, in particular $\frac{ (v_{j-1},f)}{\|v_{j-1}\|\cdot |f|}$ is continuous when considered as a function on this set. Thus 
		$$E_{j-1} \cap
		\{\beta_{\min}<|f|<\beta_{\max}  \}\cap \{\frac{ (v_{j-1},f)}{\|v_{j-1}\|\cdot |f|} > \alpha\}$$ is open in $E_{j-1} \cap \{\beta_{\min}<|f|<\beta_{\max}  \}$, and therefore it is also open in $E_{j-1}$. So, the sets $B_r$ are open in $E_{j-1}$ for every $r\in [0,1]$.
		
		By induction step (4), the fact that $\alpha,\gamma\in (0,1)$, the bound on $\|v_0\|$ and the definition of $\delta'$ we have that
		$$
		\frac{\alpha\|v_{j-1}\|}{2\beta_{\max}}\leq \frac{\alpha\gamma^{j-1}\|v_{0}\|}{2\beta_{\max}}\leq  \frac{\|v_0\|}{2\beta_{\max}}\leq \frac{2\delta'}{2\beta_{\max}} <\frac{1}{2}\tau
		$$
		
		Now, since $\lambda(B_0) =0$ and $\lambda(B_1)\geq \frac{1}{2}\tau$, we can find $r_0\in [0,1)$ such that 
		$$
		\lambda(B_{r_0}) = \frac{\alpha\|v_{j-1}\|}{2\beta_{\max}}
		$$ Now, we set $A_j := B_{r_0}$ and $E_{j} := E_{j-1}\setminus A_{j}$ so that $E_j$ is compact and so that (1), (2) and (3) are satisfied.
		
		Now set $v_j := \int_{E_j}(f+u)d\lambda = v_{j-1} - \int_{A_{j}}(f+u)d\lambda$. Now write $w := \int_{A_{j}}(f+u)d\lambda$. We have that 
		\begin{align}\label{eq:bound-on-w}
		\|w\| \leq \int_{A_{j}}(|f|+\delta) d\lambda \leq \lambda(A_{j})(\beta_{\max} + \delta) \leq 2\beta_{\max}\lambda(A_j)= \alpha\|v_{j-1}\|
		\end{align} 
		
		Also, by definition of $A_j$ (step (2) in the induction) we have
		\begin{align}\label{eq:lower-bound-integral}
		\int_{A_j}(v_{j-1},f)d\lambda &> \alpha\|v_{j-1}\|\int_{A_j} |f|d\lambda\\
		\label{eq-lower-bound-with-beta-min}\int_{A_j}|f|d\lambda &>\beta_{\min}\lambda(A_j)
		\end{align}
		
		We now have the following.
		\begin{align*}
		\|v_j\|^2 &= \|v_{j-1}-w\|^2 = \|v_{j-1}\|^2 + \|w\|^2 -2 (v_{j-1},w)\\
		&=\|v_{j-1}\|^2 + \|w\|^2 - 2\int_{A_{j}} (v_{j-1},f) d\lambda -2\int_{A_{j}} (v_{j-1},u) d\lambda\\
		&\stackrel{\eqref{eq:lower-bound-integral}}{\leq} \|v_{j-1}\|^2 + \|w\|^2 - 2\alpha  \|v_{j-1}\| \int_{A_{j}}|f|d\lambda+2\delta\|v_{j-1}\|\lambda(A_j)\\
		&= \|v_{j-1}\|^2 + \|w\|^2 - 2\alpha \|v_{j-1}\|\int_{A_{j}} (|f|+\delta) d\lambda  +4\delta\|v_{j-1}\|\lambda(A_j)\\
		&\leq \|v_{j-1}\|^2 + 	(\|w\| - 2\alpha \|v_{j-1}\|)\int_{A_{j}}(|f|+\delta) d\lambda +4\delta\|v_{j-1}\|\lambda(A_j)\\
		&\stackrel{\eqref{eq:bound-on-w}}{\leq}\|v_{j-1}\|^2 -  \alpha\|v_{j-1}\|\int_{A_{j}}(|f|+\delta) d\lambda +4\delta\|v_{j-1}\|\lambda(A_j)\\
		&\leq \|v_{j-1}\|^2 -  \alpha\|v_{j-1}\|\int_{A_{j}}|f|d\lambda +4\delta\|v_{j-1}\|\lambda(A_j)\\
		&\stackrel{\eqref{eq-lower-bound-with-beta-min}}{\leq} \|v_{j-1}\|^2 - \alpha\|v_{j-1}\|\cdot \beta_{\min}\lambda(A_{j}) 
		+4\delta\|v_{j-1}\|\lambda(A_j)\\
		&= \|v_{j-1}\|^2 (1 - (\alpha \beta_{\min}-4\delta)\cdot \frac{\lambda(A_{j})}{\|v_{j-1}\|})\\
		&\stackrel{\eqref{eq:definition:delta-prime},\eqref{eq:definition:delta}}{\leq} \|v_{j-1}\|^2 (1 - \frac{\alpha \beta_{\min}}{2}\cdot \frac{\lambda(A_{j})}{\|v_{j-1}\|})
		= \|v_{j-1}\|^2 (1-\frac{\alpha \beta_{\min}}{2}\cdot \frac{\alpha}{ 2\beta_{\max}}) \\&= \gamma^2\|v_{j-1}\|^2
		\end{align*}
		Hence, we obtain that $\|v_j\|\leq \gamma \|v_{j-1}\| \leq \gamma^{j}\|v_0\|$, which shows (4). 
		
		Lastly, we have 
		\begin{align*}
		\lambda(E_0\setminus E_j) &= \sum_{n=1}^{j} \lambda(A_n) = \frac{\alpha }{2\beta_{\max}}\sum_{n=1}^j \|v_{n-1}\|
		\leq \frac{\alpha}{2\beta_{\max}}\sum_{n=1}^j \gamma^{n-1}\|v_0\|\\
		&\leq \frac{\alpha \|v_0\|}{2\beta_{\max}}\sum_{n=0}^\infty \gamma^n =\frac{\alpha\|v_0\|}{2\beta_{\max}(1-\gamma)}=\rho \|v_0\|
		\end{align*}
		and the inductive construction is completed.
		Setting $E = \bigcap_{j=0}^\infty E_j$, we obtain a compact subset of $D'\cap (\inf D,\sup D)$ such that $\int_{E}(f+u) d\lambda = \lim\limits_{j\to \infty} v_j = 0$. 	Furthermore, $\lambda(D\setminus E) = \lambda(D\setminus E_0) + \sup_{j\geq 1} \lambda(E_0\setminus E_j)\leq 2\delta + \rho \|v_0\|\leq (2+2\rho)\delta'<\varepsilon$.
		Moreover,  $E\subseteq K$ which implies that $f|_{E}$ is continuous.	The proof is completed.
	\end{proof}
\end{lem}

	\subsection{
    {Arbitrary shrinking and rational splitting}}
	In the following lemma, for a set $K$ and a mean zero function $f\in L_\infty(K;V)$ we find a compact subset $E\subseteq K$ of specified measure such that $\int_{E}fd\lambda = 0$.

	\begin{lem}[Arbitrary shrinking]\label{lemma:shrinking}
	Let $K\subseteq [0,1]$ be compact and of positive measure, and let $f\in L_\infty(K;V)$ be mean zero. Then, for $r\in (0,\lambda(K))$ there is a compact set $E\subseteq K\cap (\inf K,\sup K)$ with $\lambda(K\setminus E)=r$
	and such that $\int_{E}fd\lambda =0$.
	\end{lem}
\begin{proof}
    Consider the collection $\mathcal{A}$ of compact subsets $E$ of $K\cap (\inf K,\sup K)$ with $\int_{E}fd\lambda = 0$ and $\lambda(K\setminus E)\leq r$. Note that this collection is non-empty by Lemma \ref{lemma:integral-on-subset}. 
We will consider this collection under the equivalence relation of sets having a symmetric difference of zero measure and will furthermore order the sets by reverse inclusion. 
Now, for a chain $\{E_i\}_{i\in I}$ in $\mathcal{A}$ (note that $I$ will either be finite or countable), 
    the set $E' := \bigcap_{i\in I} E_i$ is a compact subset of $K\cap (\inf K,\sup K)$ with  $\lambda(K\setminus E') = \sup_{i\in I} \lambda(K\setminus E_i)\leq r$, and $\int_{E'}fd\lambda = 0$. Thus $E'\in \mathcal{A}$ is an upper bound for the chain. Therefore, by Zorn's lemma, there exists a maximal element $E$ in $\mathcal{A}$. Suppose that $\lambda(K\setminus E)<r$ then, setting $\varepsilon = r - \lambda(K\setminus E)>0$ and applying Lemma \ref{lemma:integral-on-subset}, we obtain a compact subset $\widetilde{E}\subseteq E$ with $\lambda(E\setminus \widetilde{E})<\varepsilon$ and such that $\int_{\widetilde{E}}fd\lambda = 0$. We thus have
    $\lambda(K\setminus \widetilde{E}) < \lambda(K\setminus E) + \varepsilon = r$. However, this contradicts the maximality of $E$. We thus conclude that $\lambda(K\setminus E) =r$.
\end{proof}

We now prove the following lemma that finds a subset like Lemma \ref{lemma:shrinking} such that certain ratios are moreover dyadic rationals.
	\begin{lem}[Rational splitting]
	\label{lemma:rational-splitting}
Let $V$ be a finite-dimensional vector real space. Let $N\in\NN$ and let 
$K = K_1\cup ...\cup K_N\subseteq [0,1]$ 
be of positive measure,  where the sets 
$K_i\subseteq [0,1]$ are such that 
$\lambda(K_i\cap K_j) = 0$ whenever 
$i\not=j$. Let $f\in L_\infty(K;V)$ be mean zero 
and such that for $i\geq 2$ there exists some subset $B_i\subseteq K_i$ of positive measure so that we have	$\lambda(f|_{B_i}^{-1}(W)) = 0$ 
for every proper affine subspace 
$W\subsetneq {\rm Aff}(\sigma(f))$. Then, for any $R\in(0,\lambda(K))$, there exists a set	$E  = E_1\cup \dots \cup E_N$, where each set $E_i\subseteq K_i \cap (\inf K_i,\sup K_i)$ is compact and such that $\lambda(E)=R$, $\int_{E}fd\lambda =0$ and such that $\frac{\lambda(E_i)}{\lambda(E)}\in \QQ_2$ for all $1\leq i\leq N$. Here,  $\QQ_2$ is the set of all dyadic rationals.
	\begin{proof}
 
		We will prove this with induction on $N$. For $N=1$ we have $K=K_1$ and we can simply apply Lemma \ref{lemma:shrinking}. We trivially have $\frac{\lambda(E_1)}{\lambda(E)}=1\in \QQ_2$ and this proves the assertion for $N=1$. Thus, let $N\geq 2$ and assume that the assertion holds for $N-1$. We show that it also holds for $N$. Let $K = K_1\cup ...\cup K_N$ as stated. Let $f\in L_\infty(K;V)$ be mean zero and such that for $i\geq 2$ the set $B_i\subseteq K_i$ exists as stated. Furthermore, let $r>0$ be such that $r<\lambda(K)-R$ and $r<\min\{\lambda(K_i):1\leq i\leq N\}\setminus\{0\}$.
		We will assume that $\lambda(K_i)>0$ for every $i=1,..,N$ since otherwise we can set $E_i = \emptyset$ and apply the induction hypothesis to $K\setminus K_i$ which then yields the result. For convenience, we set $\widetilde{K}:=K_1\cup ...\cup K_{N-1}$.	We denote $v = \avint_{K_N} fd\lambda$ and  set 
$$h_1 = f|_{\widetilde{K}} -\avint_{\widetilde{K}}fd\lambda = f|_{\widetilde{K}} + \frac{\lambda(K_N)}{\lambda(\widetilde{K})}v,\ h_2 = f|_{K_N} -\avint_{K_N}fd\lambda = f|_{K_N} - v
$$
so that $h_1\in L_\infty(\widetilde{K};V)$ and $h_2\in L_\infty(K_N;V)$ are mean zero. Now, let us observe that for every $A\subset V,\ x\in V$, we have ${\rm Aff}(A+x)={\rm Aff}(A)+x$. Therefore,
$${\rm Aff}(\sigma(h_2))={\rm Aff}(\sigma(f|_{K_N} -\avint_{K_N}fd\lambda))={\rm Aff}(\sigma(f|_{K_N})) -\avint_{K_N}fd\lambda.
$$ 
Thus, $W:={\rm Aff}(\sigma(h_2))+\avint_{K_N}fd\lambda={\rm Aff}(\sigma(f|_{K_N}))\subset {\rm Aff}(\sigma(f))$, in particular,	$W$ is an affine subspace of ${\rm Aff}(\sigma(f))$.

Since we have
		$f|_{B_N} = h_2|_{B_N} + \avint_{K_N}fd\lambda$, it follows that
$\lambda(f|_{B_N}^{-1}(W)) =\lambda(B_N)> 0$. 
Therefore, by the assumption on $B_N$ we must have the equality 
$W = {\rm Aff}(\sigma(f))$ (as $W$ cannot be its proper subspace).   

  Now since $f,h_2$ are mean zero, we also have 
  $\Span(\sigma(f)) = \Aff(\sigma(f))$ and 
        $\Span(\sigma(h_2)) = {\rm Aff}(\sigma(h_2))$. 
        Hence $\Span(\sigma(h_2)) =  \Span(\sigma(f))$.
		
		 Now, by Lemma \ref{lemma:integral-on-subset} we can find $\delta>0$ such that for $u\in \Span(\sigma(h_2))$ with $\|u\|\leq \delta$ we can find compact $\widetilde{E_N}\subseteq K_N$ with $\lambda(K_N\setminus \widetilde{E_N})< \frac{1}{4}\lambda(K_N) r$ and
		$\int_{E_2}(h_2+u)d\lambda = 0$.		
		
		 Now choose $0<\delta'< \min\{1,\frac{r}{4},\frac{\delta}{\|v\|+1}\}$ such that $(\frac{\lambda(\widetilde{K})}{\lambda(K_N)}(1 - \delta')+1)^{-1}\in\QQ_2$. Set 
		 $$u = \delta'v$$ so that $u\in \Span(\sigma(f)) = \Span(\sigma(h_2))$ and $\|u\|\leq \delta$. Then, due to selection of $\delta$, there exists a compact set $\widetilde{E_N}\subseteq K_N$ with $\lambda(K_N\setminus \widetilde{E_N}) <\frac{1}{4}\lambda(K_N)r$ and $\int_{\widetilde{E_N}}(h_2 + u)d\lambda = 0$.
		Now, set 
        \begin{align*}
		\widetilde{r} &:= \lambda(\widetilde{K}) - \frac{\lambda(\widetilde{K})}{\lambda(K_N)}\lambda(\widetilde{E_N})(1-\delta') \\
		&= \lambda(\widetilde{K}) - \frac{\lambda(\widetilde{K})}{\lambda(K_N)}\left(\lambda(K_N) -\lambda(K_N\setminus \widetilde{E_N})\right)(1-\delta')\\
		&=\delta'\lambda(\widetilde{K}) + \frac{\lambda(\widetilde{K})}{\lambda(K_N)}\lambda(K_N\setminus \widetilde{E_N})(1 - \delta')
		\end{align*}
		So $0< \widetilde{r}\leq \delta' + \frac{1}{\lambda(K_N)}\lambda(K_N\setminus \widetilde{E_N})
		< \frac{r}{4} + \frac{r}{4} =\frac{r}{2}$, thus, in particular, $\widetilde{r}<\lambda(K_i)$ for $i=1,..,N-1$.
		We now apply the induction hypothesis, to obtain a set
		$\widetilde{E} = \widetilde{E}_1 \cup .... \cup \widetilde{E}_{N-1}\subseteq \widetilde{K}$ where $\widetilde{E}_{i}\subseteq K_i$ is compact, so that
		$\lambda(\widetilde{K}\setminus\widetilde{E}) = \widetilde{r}$, so that $\int_{\widetilde{E}}h_1d\lambda = 0$ and so that $\frac{\lambda(\widetilde{E}_i)}{\lambda(\widetilde{E})}\in \QQ_2$ for $i=1,...,N-1$.
		
		Now
		\begin{align*}
		\int_{\widetilde{E}\cup \widetilde{E_N}}f d\lambda &= \int_{\widetilde{E}}h_1 - \frac{\lambda(K_N)}{\lambda(\widetilde{K})}vd\lambda + \int_{\widetilde{E_N}}(h_2 + v)d\lambda\\
		&=-\frac{\lambda(K_N)}{\lambda(\widetilde{K})}\lambda(\widetilde{E})v + \lambda(\widetilde{E_N})v +\int_{\widetilde{E}}h_1 d\lambda + \int_{\widetilde{E_N}}h_2d\lambda\\
		&= -\frac{\lambda(K_N)}{\lambda(\widetilde{K})}
		\lambda(\widetilde{E})v + \lambda(\widetilde{E_N})v
		-\lambda(\widetilde{E_N})u\\
		&=\left(-\frac{\lambda(K_N)}{\lambda(\widetilde{K})}\lambda(\widetilde{E}) + \lambda(\widetilde{E_N})
		-\lambda(\widetilde{E_N})\delta'\right)v \\
		&=\left(-\frac{\lambda(K_N)}{\lambda(\widetilde{K})}(\lambda(\widetilde{K}) - \widetilde{r}) + \lambda(\widetilde{E_N})(1 - \delta')\right)v = 0
		\end{align*} 		
		Now, for the ratio we have 
		\begin{align*}
		\frac{\lambda(\widetilde{E_N})}{\lambda(\widetilde{E}\cup \widetilde{E_N})}
		&=	\frac{\lambda(\widetilde{E_N})}{\lambda(\widetilde{K}) - \widetilde{r} + \lambda(\widetilde{E_N})}\\
		&=	\frac{\lambda(\widetilde{E_N})}{ \frac{\lambda(\widetilde{K})}{\lambda(K_N)}\lambda(\widetilde{E_N})
			(1-\delta') + \lambda(\widetilde{E_N})}\\
		&=	\frac{1}{ \frac{\lambda(\widetilde{K})}{\lambda(K_N)}(1-\delta') + 1}\in \QQ_2
		\end{align*}
		Now, also for $i=1,...,N-1$ we have $\frac{\lambda(\widetilde{E}_i)}{\lambda(\widetilde{E}\cup \widetilde{E_N})} = \frac{\lambda(\widetilde{E}_i)}{\lambda(\widetilde{E})}\cdot\frac{\lambda(\widetilde{E})}{\lambda(\widetilde{E}\cup \widetilde{E_N})}\in \QQ_2$. Last, we have
		$\lambda(K\setminus (\widetilde{E}\cup \widetilde{E}_N)) = \lambda(\widetilde{K}\setminus \widetilde{E}) + \lambda(K_N\setminus \widetilde{E_N})
		\leq \widetilde{r} + \frac{1}{4}r < r$.

		We set $E' = \widetilde{E}\cup \widetilde{E}_N = \bigcup_{i=1}^N\widetilde{E}_i$ so that $\lambda(E')>R$. Indeed, the scalar $r$ was chosen to satisfy $0<r<\lambda(K)-R$, and we obtained that
		$\lambda(K)-\lambda(E')= \lambda(K\setminus E') \leq r<\lambda(K) -R$.	All we need to do now is to shrink the sets $\widetilde{E_i}$ for $i=1,...,N$ by a fixed ratio so that the measure of their union is exactly $R$. The rationality condition will then be preserved. We will do this construction now.
		
		Lemma \ref{lemma:shrinking} guarantees that in each set $\widetilde{E}_i$ there exists a compact subset $E_i$  satisfying $\lambda(E_i)=\dfrac{\lambda(\widetilde{E}_i)R}{\lambda(E')}$, and such that $\avint_{E_i}fd\lambda=\avint_{\widetilde{E}_i}fd\lambda$ and moreover such that $\inf \widetilde{E}_i,\sup \widetilde{E}_i\not\in E_i$. 
		Then putting $E=\bigcup_{i=1}^N E_i$ we have that $f|_{E}$ is mean zero, $\lambda(E)=R$ and $\frac{\lambda(E_i)}{\lambda(E)}=\frac{\lambda(\widetilde{E_i})}{\lambda(E')}\in \QQ_2$ for $i=1,...,N$.
		This proves the statement for $N$, and finishes the induction.
	\end{proof}
	
	\end{lem}

    	\section{Solutions for the Homological Equation over the Cantor Set}\label{section:continuous-functions}

Let  $q\in \NN,\ r\in\RR$ and let the set
$$\mathcal{C}(q,r)=\{1,\dots,q\}\times\{1,2\}^\NN,$$
be equipped with Tikhonov topology and with such product measure  $$\mu=\mu_1\times \mu_2^\NN,$$
satisfying $\mu_1(\{i\})=\dfrac{r}{q},\ i=1,\dots,q;\ \mu_2(\{j\})=\dfrac{1}{2},\ j=1,2$.

The set $\mathcal{C}(q,r)$ is thus a Cantor type set with $\mu(\mathcal{C}(q,r))=r$. Let $V$ be a finite-dimensional vector space. Denote by $C(\mathcal{C}(q,r);V)$ the Banach space of continuous $V$-valued functions on $\mathcal{C}(q,r)$.

Denote by $p_0$ the  mapping from  $\mathcal{C}(q,r)$ onto $\{1,\dots,q\}$, given by  $p_0(i;i_1,i_2,\dots)=i$ and set $\mathcal{C}(q,r,i)=p_0^{-1}(i),\ i=1,\dots,q$.
For brevity, let $C_V=\frac{8\dim(V)^2}{\log 1.5}(S_V+1)$. 

Finally, recall that for a subset $X\subseteq V$, its diameter is defined by $\Diam(X) = \sup_{x,y\in X}\|x-y\|$.
\begin{Thm}
    \label{t_cantor}
    Let $V$ be a finite-dimensional normed vector real space.
    Let $0\neq f\in C(\mathcal{C}(q,r);V)$ be mean zero, set $$a=\dfrac{\max_{i}\{\Diam(f(\mathcal{C}(q,r,i)))\}}{\|f\|},$$
    Then there exists $g\in C(\mathcal{C}(q,r);V)$ with $\|g\|\leq (S_V+a(1+C_V))\|f\|$ and  a measure preserving continuous invertible transformation $T$ of $\mathcal{C}(q,r)$ such that 
    $f = g\circ T -g$.

Furthermore, the system of sets 
$\Gamma=\{\mathcal{C}(q,r,i),\ i=1,\dots,q\},$ can be labelled in such a way that 
$$\Gamma=\{X_1,\dots,X_q\},\ T(X_i)=X_{i+1},\ i<q,\ T(X_q)=X_1.$$
and such that $\|g|_{X_1}\|\leq (1+C_V)a\|f\|$.
\end{Thm}
\begin{proof}

For every  $n\in\NN$, we denote by $p_n$ the mapping from $\mathcal{C}(q,r)$ onto $\{1,\dots,q\}\times\{1,2\}^n$, given by setting $$p_n(i;i_1,i_2,\dots,i_n,i_{n+1},\dots)=(i;i_1,i_2,\dots,i_n).$$
  
    We now let for $n\geq 0$
    $$v_n:\{1,\dots,q\}\times\{1,2\}^n\to \{1,\dots,2^nq\}$$
    be the function that arranges the elements in
    $\{1,\dots,q\}\times\{1,2\}^n$ in lexicographical order.
    Further, for $i\in \{1,..,2^nq\}$, denote
    $$I_i^n = (p_n^{-1}(v_n^{-1}(i)).$$
  The sets  $I_i^n$, $i\in \{1,..,2^nq\}$, $n\in\NN$ are clopen and form the base of topology in $\mathcal{C}(q,r)$. Clearly, we have  $$\{I_i^0:\ i=1,\dots,q\}=\{\mathcal{C}(q,r,i):\ i=1,\dots,q\}.$$
   
    Let $f_n=\sum_{i=1}^{2^nq}\chi_{I_i^n}\avint_{I_i^n} fd\mu$. Then $f_n\in C(\mathcal{C}(q,r);V),\ \|f_n - f\| \to 0$ as $n\to \infty$,
    it follows that there exists a sequence $(n_k)_{k\geq 1}$ of natural numbers such that
    for $n\geq n_k$ we have
    $$\|f_n-f\|\leq 2^{-k-2}C_d^{-1}a\|f\|. $$
    
    Setting,
    $$h_0=f_0,$$
    $$h_1=f_{n_1}-f_0,\ \|h_1\|\leq a\|f\|,$$
    $$h_k = f_{n_k} - f_{n_{k-1}},\quad{\rm so\ that}\quad
    \|h_k\|\leq 2^{-k}C_V^{-1}a\|f\|,\quad k>1$$ we  have
    $$f = \sum_{k=0}^\infty h_k.$$

    Now, for $h_0$ let us denote by $a_i$ the value of $h_0$ taken on $I_i^0$ for
    $1\leq i\leq q$.
    As $\int fd\mu =0$ we have
    $\sum_{i=1}^q a_i = 0$ so that there is a permutation
    $\pi$ of $\{1,..,q\}$ so that
    $\|\sum_{i=1}^{m}a_{\pi(i)}\|\leq S_d \|h_0\|$ for $0\leq m\leq q$.
    Now, denote by $T_0$ the measure preserving  continuous cyclic transformation of
    $\mathcal{C}(q,r)$ sending $I_{\pi(i)}^0$ to
    $I_{\pi(i+1)}^0$ for
    $1\leq i\leq q -1$ and
    sending $I_{\pi(q)}^0$ to $I_{\pi(1)}^0$.
    We now denote by $g_0:\mathcal{C}(q,r)\to V$ the  continuous function, taking on $I_{\pi(l)}^0$ the value $\sum_{i=1}^{l-1}a_{\pi(i)}$ for $l=2,...,q$ and taking value $0$ on
    the set $I_{\pi(1)}^0$.
    Then $\|g_0\|\leq S_d\|f_0\|\leq S_d\|f\|$ and for $l=2,...,q$ and $t\in I_{\pi(l)}^0$ we have $g_0(T_0(t)) - g_0(t) = \sum_{i=1}^{l}a_{\pi(i)} - \sum_{i=1}^{l-1}a_{\pi(i)} = a_{\pi(l)} = f_0(t)$. When $l=1$ and $t\in I_{\pi(1)}$, we have  $g_0(T_0(t)) - g_0(t) = \sum _{i=1}^1 a_{\pi(i)} -0=a_{\pi(1)}=h_0(t)$.\\
    
Using the same argument as in \cite{Kwapien},  for each $k\ge 0$, we denote $J_k = \{I_i^{n_k}: 1\leq i\leq 2^{n_k}q\}$, and define a sequence $\{T_k\}_{k=0}^\infty$ of measure preserving continuous transformations $T_k$ of $\mathcal{C}(q,r)$  and  functions $\{g_k\}_{k=1}^\infty $ with   $g_k\in C(\mathcal{C}(q,r);V)$ satisfying the following:
    \begin{enumerate}[(i)]
        \item $T_k$ is a cyclic rearrangement of the sets of $J_k$.
        \item $T_{k+1}$ extends $T_k$ in the sense that if $I\in J_{k}$, $I'\in J_{k+1}$ and $I'\subseteq I$ then $T_{k+1}(I')\subseteq T_k(I)$.
        \item $\|g_k\|\leq C_V\|h_k\|$.
        \item $g_k$ is constant on all the sets $I\in J_k$.
        \item $h_k = g_k\circ T_k - g_k$ on $\mathcal{C}(q,r)$.
    \end{enumerate}

Now, we suppose that the transformations $T_0,...,T_k$ and functions $g_0,...,g_k$ with given properties have been already defined. For convenience we set $n = |J_k|$ and $m = \frac{|J_{k+1}|}{|J_k|}$.
    Let $I_1,I_2,\dots,I_{n}$ be the sets from $J_k$, enumerated so that $T_k(I_i)=I_{i+1}$ when $i<n$ and $T_k(I_n)=I_1$, which can be done since $T_k$ is a cyclic rearrangement of the sets of $J_k$.
    Furthermore, for $i=1,2,\dots,n$ let us denote by $I_{i,j}$ for $j=1,2,\dots,m$ all sets from $J_{k+1}$ which are contained in $I_i$. Denote by $a_{i,j}$ the value of the function $h_{k+1}$ on $I_{i,j}$. Since
    $$\int_{I_i}h_{k+1}d\lambda=\sum_{j=1}^{m}\int_{I_{i,j}}f_{n_{k+1}} - f_{n_k}d\lambda= 0,\quad \forall I_i\in J_k,$$
    it follows that $\sum_{j=1}^m a_{i,j}=0$ for all $i=1,\dots,n$. In addition, $\|a_{i,j}\|\leq\|h_{k+1}\|$ for all $i=1,\dots,n,\ j=1,\dots,m$.
    Therefore, by Theorem \ref{ala_kwapen_lem} it follows that there exist such rearrangements $\pi_1,\dots,\pi_n$ of the numbers $\{1,\dots,m\}$ that $$\|\sum_{i=1}^l a_{i,\pi_i(j)}\|\leq M\|h_{k+1}\|$$
    for $l=1,\dots,n$ and $j=1,\dots,m$, where  $M=\frac{8d^2}{\log 1.5}$. Define a measure preserving homeomorphism $T_{k+1}:\mathcal{C}(q,r)\longrightarrow\mathcal{C}(q,r)$, by setting
    $$T_{k+1}(I_{i,\pi_i(j)})=I_{i+1,\pi_{i+1}(j)},\quad i=1,\dots,n-1,\quad j=1,\dots,m .$$
    We set
    $$b_j=\sum_{i=1}^n a_{i,\pi_i(j)},\quad j=1,\dots,m.$$ Since $\sum_{j=1}^m b_j=\sum_{i=1}^n\sum_{j=1}^m a_{i,j}=0$ and $\|b_j\|\leq M\|h_{k+1}\|$, there exists the rearrangement $\pi_0$ of the numbers $1,\dots,m$ such that $$\|\sum_{j=1}^l b_{\pi_0(j)}\|\leq MS_V\|h_{k+1}\|,\quad \forall l=1,\dots,m.$$
    Set
    $$
    T_{k+1}(I_{n,\pi_n(\pi_0(j))})=I_{1,\pi_1(\pi_0(j+1))},\quad \forall j=1,\dots,m-1$$ and set
    $$
    T_{k+1}(I_{n,\pi_n(\pi_0(m))})=I_{1,\pi_1(\pi_0(1))}.$$
 Next, we have
    $$\|\sum_{r=0}^l h_{k+1}(T_{k+1}^r(t))\|=\|\sum_{j=1}^{p-1} b_{\pi_0(j)}+\sum_{i=1}^z a_{i,\pi_i(\pi_0(p))}\|\leq C_V\|h_{k+1}\|,$$
    where $l+1=(p-1)n+z$ for some $p\in \{1,..,m\}$ and $z\in \{1,\ldots, n\}$, for every $t\in I_{1,\pi_1(\pi_0(1))}$ and every $l=0,\dots,nm-1$. 
    
    Now, let us define the function $g_{k+1}$ by setting its value on
    $T_{k+1}^l(I_{1,\pi_1(\pi_0(1))})$ equal to $\sum_{r=0}^{l-1}
    h_{k+1}(T_{k+1}^r(t))$, where $t\in
    I_{1,\pi_1(\pi_0(1))}$ for $l=1,...,nm-1$ and setting
    $g_{k+1}(I_{1,\pi_1(\pi_0(1))})=0.$ Then we have
    $$\|g_{k+1}\|\leq C_V\|h_{k+1}\|.$$
    
    Let $t\in I_{1,\pi_1(\pi_0(1))}$.
    If
    $0<l<nm-1$, then we have
    \begin{align*}
    g_{k+1}(T_{k+1}(T_{k+1}^l(t)))-g_{k+1}(T_{k+1}^l(t))&=\sum_{r=0}^l
    h_{k+1}(T_{k+1}^r(t))-\sum_{r=0}^{l-1}
    h_{k+1}(T_{k+1}^r(t))\\&=h_{k+1}(T_{k+1}^l(t)),
    \end{align*}
    and further
    $$g_{k+1}(T_{k+1}(t))-g_{k+1}(t)=h_{k+1}(t)-0=h_{k+1}(t)$$ finally yielding
    \begin{align*}g_{k+1}(T_{k+1}(T_{k+1}^{nm-1}(t)))-g_{k+1}(T_{k+1}^{nm-1}(t))&=0 -
    \sum_{r=0}^{nm-2}h_{k+1}(T_{k+1}^r(t))\\
    &=
    h_{k+1}(T_{k+1}^{nm-1}(t)).
    \end{align*}
    Thus, for every $t\in \mathcal{C}(q,r)$ we have
    $$g_{k+1}(T_{k+1}(t))-g_{k+1}(t)=h_{k+1}(t).$$
    This completes the construction of the functions
    $\{g_k\}_{k=0}^\infty$ and transformations
    $\{T_k\}_{k=0}^\infty$ with required properties.\\
    
    It follows from the construction that $T_{k+1}$ satisfies the condition $(ii)$. Hence the sequences $T_k$ and $g_k$ satisfy the conditions $(i)-(v)$. Observe that the inverse mappings $T_k^{-1}$ also satisfy the condition $(ii)$.
    
    It follows from condition $(iii)$ that the series $\sum_{k=0}^\infty g_k$ converges in $C(\mathcal{C}(q,r);V)$ to some function $g$ satisfying 
\begin{align*}
\|g\|&\leq \|g_0\| + \|g_1\| + \sum_{k=2}^{\infty}\|g_k\| \leq S_V\|f\| + C_Va\|f\| + a\|f\| \\& = (S_V+(1+C_d)a)\|f\|.
\end{align*}
 
    Next, it follows from $(ii)$ that for all $t\in \mathcal{C}(q,r)$ the sequence $T_k(t)$ converges. We then set $T(t)=\lim_{k\rightarrow\infty} T_k(t)\in\mathcal{C}(q,r)$. In addition, $T^{-1}(t)=\lim_{k\rightarrow\infty} T^{-1}_k(t)$.
   
Suppose that  $n\in\NN,\ I\in J_n$. Then, we have $T_n(I)=I'\in J_n$. It follows from  (ii) that $T_m(I)=I', m>n$. Since $I'$ is closed, it follows that $T(I)=I'$. Hence, $T$ permutes  elements of $J_n$ for every $n$. Since $\bigcup_n J_n$ is the base of topology in $\mathcal{C}(q,r)$ and generates  the $\sigma$-algebra of measurable sets,  it follows that $T$ is  a measure preserving continuous transformation of $\mathcal{C}(q,r)$

    Now, we have for $k\geq 0$ that
    $$g_k(T(x)) - g_k(x) = g_k(T_k(x)) - g_k(x) = h_k(x).$$
    Hence,
    $$g(T(x)) - g(x) = \sum_{k=0}^\infty (g_k(T(x)) - g_k(x)) = \sum_{k=0}^{\infty}h_k = f.$$
    
The final assertion of the theorem follows from the fact that  on $J_0$, $T$ is equal to $T_0$, and the fact that $g_0$ equals $0$ on $I_{\pi(1)}^0$.
\end{proof}
 
\begin{prop}\label{p_cantors}
Let $q\in \NN,\ r\in\RR$, and $\{m_n\}$ be a sequence from $\NN$. 
On the set  $\mathcal{E}=\{1,\dots,q\}\times\prod_{n=1}^{\infty}\{1,\dots,2^{m_n}\}$ we define the product topology and product measure  $$\nu=\nu_0\times\prod_{n=1}^{\infty}\nu_n,\ \nu_0(\{i\})=\dfrac{r}{q},\ \nu_n(\{j_n\})=\dfrac{1}{2^{m_n}},\ 1\leq i\leq q,\ 1\leq j_n\leq 2^{m_n}.$$
Then there exists a measure preserving homeomorphism $\varphi:\mathcal{C}(q,r)\longrightarrow\mathcal{E}$, such that $\varphi(\mathcal{C}(q,r,i))=\{i\}\times\prod_{n=1}^{\infty}\{1,\dots,2^{m_n}\},\ i=1,\dots,q$.
\end{prop}
\begin{proof}
Let $\varphi_0$ be an identity mapping of $\{1,\dots,q\}$ onto itself, 
$$\varphi_n:\{1,2\}^{m_n}\longrightarrow\{1,\dots,2^{m_n}\}$$ be a bijection such that 
$$\varphi_n(i_1,\dots,i_{2^{m_n}})=1+\sum_{k=1}^{m_n}2^{k-1}(i_k-1),\ n \geq 1.$$

The compact $\mathcal{C}(q,r)$ can be represented as  
$$\mathcal{C}(q,r)=\{1,\dots,q\}\times\prod_{n=1}^{\infty}\{1,2\}^{m_n}.$$
Define the bijection  $\varphi:\mathcal{C}(q,r)\longrightarrow\mathcal{E}$ as the product $\varphi=\prod_{n=0}^{\infty}\varphi_n$. Since every $\varphi_n$ is measure preserving, it follows that $\varphi$ is also measure preserving.

Let  $k\in\NN,\ x\in \{1,\dots,q\}\times\prod_{n=1}^{k}\{1,2\}^{m_n},\ P(x)=x\times\prod_{n=k+1}^{\infty}\{1,2\}^{m_n}$. We have $\varphi(P(x))=(\prod_{n=0}^{k}\varphi_n)(x)\times\prod_{n=k+1}^{\infty}\{1,\dots,2^{m_n}\}$.
    
Recalling that the sets $P(x)$ form the base of topology in $\mathcal{C}(q,r)$, and the sets $\varphi(P(x))$ form the base of topology in $\mathcal{E}$, we  that conclude $\varphi$ is homeomorphism.
\end{proof}

	\section{Solutions to the Homological Equation -- The Affinely Homogeneous Setting}\label{section:affinely-homogeneous}
	In this section we show that for a mean zero affinely homogeneous function $f\in L_\infty(D;V)$ we can solve the equation $f = g\circ T - g$. The transformation $T$ we construct here is \textit{not} ergodic. We will resolve the issue of ergodicity in the next section.
	
	Note that if a function $f$ is affinely homogeneous, then for any $v\in V$ and any measurable subset $D'\subseteq D$ also $f|_{D'} + v$ is affinely homogeneous. Moreover, note that the conditions of Lemma \ref{lemma:rational-splitting} on the function $f\in L_\infty(D;V)$ are satisfied when $f$ is mean zero and affinely homogeneous.

    Let $D\subset [0,1]$ be a measurable set, $f\in L_\infty(D;V)$ be a mean zero function, $q\in \NN$, $R\in (0,\lambda(D))$, $\mathcal{F}:\mathcal{C}(q,R)\longrightarrow D$.
    The system  $(q,\mathcal{F},R)$ is said to be  \textit{a Cantor tower} for $f$, if $\mathcal{F}$ is a measure preserving continuous injection, the function $f|_{\mathcal{F}(\mathcal{C}(q,R))}$ is continuous and $\int_{\mathcal{F}(\mathcal{C}(q,R))}fd\lambda=0$.
 
    \begin{prop}\label{p_affinely_homogeneous}
        Let $V$ be a finite-dimensional real normed space. Let $D\subset [0,1]$ be a measurable set, $f\in L_\infty(D;V)$ be a mean zero affinely homogeneous function, $R\in (0,\lambda(D))$.
        
        (i). For every $q\in \NN$ there exists a Cantor tower $(q,\mathcal{F},R)$ for $f$.
        
        (ii). For every $\varepsilon>0$ there exists a Cantor tower  $(q,\mathcal{F},R)$ for $f$ such that
        \begin{align}
        \Diam(f(\mathcal{F}(\mathcal{C}(q,R,i))))<\varepsilon,\ i=1,\dots,q.\label{p_affinely_homogeneous_l1}
        \end{align}
    \end{prop}
\begin{proof}

 The assertion  of (ii) is different from that of (i) since the scalar $q$ is not given in advance there, but needs to be determined so as to satisfy the assumption (\ref{p_affinely_homogeneous_l1}).
  
Having located the sought for value of $q$ in the case (ii) the construction of the Cantor tower is done for both cases (i) and (ii) simultaneously.
  
 By Lemma \ref{lemma:integral-on-subset}, we know that there exists a compact set $K'\subset D$, with $\lambda(K')>R$ such that $f$ is continuous and mean zero on $K'$.
  
For every $\varepsilon>0$ there are points $x_0=\inf K'<x_1<\dots<x_n=\sup K'$, such that
    $\Diam(f([x_{i-1},x_i]\cap K'))<\varepsilon,\ i=1,\dots,n$.    
Let $\{K'_1,\dots,K'_m\}$ be a subfamily of     
    $\{[x_{i-1},x_i]\cap K':\ i=1,\dots,n\}$, consisting of all sets of non-zero measure.  
  
By Lemma \ref{lemma:rational-splitting}, there exist such compact sets  $K''_1\subset K'_1,\dots,K''_m\subset K'_m$, that   
  $\int_{K} fd\lambda=0,\ \lambda(K)>R$, where $K=K''_1\cup\dots\cup K''_m$, and $\dfrac{\lambda(K''_i)}{\lambda(K)}\in \QQ,\ i=1,\dots,m$. 
  
Hence the compact set $K$ admits a splitting $\{K_1,\dots,K_q\}$ inscribed into the splitting $\{K''_1,\dots,K''_m\}$, so that  $\lambda(K_1)=\dots=\lambda(K_q)$, where $q$ is the common denominator of the ratios  $\dfrac{\lambda(K''_i)}{\lambda(K)},\ i=1,\dots,m;\ \sup K_i \leq \inf K_{i+1},\ i=1,\dots,m$.
  
Thus, for the case (ii), we have located the scalar $q$ and have constructed the compacts $K_1,\dots,K_q$ so that $\Diam(f(K_i))<\varepsilon,\ i=1,\dots,q$.
  
In the case (i), we set $K=K'$ and for a $q$ given in advance, we locate the points $x_0=\inf K<x_1<\dots<x_q=\sup K$, so that $$\lambda([x_{i-1},x_i]\cap K)=\dfrac{\lambda(K)}{q},\ i=1,\dots,q.$$ 
We set $$K_i=[x_{i-1},x_i]\cap K,\ i=1,\dots,q.$$
  
Now, all we need is to build a Cantor tower $(q,\mathcal{F},R)$, so that $\mathcal{F}(\mathcal{C}(q,R,i))\subset K_i,\ i=1,\dots,q$.   
  
Let us fix a decreasing sequence $$R_0=\lambda(K)>R_1>R_2>\dots>R_n>\dots,\ \lim_n R_n =R.$$
 
We shall build a sequence $(m_n)$ of positive integers and the sets  $K_a,\ a\in \mathcal{E}_n$, where  
$$\mathcal{E}_n=\{1,\dots,q\}\times\prod_{i=1}^{n}\{1,\dots,2^{m_n}\},\ n\geq 0.$$
Below, throughout this proof, the notation $| \mathcal{E}_n|$ stands for $Card\{\mathcal{E}_n\}$.
  
For $m\leq n$, we define projection $p_{n,m}:\mathcal{E}_n\longrightarrow\mathcal{E}_m$, by setting 
$$p_{n,m}(i_0;i_1,\dots,i_m,\dots,i_n)=(i_0;i_1,\dots,i_m).$$
  
We set
$$C_n=\bigcup_{a\in\mathcal{E}_n}K_a,\ C=\bigcap_{n=0}^\infty C_n.
$$ 
Clearly, we have $C_0=K$.
  
The sets  $K_a$ should satisfy the following conditions:
  	\begin{enumerate}
      \item For $a\in \mathcal{E}_n$ the set $K_a$ is a compact subset of $[0,1]$. For $a_1,a_2\in \mathcal{E}_n$, $a_1\neq a_2$ we have either ${\rm sup} K_{a_1}\leq {\rm inf} K_{a_2}$ or else ${\rm sup} K_{a_2}\leq {\rm inf} K_{a_1}$.
      \item If $a\in \mathcal{E}_{n-1}$ and $b\in \mathcal{E}_{n}$ are such that $p_{n,n-1}(b)=a$, then $K_b\subseteq K_a$ and $\Diam(K_b)\leq\dfrac{1}{2}\Diam(K_a)$.
      \item $\lambda(K_a)=M_n:=\dfrac{R_n}{|\mathcal{E}_n|}$ for all $a\in \mathcal{E}_n$.
      \item $\int_{C_n}f d\lambda = 0$ for all $n\geq 0$.
      \item The sets $K_a\cap C_{n+1}$ and $K_b\cap C_{n+1}$ are disjoint for $a,b\in \mathcal{E}_n$ whenever $a\not=b$.
      \item $\lambda(K_a\cap C_n)=\dfrac{R_n}{|\mathcal{E}_{k}|}$ 
      for any $k<n,\ a\in \mathcal{E}_{k}$.
  \end{enumerate}

The construction of the sequence $(m_n)$ and of compact sets $K_a$ will be done via induction on $n$. 
  
If  $n=0$, the set $\mathcal{E}_0=\{1,\dots,q\}$ and the compacts $K_1,\dots,K_q$ are already determined.
  
Let $n\geq 0$ and assume that the set $\{m_1,\dots,m_n\}$ (when $n=0$ this set is empty) and the compacts $K_a,\ a\in \mathcal{E}_k,\ k\leq n$ have been determined. We define $m_{n+1}$ and $K_a$ for all $a\in \mathcal{E}_{n+1}$.

			Fix $a\in \mathcal{E}_n$ and set 
			$$K_a^L := K_a \cap \left[\inf K_a,\frac{\inf K_a + \sup K_a}{2}\right]$$ and $$K_a^R :=  K_a \cap \left[\frac{\inf K_a + \sup K_a}{2},\sup K_a\right].$$
			Observe that ${\rm Diam }(K_a^L)\leq \frac12 {\rm Diam }(K_a)$ and ${\rm Diam }(K_a^R)\leq \frac12 {\rm Diam }(K_a)$.
			We set
			$$h_a =  f|_{K_a} - \frac{1}{\lambda(K_a)}\int_{K_a}fd\lambda
			$$
			Then $h_a\in L_\infty(K_a;V)$ is such that $\int_{K_a}h_a d\lambda = 0$. We will now show that we can apply Lemma \ref{lemma:rational-splitting} to the set $K_a = K_a^L \cup K_a^R$, the mean zero function $h_a$ and the scalar $\frac{R_{n+1}}{|\mathcal{E}_n|}$.
			First of all $0<\frac{R_{n+1}}{|\mathcal{E}_n|}<\frac{R_n}{|\mathcal{E}_n|}=\lambda(K_a)$ and $\lambda(K_a^L\cap K_a^R)=0$.	Furthermore, $h_{a}$ is mean zero and as $f$ is affinely homogeneous this is also true for $h_{a}$. This shows that we can indeed apply Lemma \ref{lemma:rational-splitting}
			to obtain a subset 
$$\widetilde{K}_a = \widetilde{K}_a^L{\cup} \widetilde{K}_a^R\subset K_a\cap (\inf K_a,\sup K_a)
$$ (we emphasize the importance of the preceding inclusion for the validity of condition (5) above!) with $\widetilde{K}_a^L\subset K_a^L$ and $\widetilde{K}_a^R\subset K_a^R$ both compact and of positive measure, so that
			$\lambda(\widetilde{K}_a)= \dfrac{R_{n+1}}{|\mathcal{E}_{n}|}$ and $\int_{\widetilde{K}_a}h_a = 0$ and so that 		$\frac{\lambda(\widetilde{K}_a^L)}{\lambda(\widetilde{K}_a)} = \frac{p_a}{2^{q_a}}$ for some integer $p_a\ge 0$ and positive integer $q_a$.

			Now, set 
			$$m_{n+1} = 1+\sum_{a\in \mathcal{E}_n}q_a,\quad k_a=2^{m_{n+1}-q_a}p_a.$$
			We now select points 
			$$x_a^0<x_a^1<\dots<x_a^{k_a}=\frac{\inf K_a + \sup K_a}{2}<\dots<x_a^{2^{m_{n+1}}}
			$$
			in $K_a$ so that for $1\leq i\leq 2^{m_{n+1}}$ the sets 
			$$K_a^i:=\widetilde{K}_a\cap [x_a^{i-1},x_a^{i}]$$ all have equal measure 
			$$\lambda(K_a^i) = \frac{\lambda(\widetilde{K}_a)}{2^{m_{n+1}}} = \frac{R_{n+1}}{|\mathcal{E}_{n}|2^{m_{n+1}}}$$  and moreover  
			$$K_a^i\subset \widetilde{K}_a^L,\quad \forall i\leq k_a$$ and 
			$$
			K_a^i\subset \widetilde{K}_a^R,\quad \forall k_a<i\leq 2^{m_{n+1}}.
			$$ 
			
			Now if $b = a\times i \in \mathcal{E}_{n+1}$ with $1\leq i\leq 2^{m_{n+1}}$ then we define $K_b=K_a^i$. 
            
By the construction, conditions (1)-(3) hold for  $K_c,\ c\in \mathcal{E}_{n+1}$,  and for  $a,b\in\mathcal{E}_{n}$ condition (5) is also satisfied.
            
Now, we verify that condition (4) holds.             
            \begin{align*}
            \int_{C_{n+1}}fd\lambda &=\sum_{a\in \mathcal{E}_{n}}\int_{\widetilde{K}_a}fd\lambda\\
            &=\sum_{a\in \mathcal{E}_n} \frac{\lambda(\widetilde{K}_a)}{\lambda(K_a)}\int_{K_a}fd\lambda d\lambda\\
            &=\sum_{a\in \mathcal{E}_n} \frac{R_{n+1}}{|\mathcal{E}_{n}|\lambda(K_a)}\int_{K_a}fd\lambda d\lambda\\
            &=\frac{R_{n+1}}{|\mathcal{E}_{n}|M_n}\sum_{a\in \mathcal{E}_n} \int_{K_a}fd\lambda d\lambda = \frac{R_{n+1}}{|\mathcal{E}_{n}|M_n}\int_{C_n}fd\lambda d\lambda = 0
            \end{align*} 
            \\
 Now, we verify that condition (6) holds.             
 To this end, we observe that the number of compacts $K_x,\ x\in\mathcal{E}_{n+1},$ contained inside of the set $K_a,\ a\in\mathcal{E}_{n}$ is equal to $\dfrac{|\mathcal{E}_{n+1}|}{|\mathcal{E}_n|}$.
            Hence, for $k<n,\ a\in\mathcal{E}_k$ the number of  $K_x,\ x\in\mathcal{E}_{n},     $ contained inside of the set $K_a$ is equal to
            $$\dfrac{|\mathcal{E}_{n}|}{|\mathcal{E}_{n-1}|}\cdot\dfrac{|\mathcal{E}_{n-1}|}{|\mathcal{E}_{n-2}|}\cdot\dots\cdot\dfrac{|\mathcal{E}_{k+1}|}{|\mathcal{E}_{k}|}=\dfrac{|\mathcal{E}_{n}|}{|\mathcal{E}_{k}|}.$$
Therefore, we have  $\lambda(K_a\cap C_n)=\dfrac{|\mathcal{E}_{n}|}{|\mathcal{E}_{k}|}\cdot\dfrac{R_n}{|\mathcal{E}_n|}=\dfrac{R_n}{|\mathcal{E}_k|}$.\\

	This completes the construction of compacts $K_a$.\\
Now, we show that $$\lambda(C)=R,\ \int_{C}fd\lambda=0.$$
    
 Indeed, we have $C_{n+1}\subset C_n,\ n\geq 0,\ \lambda(C)=\lim_n\lambda(C_n)=\lim_n R_n=R$ and $\|\int_{C}fd\lambda\| \leq  \|\int_{C_n}fd\lambda\| + \lambda(C_n\setminus C)\|f\|_\infty = \lambda(C_n\setminus C)\|f\|_\infty=(R_n-R)\|f\|_\infty \to 0$ as $n\to \infty$.\\
    
Further, by Proposition \ref{p_cantors}, we may identify $\mathcal{C}(q,R)$ with $\mathcal{E}_\infty=\{1,\dots,q\}\times\prod_{i=1}^{\infty}\{1,\dots,2^{m_n}\}$ with the measure $\nu$ introduced there. 
    
For every $n\geq 0$, we define the mapping $p_n:\mathcal{E}_\infty\longrightarrow\mathcal{E}_n$, by setting 
$$p_n(i_0;i_1,\dots,i_n,i_{n+1},\dots)=(i_0;i_1,\dots,i_n).$$
    
For every $a\in\mathcal{E}_\infty$, we set $\mathcal{F}(a)=\bigcap_{n=0}^\infty K_{p_n(a)}$. Combining the equality $p_{n+1,n}\circ p_{n+1}=p_n$ and condition (2), we infer that $|\mathcal{F}(a)|=Card(\mathcal{F}(a))=1$. Therefore, the mapping $\mathcal{F}:\mathcal{E}_\infty\longrightarrow C$ is correctly defined. 
    
 Let $a,b\in\mathcal{E}_\infty,\ a\neq b$. Then, there exists $n$, such that $p_n(a)\neq p_n(b)$. Due to  (5), we have $\mathcal{F}(a)\neq \mathcal{F}(b)$, i.e.\@ the mapping $\mathcal{F}$ is injective.
             
Let $x\in C$. It follows from the construction of $C$ that for every $n$ there exists a unique $a_n\in \mathcal{E}_n$, such that $x\in K_{a_n}$. Appealing to conditions (1) and (2), we infer that $p_{n,n-1}(a_n)=a_{n-1},\ n>1$. Hence, there exists $a\in \mathcal{E}_\infty$, such that $p_n(a)=a_n$. This guarantees $\mathcal{F}(a)=x$, and so the mapping $\mathcal{F}$ is surjective.
             
Suppose that the sequence $\{a_{(n)}\}\subset \mathcal{E}_\infty$ converges to $a\in \mathcal{E}_\infty$. This means that for every  $n$ there exists the index  $k_n$, such that $p_{n}(a_{(m)})=p_n(a)$ when $m>k_n$, that is. $\mathcal{F}(a_{(m)})\in K_{p_n(a)}$.                 
Due to (2),we have  $|\mathcal{F}(a_{(m)})-\mathcal{F}(a)|\leq \dfrac{1}{2^n}$. This implies that the mapping $\mathcal{F}$ is continuous.
             
Recalling that the set $\mathcal{E}_\infty$ is compact, we infer that the mapping $\mathcal{F}^{-1}$ is also continuous. 
             
Appealing to (6), we see that $$\lambda(K_a\cap C)=\dfrac{R}{|\mathcal{E}_n|}, \quad \forall n\geq 0,\ a\in \mathcal{E}_n.$$
 However, $K_a\cap C=\mathcal{F}(p_n^{-1}(a))$, $\mu(p_n^{-1}(a))=\dfrac{R}{|\mathcal{E}_n|}$, that is 
 $$\lambda(\mathcal{F}(p_n^{-1}(a)))=\mu(p_n^{-1}(a)).
 $$
             Taking into account that the sets  $p_n^{-1}(a),\ n\geq 0,\ a\in \mathcal{E}_n,$ generate the $\sigma$-algebra of measurable subsets in $\mathcal{E}_\infty$, we conclude that the mapping $\mathcal{F}$ is measure preserving.            
		\end{proof}
    \begin{prop}
           \label{t_local_affinely_homogeneous}
           Let $V$ be a finite-dimensional real normed space.
           Let $D\subseteq [0,1]$ be of positive measure and let $f\in L_\infty(D;V)$ be mean zero and affinely homogeneous. Then for any $\varepsilon>0$ and $R\in (0,\lambda(D))$ there exists measurable set $C\subset D,\ \lambda(C)=R$, $g\in L_\infty(C;V)$ with $\|g\|\leq (S_V+\varepsilon)\|f\|_\infty$ and  a mod $0$ measure preserving invertible transformation $T$ of $C$ such that $f = g\circ T -g$.
    \end{prop}   
    \begin{proof}
It follows from Proposition \ref{p_affinely_homogeneous}(ii) that for $f$ there exists a Cantor tower $(q,\mathcal{F},R)$, such that
       $$
        \Diam(f(\mathcal{F}(\mathcal{C}(q,R,i))))<\dfrac{\varepsilon}{(1+C_V)},\ i=1,\dots,q.
       $$
The proof is completed by appealing to Theorem \ref{t_cantor}.
    \end{proof}

\begin{Thm}
    \label{t_affinely_homogeneous}
    Let $V$ be a finite-dimensional real normed space.
    Let $D\subseteq [0,1]$ be of positive measure and let $f\in L_\infty(D;V)$ be mean zero and affinely homogeneous. Then for any $\varepsilon>0$ there is a $g\in L_\infty(D;V)$ with $\|g\|\leq (S_V+\varepsilon)\|f\|_\infty$ and there is a mod $0$ measure preserving invertible transformation $T$ of $D$ such that $f = g\circ T -g$.
\begin{proof}
By Zorn's lemma, there exists a maximal family   $\{K_i\}_{i\in I}$ of pairwise disjoint compact subsets of $D$ with positive measure, such that
there exists $g_i\in L_\infty(K_i;V),\ \|g_i\|_\infty\leq (S_V+\varepsilon)\|f\|_\infty$ and a mod $0$ measure preserving transformation $T_i$ of $K_i$ such that $f|_{K_i} = g_i\circ T_i -g_i$. Clearly, the set of indices $I$ is, at most, countable.

It suffices to show that $\lambda(D\setminus\bigcup_{i\in I}K_i)=0$. Indeed, in this case we define $g$ and $T$ so that $g|_{K_i}=g_i,\ T|_{K_i}=T_i$ for any $i\in I$.

Suppose that the set $D_0:=D\setminus\bigcup_{i\in I}K_i$ has a non-zero measure.

By Proposition \ref{t_local_affinely_homogeneous}, in
$D_0$ then there exists a compact subset  $K_0$, such that $f|_{K_0}=g_0\circ T_0-g_0$ for some function 
$g_0\in L_\infty(K_0;V),\ \|g_0\|_\infty\leq (S_V+\varepsilon)\|f\|_\infty,$
and a mod $0$ measure preserving transformation $T_0$ of $K_0$.
This is a contradiction with the assumption concerning maximality of the family $\{K_i\}_{i\in I}$. Hence $\lambda(D_0)=0$.

\end{proof}
\end{Thm}

	\section{Proof of main results for general mean zero functions}\label{section:proof-main-theorem}
We begin this section with the two lemmas, which are based on  classical results.
\begin{lem}\label{lemma:diophantine-approximation}
    Given $x\in\RR^n$ with $1,x_1,....,x_n$ rationally independent, and let $\varepsilon>0$. Then for any given non-zero vector $v\in \RR^n$ there are integers $q\geq 1$, $p_1,...,p_n\in \ZZ$ so that for the vector $w\in\RR^n$ with $w_l = \frac{p_l}{q} -x_l$ we have $\|w\|_\infty<\frac{\varepsilon}{q}$
    and so that  we have $(w,v)>0$.
    \begin{proof}
        Let us denote $\alpha_l = sign(v_l)\frac{\varepsilon}{2}$.
        
        Since $1,x_1,...,x_n$ are rationally independent, we can by \cite[Theorem 442]{Hardy_Wright} find integers $q\geq 1$ and $p_1,...,p_n\in \ZZ$ such that
        $$|qx_l - p_l + \alpha_l|< \frac{\varepsilon}{2}$$
        Now since $|qx_l - p_l|<\frac{\varepsilon}{2} + |\alpha_l| = \varepsilon$ we have for the vector $w\in \RR^n$ given by $w_l = \frac{p_l}{q} - x_l$ that $\|w\|_\infty<\frac{\varepsilon}{q}$.		
        Moreover, since $|\alpha_l| = \frac{\varepsilon}{2}$ we have that $sign(w_l) = sign(p_l - x_lq) = sign(\alpha_l) = sign(v_l)$. 
        Note furthermore that, since $x_l$ is irrational for all $l$, we have that $|w_l| = |\frac{p_l}{q} - x_l|>0$. Now, since also $v\not=0$ we have that $(w,v) >0$. \\
    \end{proof}
\end{lem}

\begin{lem}\label{lemma:measure_convergent}
Let  $\{T_n\}$ be a sequence of mod $0$ measure preserving automorphisms  of the interval $[0,1]$ and let $T_n\longrightarrow T,\ T_n^{-1}\longrightarrow S$ in measure. Then
    $T,S$ are also mod $0$ measure preserving automorphisms  of the interval $[0,1]$,  and $S=T^{-1}$ a.e.
\end{lem}
\begin{proof}
 The fact that $T$  $S$ are measure preserving  follows from  \cite[Prop. 9.9.10]{Bogachev2}. The fact that the equality 
 $S=T^{-1}$ holds a.e. follows from \cite[Coroll. 9.9.11]{Bogachev2}.
\end{proof}

The following lemma plays a crucial role in the proof of our main result.

\begin{lem}\label{lemma:countable-partition-of-set}
	Let $V$ be a finite-dimensional real normed space.
    Let $f\in L_\infty([0,1];V)$ be mean zero and $\varepsilon>0$. There exists a sequence $(q_i)_{i\geq 1}$ in $\NN$ with $q_i\geq 2$, and a partition $\{A_{i,j}: i\geq 1, 1\leq j\leq q_i\}$ of $[0,1]$ into the sets of positive measure, together with a measure preserving invertible transformation 
    $T$ on $[0,1]$ such that $$
    T(A_{i,j}) = A_{i,j+1}, \quad \forall i\geq 1,\quad 1\leq j<q_i.$$
     Moreover, this can be done so that if we denote $A=\bigcup_{i=1}^{\infty}A_{i,1}$ and define $h:A\to V$ by $h|_{A_{i,1}} = \sum_{j=1}^{q_i}f\circ T^{j-1}$, then we have that $h$ is mean zero and has $\|h\|_\infty<\varepsilon$. Also we ensure for $i\geq 1$ that
$$\|\sum_{j=1}^{l-1}f\circ T^{j-1}\|_{L_\infty(A_{i,1};V)}< S_V\|f\|_{\infty}+\varepsilon,\ l=1,...,q_i.$$
    \begin{proof}
       By Theorem \ref{t-reduction-affinely-partially-homogeneous}, we can find a subset $D\subseteq [0,1]$ of positive measure on which $f$ is mean zero, and an integer $n\geq 1$ and a partition $\{D_1,...,D_{n}\}$ of $D$ s.t. $f|_{D_l}$ is affinely homogeneous for $l=1,...,n$. 
        
        Let $\varepsilon>0$. We will set $$\varepsilon' = \frac{1}{2}\min\{1,\frac{\varepsilon}{3},\min_i(\lambda(D_i))\}>0.$$ 
       
       Every function $f_{(l)}=f|_{D_l}-\avint_{D_l}fd\lambda$ is mean zero and affinely homogeneous on $D_l,\ l=1,\dots,n$.
        
By  Proposition \ref{p_affinely_homogeneous}(ii), we know that for every $l=1,\dots,n$ for the function $f_{(l)}$ there exists a Cantor tower $(q_{(l)},\mathcal{F}_{(l)},\frac{1}{2}\lambda(D_l))$, satisfying 
         $$K_l:=\mathcal{F}_{(l)}(\mathcal{C}(q_{(l)},\frac{1}{2}\lambda(D_l))))\subset D_l,$$
         $$K_{l,m}:=\mathcal{F}_{(l)}(\mathcal{C}(q_{(l)},\frac{1}{2}\lambda(D_l),m)),$$
         $$\Diam(f(K_{l,m}))<\dfrac{\varepsilon'}{(1+C_V)(\|f\|_\infty+1)},\ m=1,\dots,q_{(l)}$$
         (recall that $C_V=\frac{8\dim(V)^2}{\log 1.5}\cdot(S_V+1)$).
         
         Then $\lambda(K_l)=\frac{1}{2}\lambda(D_l),\ \avint_{K_l}fd\lambda = \avint_{D_l}fd\lambda$ for any $l$. Define $K=\bigcup_{l=1}^n K_l$. Then
         $$\int_K fd\lambda=\sum_{l=1}^{n}\int_{K_l} fd\lambda=\sum_{l=1}^{n}\dfrac{\lambda(K_l)}{\lambda(D_l)}\int_{D_l} fd\lambda=\dfrac{1}{2}\int_{D} fd\lambda=0.$$
                  
         Then $\{K_{l,1},...,K_{l,q_{(l)}}\}$ is a partition of $K_l$, such that $\lambda(K_{l,m})=\frac{\lambda(K_l)}{q_{(l)}}$ for all $l,m$.\\

        Let $x\in\RR^n$ be the vector given by $x_l = \frac{\lambda(K_l)}{q_{(l)}\lambda(K)}>0$. We can find a maximal subset $\mathcal{J}\subseteq \{1,...,n\}$ of indexes s.t. $\{1\}\cup \{x_{j}:j\in \mathcal{J}\}$ are rationally independent. Then
        $b_l x_l= a_{l,0} + \sum_{j\in \mathcal{J}}a_{l,j}x_j$ for $l=1,...,n$ for some integers $a_{l,j}$ and non-zero integers $b_l$. We set $M = 2\left|\prod_{l=1}^nb_l\right|\cdot\max\{|a_{l,j}|:j\in \mathcal{J}\cup \{0\}, l=1,...,n\}$.			
        We will denote $q_{(0)}=\max_{l}q_{(l)}$ and $\rho = \min_{l} \lambda(K_l)>0$ and $x_0 = \min_{l}x_l>0$ and
        $N=nM^2\max\{\frac{q_{(0)}}{\rho},2nq_{(0)},\frac{2}{x_0},\frac{nq_{(0)}\|f\|_\infty}{\varepsilon'}\}$.\\
        
        If $\mathcal{J}$ is empty, we set $\widetilde{q}_i=1$ for $i\in\NN$.
        Now suppose $\mathcal{J}$ is non-empty. We will denote $\RR^{\mathcal{J}}$ for the vector space of functions $\mathcal{J}\to \RR$ equipped with the Euclidean norm and we denote $\mathbb{S}(\mathcal{J})$ for the set of all unit vectors in $\RR^{\mathcal{J}}$. For every $v\in \mathbb{S}(\mathcal{J})$ we can by Lemma \ref{lemma:diophantine-approximation} find integers $\widetilde{q_{v}}\geq 1$ and $\widetilde{p_{v,j}}\in \ZZ$ for $j\in \mathcal{J}$ 
        s.t.
        for the vector
        $\widetilde{w_v}\in \RR^{\mathcal{J}}$ given by $(\widetilde{w_v})_j = \frac{\widetilde{p_{v,j}}}{\widetilde{q_v}} - x_j$ we have $\|\widetilde{w_v}\|_\infty<\frac{1}{\widetilde{q_v}N}$		
        and so that $(\widetilde{w_v},v)> 0$. We can now find a sequence $(\xi_i)_{i\geq 1}$ in $\mathbb{S}(\mathcal{J})$ such that $\{\widetilde{w_{\xi_i}}: i\in \NN\}$ is dense in $\{\widetilde{w_v}:v\in \mathbb{S}(\mathcal{J})\}$. 	
        Now, by the density we have for $v\in \mathbb{S}(\mathcal{J})$ that we can find $i\geq 1$ such that $\|\widetilde{w_{\xi_i}}-\widetilde{w_{v}}\|_2<( \widetilde{w_v},v)$, and hence 
        $$(\widetilde{w_{\xi_i}},v) =  (\widetilde{w_{v}},v) + (\widetilde{w_{\xi_i}}-\widetilde{w_{v}},v)
        \geq (\widetilde{w_{v}},v) - \|\widetilde{w_{\xi_i}}-\widetilde{w_{v}}\|_2>0.
        $$ 
        For $i\geq 1$ we will now denote the integers $\widetilde{q}_i := \widetilde{q}_{\xi_i}$ and  $\widetilde{p_{i,j}}:=\widetilde{p}_{\xi_i,j}$ for $j\in\mathcal{J}$ and furthermore denote the vector $\widetilde{w_{i}} := \widetilde{w_{\xi_i}}$ for which we have the bound
        $$\|\widetilde{w_i}\|_\infty = \|\widetilde{w_{\xi_i}}\|_\infty\leq \frac{1}{\widetilde{q_{\xi_i}}N}=\frac{1}{\widetilde{q_i}N}.
        $$ Moreover, by what we just showed we have that for every non-zero $v\in \RR^{\mathcal{J}}$ we can find $i\geq 1$ such that $(\widetilde{w_i},v) >0$.\\
        
        Regardless on whether $\mathcal{J}$ is empty or non-empty, we now fix $i\geq 1$ and set $q_i = M\widetilde{q_i}\geq 2$ and for $l=1,...,n$ we denote {
            $$p_{i,l} = \frac{q_i}{b_l\widetilde{q_i}}\left(a_{l,0}\widetilde{q_i} + \sum_{j\in \mathcal{J}}a_{l,j}\widetilde{p_{i,j}}\right)=\frac{q_i}{b_l}\left(a_{l,0} + \sum_{j\in \mathcal{J}}a_{l,j}\frac{\widetilde{p_{i,j}}}{\widetilde{q_i}}\right),$$ 
            which is an integer since $\frac{q_i}{b_l\widetilde{q}_i} =\frac{M}{b_l}\in\NN$.}
        We now define the vector $w_i\in \RR^n$ as $$(w_i)_l := \frac{p_{i,l}}{q_i} - x_l = \sum_{j\in \mathcal{J}}\frac{a_{l,j}}{b_l}\left(\frac{\widetilde{p_{i,j}}}{\widetilde{q_i}}-x_j\right)$$
        
        If $\mathcal{J}$ is empty, this is the empty sum so that $w_i =0$. If $\mathcal{J}$ is non-empty we have the bound
        \begin{align*}
        \|w_i\|_\infty  
        \leq \sum_{j\in \mathcal{J}}M\cdot\|\widetilde{w_i}\|_\infty
        \leq \frac{nM}{\widetilde{q_i}N}
        = \frac{1}{q_i}\cdot \frac{nM^2}{N}
        \end{align*}

        For $i\geq 1$ we now choose $c_i>0$ with $\sum_{i=1}^\infty c_i=\frac{1}{3}\lambda(K)$. For $1\leq l\leq n$ and $1\leq m\leq q_{(l)}$ we then have
        \begin{align*}
        \sum_{i=1}^\infty c_i \frac{p_{i,l}}{q_i}=& \sum_{i=1}^\infty c_i \left(\frac{\lambda(K_l)}{q_{(l)}\lambda(K)} + (w_i)_l\right)\leq 
        \sum_{i=1}^{\infty} \frac{c_i}{\lambda(K)} \left(\frac{\lambda(K_l)}{q_{(l)}} + \frac{\rho}{q_{(0)}}\right) \\  &\leq 2\frac{\lambda(K_{l})}{q_{(l)}}\sum_{i=1}^\infty \frac{c_i}{\lambda(K)} < \lambda(K_{l,m})
        \end{align*}
        
        Now, for $i\geq 1$ let $\{I_{l,m}^{(i)}: 1\leq l\leq n, 1\leq m\leq q_{(l)}\}$ be a partition of $\{1,...,q_i\}$ so that 
        $I_{l,m}^{(i)}$ has size $|I_{l,m}^{(i)}|=p_{i,l}$. 
Let us fix the bijections $$\alpha_{l,m}^{(i)} :I_{l,m}^{(i)}\longrightarrow\{1,\dots,p_{i,l}\},\ \beta_{l,m}^{(i)} :I_{l,m}^{(i)}\times\{1,2\}^\NN\longrightarrow\{1,\dots,p_{i,l}\}\times\{1,2\}^\NN$$
 by setting $$\beta_{l,m}^{(i)}(i_0;i_1,\dots)=(\alpha_{l,m}^{(i)}(i_0);i_1,\dots).$$ 
        We note that such a partition actually exists. Namely, $q_i = \sum_{l=1}^nq_{(l)} p_{i,l}$ because of the fact that the difference
         \begin{align*}
         |q_i-\sum_{l=1}^nq_{(l)}p_{i,l}|& \leq \sum_{l=1}^{n}|q_i\frac{\lambda(K_l)}{\lambda(K)}-q_{(l)}p_{i,l}|\leq q_i\sum_{l=1}^{n}q_{(l)} |x_l-\frac{p_{i,l}}{q_i}|\\ &\leq q_i nq_{(0)}\|w_i\|_\infty\leq  \frac{n^2M^2\kappa}{N}\leq \frac{1}{2}
         \end{align*} 
          is an integer. Also, importantly $p_{i,l}$ is an integer with
          $p_{i,l}\geq q_ix_l-|q_ix_l - p_{i,l}|>x_l-q_i\|w_i\|_\infty\geq x_l -\frac{x_0}{2}>0$, which shows that $p_{i,l}\in\NN$.\\
        
        We shall now do the following for all $1\leq l\leq n$ and $1\leq m\leq q_{(l)}$. By Proposition \ref{p_affinely_homogeneous}(i) we can find a       
        Cantor tower $(p_{1,l},\mathcal{F}^{(1)}_{l,m},c_{1}\frac{p_{1,l}}{q_{1}})$ for the function $f|_{K_{l,m}}-\avint_{K_{l,m}}fd\lambda$. We have
        \begin{align*}
        E_{l,m}^{(1)}:=&\mathcal{F}^{(1)}_{l,m}(\mathcal{C}(p_{1,l},c_{1}\frac{p_{1,l}}{q_{1}}))\subset K_{l,m},\ \lambda(E_{l,m}^{(1)})=c_{1}\frac{p_{1,l}}{q_{1}},\\
        \avint_{E_{l,m}^{(1)}}fd\lambda &= \avint_{K_{l,m}}fd\lambda.
        \end{align*}    
        
        Now, since $\lambda(K_{l,m}\setminus E_{l,m}^{(1)})>\sum_{i=2}^{\infty} c_i\frac{p_{i,l}}{q_i}>c_2\frac{p_{2,l}}{q_2}$ and arguing as above, we can find a measure preserving homeomorphism
    $$\mathcal{F}^{(2)}_{l,m}:\mathcal{C}(p_{2,l},c_{2}\frac{p_{2,l}}{q_{2}})\longrightarrow E_{l,m}^{(2)}\subseteq K_{l,m}\setminus E_{l,m}^{(1)},$$ such that $$\avint_{E_{l,m}^{(2)}}fd\lambda = \avint_{K_{l,m}\setminus  E_{l,m}^{(1)}}fd\lambda=\avint_{K_{l,m}}fd\lambda.$$
Repeating the same arguments further, for every $i\geq 1$, we can find a measure preserving homeomorphism $$\mathcal{F}^{(i)}_{l,m}:\mathcal{C}(p_{i,l},c_{i}\frac{p_{i,l}}{q_{i}})\longrightarrow E_{l,m}^{(i)}\subseteq K_{l,m}\setminus (E_{l,m}^{(1)}\cup\dots\cup E_{l,m}^{(i-1)}),$$ such that $\avint_{E_{l,m}^{(i)}}fd\lambda = \avint_{K_{l,m}}fd\lambda$.    
        
        Now that we have defined the sets $E_{l,m}^{(i)}$ for $i\geq 1$, $1\leq l\leq n$ and $1\leq m\leq q_{(l)}$ that are all pairwise disjoint, we define the sets $E^{(i)} = \bigcup_{l=1}^n\bigcup_{m=1}^{q_{(l)}} E_{l,m}^{(i)}$ for $i\geq 1$. These sets for $i\geq 1$ are then also pairwise disjoint and have measure $\lambda(E^{(i)}) = \sum_{l=1}^n\sum_{m=1}^{q_{(l)}}c_i\frac{p_{i,l}}{q_i} = \frac{c_i}{q_i}\sum_{l=1}^{n}q_{(l)} p_{i,l} = c_i$. Furthermore, for every $i$, gluing homeomorphisms $\mathcal{F}^{(i)}_{l,m}$, we obtain  a measure preserving homeomorphism  $\mathcal{F}_{i}:\mathcal{C}(q_i,c_{i})\longrightarrow E^{(i)}$ defined by setting $\mathcal{F}_{i}|_{I_{l,m}^{(i)}\times\{1,2\}^\NN}=\mathcal{F}^{(i)}_{l,m}\circ\beta_{l,m}^{(i)}$.
        
        We then also have for the sets $\widetilde{A}_{i,j} :=\mathcal{F}_{i}(\mathcal{C}(q_i,c_{i},j))$ for $j=1,...,q_{i}$ that $\widetilde{A}_{i,j}\subseteq E_{l,m}^{(i)}\subseteq K_{l,m}$
        where $l,m$ are such that $j\in I_{l,m}^{(i)}$. 
        This means that for $j=1,...,q_i$ we have
        $\Diam(f(\widetilde{A}_{i,j}))<(1+C_V)^{-1}(\|f\|_\infty+1)^{-1}\varepsilon'$ as we have this bound for $\Diam(f(K_{l,m}))$ for all $1\leq l\leq n$ and $1\leq m\leq q_{(l)}$.
        Furthermore, since $E^{(i)}\subseteq K$ we have that $f|_{E^{(i)}}$ is continuous. Also,  we have
        \begin{align*}\int_{E^{(i)}}fd\lambda &= \sum_{l=1}^{n}\sum_{m=1}^{q_{(l)}}
        \int_{E_{l,m}^{(i)}} fd\lambda\\
        &= \sum_{l=1}^{n}\sum_{m=1}^{q_{(l)}}
        \frac{\lambda(E_{l,m}^{(i)})}{\lambda(K_{l,m})}\int_{K_{l,m}} fd\lambda\\
        &=c_i\sum_{l=1}^{n}\frac{p_{i,l}}{q_i}\cdot\frac{q_{(l)}}{\lambda(K_l)}\int_{K_{l}} fd\lambda\\
        &=c_i\sum_{l=1}^{n}
        \left(\frac{p_{i,l}}{q_i} - \frac{\lambda(K_l)}{q_{(l)}\lambda(K)}\right)\cdot
        \frac{q_{(l)}}{\lambda(K_l)}\int_{K_{l}} fd\lambda\\
        &=\lambda(E^{(i)})\sum_{l=1}^{n}
        (w_i)_lq_{(l)}\avint_{K_{l}} fd\lambda
        \end{align*}
        Therefore we have that $\int_{E^{(i)}}fd\lambda =0$ whenever $w_i=0$ and in general we have the bound
        \begin{align*}\|\avint_{E^{(i)}}fd\lambda\|&\leq \sum_{l=1}^{n}\|(w_i)_lq_{(l)} \avint_{K_l}fd\lambda\|
        \leq nq_{(0)}\|w_i\|_\infty\|f\|_\infty
        \\ &\leq \frac{1}{q_i}\cdot \frac{n^2M^2q_{(0)}\|f\|_\infty}{N}\leq \frac{\varepsilon'}{q_i}.
        \end{align*}
Now, define a function $f^{(i)}:E^{(i)}\to V$ by 
$$f^{(i)} := f|_{E^{(i)}} - \avint_{E^{(i)}}fd\lambda,$$ 
which is mean zero and continuous. Moreover, we have for $j\in \{1,...,q_i\}$ that $$\Diam(f^{(i)}(\widetilde{A}_{i,j}))=\Diam(f(\widetilde{A}_{i,j}))<\dfrac{\varepsilon'}{(1+C_V)(\|f\|_\infty+1)}.$$ 
Therefore, it follows from Theorem \ref{t_cantor} that when $f^{(i)}\not=0$ there exists a
        rearrangement $\{A_{i,j}\}_{j=1}^n$ of the sets $\{\widetilde{A_{i,j}}\}_{j=1}^n$, a measure preserving invertible transformation $T^{(i)}$ of $E^{(i)}$ such that $T^{(i)}(A_{i,j}) = A_{i,j+1}$ for $j=1,..,q_i-1$ and $T^{(i)}(A_{i,q_i})=A_{i,1}$, and we obtain a function $g^{(i)}\in L_\infty(E^{(i)};V)$ with $\|g^{(i)}\|_\infty\leq S_V\|f^{(i)}\|_\infty + \varepsilon'$ and $\|g^{(i)}|_{A_{i,1}}\|_\infty\leq \varepsilon'$ and so that $f=g^{(i)}\circ T^{(i)} - g^{(i)}$. When $f^{(i)}=0$ this can also be done by simply taking $g^{(i)}=0$ and taking $T^{(i)}$ in the given form.\\

        We will now define a transformation $T:\bigcup_{i\geq 1} E^{(i)}\to \bigcup_{i\geq 1} E^{(i)}$ as $T|_{A_{i,j}} := T^{(i)}|_{A_{i,j}}$. We now set $A = \bigcup_{i\geq 1} A_{i,1}$ and  $h:A\to V$ as $h|_{A_{i,1}} = \sum_{j=1}^{q_i}f\circ T^{j-1}$.
        We then have
        \begin{align*}
        \|h|_{A_{i,1}}\|_\infty &\leq
        \|\sum_{j=1}^{q_i}f^{(i)}\circ T^{j-1}\|_{L_\infty(A_{i,1};V)} + q_i\|\avint_{E^{(i)}}fd\lambda\| \\
        &\leq \|\sum_{j=1}^{q_i}f^{(i)}\circ (T^{(i)})^{j-1}\|_{L_\infty(A_{i,1};V)} + \varepsilon'\\
        &\leq \|g^{(i)}\circ (T^{(i)})^{q_i} - g^{(i)}\|_{L_\infty(A_{i,1};V)}+ \varepsilon'\\
        &\leq \|g^{(i)}\circ (T^{(i)})^{q_i}\|_{L_\infty(A_{i,1};V)} + \|g^{(i)}\|_{L_\infty(A_{i,1};V)}+ \varepsilon'\\
        &\leq 2\|g^{(i)}|_{A_{i,1}}\|_\infty + \varepsilon'\\
        &\leq 3\varepsilon'
        \end{align*}
        and hence $\|h\|_\infty<\varepsilon$. Furthermore, we have for $l=1,...,q_i$ that
            \begin{align*}
            \|\sum_{j=1}^{l-1}f\circ T^{j-1}\|_{L_{\infty}(A_{i,1};V)}
            &\leq 	\|\sum_{j=1}^{l-1}f^{(i)}\circ T^{j-1}\|_{L_{\infty}(A_{i,1};V)} + q_i\|\avint_{E^{(i)}}fd\lambda\|\\
            &\leq 	\|g^{(i)}\circ T^{l} - g^{(i)}\|_{L_{\infty}(A_{i,1};V)} + \varepsilon'\\
            &\leq \|g^{(i)}\circ T^{l}\|_{L_\infty(A_{i,1};V)} + \|g^{(i)}\|_{L_{\infty}(A_{i,1};V)} + \varepsilon'\\
            &\leq \|g^{(i)}\|_{\infty} + \|g^{(i)}|_{A_{i,1}}\|_{\infty} + \varepsilon'\\
            &\leq (S_V\|f^{(i)}\|_{\infty}+\varepsilon') + \varepsilon' + \varepsilon'\\
            &\leq S_V\|f\|_\infty + 3\varepsilon'\\
            &< S_V\|f\|_\infty +\varepsilon
            \end{align*}

        We will now show that there exists a subset $A'\subseteq A$ of positive measure for which $h|_{A'}$ is mean zero. If $\mathcal{J}$ is empty, then $w_i = 0$ for $i\geq 1$, so that we already have $$\int_{A}hd\lambda = \sum_{i=1}^{\infty}\int_{A_{i,1}}hd\lambda = \sum_{i=1}^{\infty}\int_{E^{(i)}}fd\lambda=0$$ and can therefore take $A' = A$. Now assume that $\mathcal{J}$ is non-empty. Let $u\in V$ and define $v\in \RR^{\mathcal{J}}$ as $v_j = \sum_{l=1}^{n}
        \frac{a_{l,j}}{b_l}q_{(l)}( u,\avint_{K_{l}} fd\lambda)$.
        Now for $i\geq 1$ we have
        \begin{align*}
        (u,\avint_{A_{i,1}}hd\lambda)
        &= \frac{q_i}{\lambda(E^{(i)})}( u,\int_{E^{(i)}}fd\lambda)\\ 
        &=q_i( u,\sum_{l=1}^{n}
        (w_i)_lq_{(l)}\avint_{K_{l}} fd\lambda)\\
        &=q_i\sum_{l=1}^{n}
        (w_i)_lq_{(l)}( u,\avint_{K_{l}} fd\lambda)\\
        &=q_i\sum_{l=1}^{n}
        \left(\sum_{j\in\mathcal{J}}\frac{a_{l,j}}{b_l}(\widetilde{w_i})_j\right)q_{(l)}(u,\avint_{K_{l}} fd\lambda)\\
        &=q_i\sum_{j\in\mathcal{J}}(\widetilde{w_i})_j\sum_{l=1}^{n}
        \frac{a_{l,j}}{b_l}q_{(l)}(u,\avint_{K_{l}} fd\lambda)\\
        &=q_i(\widetilde{w_i},v)
        \end{align*}
        Now, if $v=0$ then this expression is zero for all $i$. That is $u$ is orthogonal to the subspace spanned by $\{\avint_{A_{i,1}}hd\lambda:i\geq 1\}$.
 If $v\not=0$, then by what we have established before, there exists $i\geq 1$ such that $$( u,\avint_{A_{i,1}}hd\lambda) =q_i(\widetilde{w_i},v) >0.$$ 
This means that necessarily $0\in \Conv(\{\avint_{A_{i,1}}hd\lambda:i\geq 1\})$.
            Indeed, suppose that $0\notin \Conv(\{\avint_{A_{i,1}}hd\lambda:i\geq 1\})$. Then, by the Hahn-Banach theorem, there exists $u\in \RR^{\mathcal{J}}$ such that $( u, \int_{A_{i,1}}hd\lambda) <0$ for all $i$'s.
            However, this would contradict the fact that for every non-zero $v$ there exists $\widetilde{w_i}$ such that $(\widetilde{w_i}, v) >0$. We conclude that $0\in \Conv(\{\avint_{A_{i,1}}hd\lambda:i\geq 1\})$, and therefore we can find  (appealing to Theorem \ref{t_ljap}) a subset $A'\subseteq A$ of positive measure on which $h$ is mean zero.\\
        
        We now set $A_{i,j}' := T^{j-1}(A_{i,1}\cap A')$ for $i\geq 1$ and $1\leq j\leq q_i$. Furthermore we set $A_0' := \bigcup_{i\geq 1}\bigcup_{j=1}^{q_i}A_{i,j}'$ and define $T':A_0'\to A_0'$ as $T'|_{A_{i,j}'} = T|_{A_{i,j}'}$ for $j<q_i$ and as $T'|_{A_{i,q_i}'} = T|_{A_{i,q_i}'}^{-q_i+1}$. Last we set $h' = h|_{A_0'}$. Now all properties of the partition 
        $\{A_{i,j}':i\geq 1, 1\leq j\leq q_i\}$ of $A_0'$, the function $h'$ on $A'$ and the transformation $T'$ of $A_0'$ are satisfied, except for the fact that $A_0'\not=[0,1]$. However, by Zorn's lemma we can iterate this argument to obtain a partition of the entire interval $[0,1]$ and this completes the proof.
    \end{proof}
\end{lem}

Now, we are fully prepared to start the proof of our main result.
We explain the main idea of the proof.
Intuitively, in order to solve the equation for the function $f$, we build another bounded mean zero function $\widetilde{h}^{(1)}$ on a smaller domain. We then solve the equation for the function $\widetilde{h}^{(1)}$, and from this we will obtain a solution for the function $f$ itself. However, the way that we solve the equation for the function $\widetilde{h}^{(1)}$ is done by building yet another bounded mean zero function $\widetilde{h}^{(2)}$, on an even smaller domain, and solving the equation for this function. It follows inductively that we first have to build an entire sequence of bounded mean zero functions $(\widetilde{h}^{(k)})_{k\geq 0}$ on nested domains $A^{(0)}\supseteq A^{(1)}\supseteq ...$. Once we have done that we can in fact solve the equation for all these functions simultaneously. In particular we find a solution for the function $f$. Moreover, by adding coordinate functions to the function $\widetilde{h}^{(k)}$ in every layer of the construction, we can ensure that the final transformation is also ergodic.\\

We now prove the following result that gives us Theorem \ref{theorem:main-result} and Theorem \ref{theorem:bound-partial-sums} simultaneously.
\begin{Thm}
	Let $V$ be a finite dimensional normed vector real space.
	Let $f\in L_\infty([0,1];V)$ be a mean zero function and let $\varepsilon>0$. Then there exists a function $g\in L_\infty([0,1];V)$ with $\|g\|_\infty\leq (S_V+\varepsilon)\|f\|_\infty$, and an ergodic mod $0$ measure preserving invertible transformation $T$ of $[0,1]$ such that $f = g\circ T - g$.
	
	Furthermore, there exists a set $X\subseteq [0,1]$ of positive measure such that we have for $k\geq 0$ the bound on the partial sums $\|\sum_{j=0}^k f\circ T^j\|_{L_\infty(X;V)} \leq (S_V+\varepsilon)\|f\|_\infty$.

	\begin{proof}
		Let $(V,\|\cdot\|_V)$, $f$ and $\varepsilon$ be given as stated.   Denote $d=\dim(V)$. We shall for $k\geq 0$ write $V_k = V\times \RR^{k}$ for the $(d+k)$-dimensional vector space with norm $\|(v,w)\|_{V_k} = \|v\|_{V} + \|w\|_1$. Let $\{D_l\}_{l\geq 1}$ be an enumeration of all the sets $$\{\bigcup_{i=1}^N(a_i,b_i): N\in \NN, a_i,b_i\in\QQ \text{ with }0\leq a_i\leq b_i\leq 1\}.$$
		We define corresponding mean zero functions $Z_{l}^{(0)}: A^{(0)}\to \RR$ given by 
		$$Z_l^{(0)} = \chi_{D_l} - \lambda(D_l)$$
		
		We can assume that $f\not=0$ since otherwise the statement is trivial.
		We first set
		$\varepsilon' = \frac{\varepsilon\|f\|_\infty}{2(S_V+1)}>0$ and for $k\geq 0$ set
		
		\begin{align}\label{eq:definition-epsilonk}
		\varepsilon'_k = \frac{\varepsilon'}{2^{k+2}(d+k+1)}>0
		\end{align} 
		We now define $\widetilde{h}^{(0)} := f$ and $A^{(0)} := [0,1]$.
		Since $\widetilde{h}^{(0)}$ is mean zero we can by Lemma \ref{lemma:countable-partition-of-set} find a sequence $(q_i^{(0)})_{i\geq 1}$ in $\NN$ with $q_i^{(0)}\geq 2$, a partition $\{A_{i,j}^{(0)}:i\geq 1,1\leq j\leq q_i^{(0)}\}$ of $A^{(0)}$ and a measure preserving {invertible} transformation $T^{(0)}$ of $A^{(0)}$ for which $T^{(0)}(A_{i,j}^{(0)})=A_{i,j+1}^{(0)}$ for $j=1,...,q_i^{(1)}-1$ and $T^{(0)}(A_{i,q_i^{(0)}}^{(0)})=A_{i,1}^{(0)}$.
		Furthermore, this can be done so that, if we denote $A^{(1)}=\bigcup_{i\geq 1}A_{i,1}^{(1)}$ and define the function
		$h^{(1)}:A^{(1)}\to V_0$ by 
		$h^{(1)}|_{A_{i,1}^{(1)}} = \sum_{j=1}^{q_i^{(0)}}\widetilde{h}^{(0)}\circ (T^{(0)})^{j-1}$, we have that $h^{(1)}$ is mean zero and $\|h^{(1)}\|_\infty\leq \varepsilon'_1$. Moreover Lemma \ref{lemma:countable-partition-of-set} gives us for $i\geq 1$ and $1\leq l\le q_i$ the bound $\|\sum_{j=1}^{l-1}\widetilde{h}^{(0)}\circ (T^{(0)})^{j-1}\|_{L_\infty(A_{i,1}^{(0)};V)}\leq S_{V}\|\widetilde{h}^{(0)}\|_\infty+\varepsilon'_1$.
		
		For $l\geq 1$ define a function
		$Z_l^{(1)}:A^{(1)}\to \RR$ as $Z_l^{(1)}|_{A_{i,1}^{(0)}} = \sum_{j=1}^{q_i^{(0)}} Z_{l}^{(0)}\circ (T^{(0)})^{j-1}$.
		We note that
		\begin{align*}
		\|Z_l^{(1)}\|_1 &=\sum_{i\geq 1}\int_{A_{i,1}^{(0)}}|Z_l^{(1)}|d\lambda \\
		&\leq\sum_{i\geq 1}\sum_{j=1}^{q_i^{(0)}}\int_{A_{i,1}^{(0)}}|Z_l^{(0)}|\circ (T^{(0)})^{j-1}d\lambda \\ 
		&=\sum_{i\geq 1}\sum_{j=1}^{q_i^{(0)}}\int_{A_{i,j}^{(0)}}|Z_l^{(0)}|d\lambda \\      
		&=\int_{A^{(0)}}|Z_{l}^{(0)}|d\lambda\\ 
		&= \|Z_l^{(0)}\|_1
		\end{align*} which shows that $Z_l^{(1)}\in L_1(A^{(1)};\RR)$. Furthermore, in fact       
		
		\begin{align*}
		\int_{A^{(1)}}Z_l^{(1)}d\lambda &=\sum_{i\geq 1}\int_{A_{i,1}^{(0)}}Z_l^{(1)}d\lambda \\
		&=\sum_{i\geq 1}\sum_{j=1}^{q_i^{(0)}}\int_{A_{i,1}^{(0)}}Z_l^{(0)}\circ (T^{(0)})^{j-1}d\lambda \\ 
		&=\sum_{i\geq 1}\sum_{j=1}^{q_i^{(0)}}\int_{A_{i,j}^{(0)}}Z_l^{(0)}d\lambda \\      
		&=\int_{A^{(0)}}Z_{l}^{(0)}d\lambda\\ 
		&= 0  
		\end{align*}
		which shows that $Z_l^{(1)}$ is  mean zero.
		As the bounded functions are dense in $L_1(A^{(1)};\RR)$, we can find a $\hat{Z}_1^{(1)} \in L_\infty(A^{(1)};\RR)$ with $\|Z_{1}^{(1)}-\hat{Z}_1^{(1)}\|_1 \leq \frac{\varepsilon_1'}{2}$.
		As $Z_{1}^{(1)}$ is mean zero, we then moreover have $|\int_{A^{(1)}}\hat{Z}_1^{(1)}d\lambda|\leq \|Z_{1}^{(1)}-\hat{Z}_1^{(1)}\|_1 \leq \frac{\varepsilon_1'}{2}$.
		We now define a mean zero function $\widetilde{Z}_{1}^{(1)}$ in $L_\infty(A^{(1)};\RR)$ as $\widetilde{Z}_1^{(1)} = \hat{Z}_1^{(1)} - \avint_{A^{(1)}}\hat{Z}_1^{(1)}d\lambda$. We obtain that 
		\begin{align*}
		\|Z_{1}^{(1)} - \widetilde{Z}_{1}^{(1)}\|_1 \leq \|Z_{1}^{(1)} - \hat{Z}_1^{(1)}\|_1 + |\int_{A^{(1)}}\hat{Z}_1^{(1)}d\lambda| \leq \frac{\varepsilon_1'}{2} + \frac{\varepsilon_1'}{2} = \varepsilon_1'
		\end{align*}
		Now define $\widetilde{h}^{(1)}\in L_\infty([0,1];V_1)$ as 
		$\widetilde{h}^{(1)} = (h^{(1)}, \frac{\varepsilon_1'\widetilde{Z}_1^{(1)}}{\|\widetilde{Z}_1^{(1)}\|_\infty + 1})$. We then have
		$$\|\widetilde{h}^{(1)}\|_\infty = \|h^{(1)}\|_\infty + \frac{\varepsilon_1'\|\widetilde{Z}_1^{(1)}\|_\infty}{\|\widetilde{Z}_1^{(1)}\|_\infty + 1}\leq 2\varepsilon_1'$$
		
		As $\widetilde{h}^{(1)}$ is mean zero we can apply the same construction as for $\widetilde{h}^{(0)}$ to this function.\\

		We thus see that inductively for $k\geq 0$ we can find
		\begin{itemize}
			\item A sequence $(q_i^{(k)})_{i\geq 1}$ in $\NN$ with $q_i^{(k)}\geq 2$.
			\item A partition $\{A_{i,j}^{(k)}: i\geq 1,1\leq j\leq q_i^{(k)}\}$ of $A^{(k)}$ of sets of positive measure.
			\item A measure preserving {invertible} map
			$T^{(k)}:A^{(k)}\to A^{(k)}$ defined with
			$T^{(k)}(A_{i,j}^{(k)})=A_{i,j+1}^{(k)}$ for $j<q_i^{(k)}$ and $T^{(k)}(A_{i,q_i}^{(k)})=A_{i,1}^{(k)}$.
			\item A set $A ^{(k+1)}=\bigcup_{i\geq 1}A_{i,1}^{(k)}$.
			\item A mean zero function $h^{(k+1)}:A^{(k+1)}\to V_k$ given by 
			\begin{align}\label{eq:definition-hk}
			h^{(k+1)}|_{A_{i,1}^{(k)}} = \sum_{j=1}^{q_i^{(k)}} \widetilde{h}^{(k)}\circ (T^{(k)})^{j-1}.
			\end{align}
			\item For $l\geq 1$ mean zero functions $Z_l^{(k+1)}\in L_1(A^{(k+1)},\RR)$ given by
			\begin{align}\label{eq:definition-Zlk}
			Z_{l}^{(k+1)}|_{A_{i,1}^{(k)}} = \sum_{j=1}^{q_i^{(k)}} Z_{l}^{(k)}\circ (T^{(k)})^{j-1}.
			\end{align}
			\item A mean zero function $\widetilde{Z}_{k+1}^{(k+1)}\in L_\infty(A^{(k+1)};\RR)$ with 
            \begin{align}\label{eq:Z-tilde-norm}
               \|Z_{k+1}^{(k+1)}-\widetilde{Z}_{k+1}^{(k+1)}\|_1\leq \varepsilon_{k+1}'.
            \end{align}
           	\item A mean zero function $\widetilde{h}^{(k+1)}\in L_\infty(A^{(k+1)};V_{k+1})$ given by 
			\begin{align}\label{eq:definition-hk-tilde}
			\widetilde{h}^{(k+1)} = (h^{(k+1)},\frac{\varepsilon_k' \widetilde{Z}_{k+1}^{(k+1)}}{\|\widetilde{Z}_{k+1}^{(k+1)}\|_\infty + 1}).
			\end{align}
		\end{itemize}
		Furthermore, for $k\geq 0$ the construction gives the bounds
		\begin{align*}
		\|h^{(k+1)}\|_\infty &\leq \varepsilon'_{k+1},\\
		\|\widetilde{h}^{(k+1)}\|_\infty &\leq 2\varepsilon'_{k+1}
		\end{align*}
		and the bounds for $i\geq 1$ and $j=1,...,q_i^{(k)}$
		\begin{align}\label{eq:definition-norm}
		\|\sum_{l=1}^{j-1}\widetilde{h}^{(k)}\circ (T^{(k)})^{l-1}\|_{L_\infty(A_{i,1}^{(k)};S_{V_k})}&\leq S_{V_k}\|\widetilde{h}^{(k)}\|_\infty+\varepsilon'_k.
		\end{align}
		
		We will now turn to defining the transformation $T$, the function $g$ and the set $X$. For $k\geq 0$ define a mapping $P^{(k)}:A^{(k)}\to A^{(k+1)}$ as 
		\begin{align}\label{eq:definition-P}
		P^{(k)}|_{A_{i,j}^{(k)}} := (T^{(k)})^{1-j}
		\end{align} which `projects' a point in $A^{(k)}$ to a point in $A^{(k+1)}$. 
		We now define measure preserving {invertible} transformations $T_k:[0,1]\to [0,1]$ for $k\geq 0$ as follows. We define $T_k|_{A_{i,j}^{(k)}} := T^{(k)}$ for $i\geq 1$ and $1\leq j< q_i^{(k)}$, define $T_k|_{A_{i,q_i^{(k)}}^{(k)}} := P^{(k)}|_{A_{i,q_i^{(k)}}^{(k)}}$ for $i\geq 1$, and we define $T_k|_{[0,1]\setminus A^{(k)}} = \Id_{[0,1]\setminus A^{(k)}}$.
		We now define, for $k_0\geq 0$ transformations $R_{k_0}:A^{(k_0)}\to A^{(k_0)}$ as 
		\begin{align}\label{eq:definition-Rk}
		R_{k_0} := \lim\limits_{N\to \infty}T_N \circ T_{N-1}\circ ... \circ T_{k_0}|_{A^{(k_0)}}
		\end{align}
		where convergence is taken with respect to the measure topology. Indeed the limit exists due to the fact that 
		$A^{(k+1)}\subseteq A^{(k)}$ for $k\geq 1$ and $\lim\limits_{k\to\infty}\lambda(A^{(k)})=0$, and $T_{k'}|_{[0,1]\setminus A^{(k)}} = \Id_{[0,1]\setminus A^{(k)}}$ for all $k'\geq k$. Likewise the limit of the inverses exists. Now, since the maps $T_k$ for $k\geq 0$ are all measure preserving we have by by Lemma \ref{lemma:measure_convergent}, that $R_{k_0}$ is a mod $0$ measure preserving invertible transformation. We define our final transformation as $T := R_0$.\\
		
		Furthermore, define $g_k:A^{(k)}\to V_k$ as 
		\begin{align}\label{eq:definition-gk}
		g_k|_{A_{i,j}^{(k)}} := \left(\sum_{l=1}^{j-1}\widetilde{h}^{(k)}\circ (T^{(k)})^{l-1}\right) \circ P^{(k)} 
		\end{align}
		Note here that on $A^{(k+1)}= \bigcup_{i\geq 1}A_{i,1}^{(k)}$ the function $g_k$ is defined as $g_k|_{A^{(k+1)}}=0$.
		We now define for integers $k_1\geq k_2$ coordinate projections $p_{k_1,k_2}:V_{k_1}\to V_{k_2}$ as 
		\begin{align}\label{eq:definition-coordinate-projection}
		p_{k_1,k_2}(v,w_1,...w_{k_1}) := (v,w_1,...,w_{k_2})
		\end{align}
		
		We now define for $k_0\geq 0$ functions $r_{k_0}:A^{(k_0)}\to V_{k_0}$ as
		\begin{align}\label{eq:definition-rk}
		r_{k_0} &:= \sum_{j=k_0}^{\infty}p_{j,k_0} \circ g_j\circ P^{(j-1)}\circ ... \circ P^{(k_0)}
		\end{align}
		{We show that these series converge.
		Namely, for $k\geq 0$ we have that $S_{V_k}\leq \dim(V_k) = d+k$ \cite{grinberg_value_1980}, and so we obtain for $k_0\geq 0$ that
		\begin{align*}
		\sum_{{j}=k_0}^{\infty}\|p_{{j},k_0}\circ g_j\|_\infty
		&\leq \sum_{k=0}^{\infty}\|p_{k,k_0}\circ g_k\|_\infty\\
		&\leq \sum_{k=0}^{\infty}\|g_k\|_\infty\\
		&\leq \|g_0\|_\infty + \sum_{k\geq 1}\max_{\substack{{i\geq 1}\\ {1\leq j\leq q_{i}^{(k)}}}}\|g_k|_{A_{i,j}^{(k)}}\|_\infty\\ 
		&\stackrel{\eqref{eq:definition-norm}}{\leq}\|g_0\|_\infty +  \sum_{k\geq 1}\left(S_{V_{k}}\|\widetilde{h}^{(k)}\|_\infty + \varepsilon'_k\right)\\   
		&\leq \|g_0\|_\infty + \sum_{k\geq 1}(2S_{V_{k}}+1)\varepsilon'_k\\
		&\stackrel{\eqref{eq:definition-epsilonk}}{\leq} \|g_0\|_\infty + \sum_{k\geq 1}\frac{\varepsilon'}{2^{k+1}}\\
		&\leq \|g_0\|_\infty + \varepsilon'<\infty,
		\end{align*}
		which shows that the series from \eqref{eq:definition-rk} converge absolutely, and shows that $r_{k_0}\in L_\infty(A^{(k_0)};V_{k_0})$ is well-defined.}
		
		We now define our function $g$ as $g := r_0$ and the set $X$ as $X := A^{(1)}$.
		We may now prove the statements from the theorem.\\

		{1) We start by proving the bound on $\|g\|_{\infty}$.
		For this, it follows from the previous calculation that
		\begin{align*}
		\|g\|_\infty &\leq \sum_{k\geq 0}\|p_{k,0}\circ g_k\|_\infty\\
		&\leq \|g_0\|_\infty + \varepsilon'\\
		&\leq \left(\max_{\substack{i\geq 1\\ 1\leq j\leq q_{i}^{(0)}}}\|g_0|_{A_{i,j}^{(0)}}\right) + \varepsilon'\\
		&\stackrel{\eqref{eq:definition-norm}}{\leq}(S_{V_0}\|\widetilde{h}^{(0)}\|_\infty + \varepsilon_0') + \varepsilon'\\
		&\leq S_V\|f\|_\infty +  2\varepsilon'\\
		&\leq (S_V+\varepsilon)\|f\|_\infty,
		\end{align*}
		which gives us the bound on $g$.}\\

		2) We shall now turn to show that $\widetilde{h}^{(k_0)} = r_{k_0}\circ R_{k_0} - r_{k_0}$ holds for $k_0\geq 0$. In particular this will show the equation $f = g\circ T - g$.
		
		For $x\in A^{(k_0)}$ and $k\geq k_0$ define 
		\begin{align}\label{eq:definition_xk}
		x_k:=P^{(k-1)}\circ ...\circ P^{(k_0)}(x)\in A^{(k)}
		\end{align}
		Note here that $x_{k_0}$ is simply defined as $x_{k_0} = x$.
		Further, for $k\geq k_0$ denote
		$$B^{(k)}:=\bigcup_{i\geq 1}A_{i,q_i^{(k)}}^{(k)}$$
		If $x_k\in B^{(k)}$ for some $k\geq k_0$, we have $T_k(x_k) = x_{k+1}$. Therefore, in the case that $x_k\in B^{(k)}$ for all $k\geq k_0$, we have that $R_{k_0}(x) \in \bigcap_{k\geq 1} A^{(k)}$. Since $R_{k_0}$ is measure preserving and $\lim\limits_{k\to \infty}\lambda(A^{(k)})=0$ we thus have for almost all $x\in A^{(k_0)}$ that there is a $N\geq k_0$ with $x_N\not\in B^{(N)}$. Let us denote $N(x)$ for the minimal integer (greater or equal to $k_0$) with this property. We let $x\in A^{(k_0)}$ be s.t. $N(x)$ is finite. For $k=k_0,...,N(x)-1$, we have $x_k\in B^{(k)}$, and so $T_k(x_k) = x_{k+1}$. This means that 
		\begin{align}\label{eq:expression_xN}
		x_{N(x)} = T_{N(x)-1}\circ...\circ T_{k_0}(x)
		\end{align}
			Next, by \eqref{eq:definition_xk} and by definition of $N(x)$, we have that $x_{N(x)}\in A^{(N(x))}\setminus B^{(N(x))}$, and therefore $x_{N(x)}\in A_{i,j}^{(N(x))}$ for some $i\geq 1$ and $1\leq j\leq q_i^{(N(x))}-1$.
		This guarantees that 
		\begin{align}\label{eq:image-xN-one-up}
		T_{N(x)}(x_{N(x)}) = T^{(N(x))}(x_{N(x)})\in A_{i,j+1}^{(N(x))}
		\end{align} and hence 
		\begin{align}\label{eq:PTx=Px1}
        \begin{split}
		P^{(N(x))}\circ T_{N(x)}(x_{N(x)}) &= (T^{(N(x))})^{1-(j+1)}(T^{(N(x))}(x_{N(x)}))\\
		&=(T^{(N(x))})^{1-j}(x_{N(x)}) \\
		&= P^{(N(x))}(x_{N(x)})
        \end{split}
		\end{align}
		From \eqref{eq:image-xN-one-up} it follows moreover that $T_{N(x)}(x_{N(x)})\in A^{(N(x))}\setminus A^{(N(x)+1)}$.
		Using this we obtain
		\begin{align}\label{eq:expression-Rx1}
        \begin{split}
		R_{k_0}(x) &\stackrel{\eqref{eq:definition-Rk}}{=} \lim\limits_{M\to \infty}T_M\circ ...\circ T_{k_0}(x)\\ 
		&\stackrel{\eqref{eq:expression_xN}}{=} \lim\limits_{M\to\infty}T_M\circ ...\circ T_{N(x)}(x_{N(x)})\\
		&= T_{N(x)}(x_{N(x)}) \in A^{(N(x))}
        \end{split}
		\end{align}
		From (\ref{eq:PTx=Px1}) and (\ref{eq:expression-Rx1}) it follows that
		\begin{align}\label{eq:PR=P}
			P^{(N(x))}\circ R_{k_0}(x) &= P^{(N(x))}\circ T_{N(x)}(x_{N(x)}) = P^{(N(x))}(x_{N(x)})
		\end{align}

		We note that for $k\geq k_0$ we have by definition that $P^{(k)}$ acts as the identity on $A^{(k+1)}$. From this and the fact from (\ref{eq:expression-Rx1}) that $R_{k_0}(x)\in A^{(N(x))}\subseteq ...\subseteq A^{(k_0)}$, it follows for $k=k_0,...,N(x)-1$ that 
        \begin{align}\label{eq:Rk-is-P-invariant}
		P^{(k)}\circ ... \circ P^{(k_0)}\circ R_{k_0}(x) = R_{k_0}(x)
		\end{align}
		
		We now obtain
		\begin{align*}
		P^{(N(x))}\circ ... \circ P^{(k_0)}\circ R_{k_0}(x) &\stackrel{\eqref{eq:Rk-is-P-invariant}}{=}
		P^{(N(x))}\circ R_{k_0}(x) \\
		&\stackrel{\eqref{eq:PR=P}}{=} P^{(N(x))}(x_{N(x)})\\
		&\stackrel{\eqref{eq:definition_xk}}{=} P^{(N(x))}\circ ...\circ P^{(k_0)}(x)
		\end{align*}
		More generally, we now find for $M\geq N(x)$ that
		\begin{align}\label{eq:Rk-vanishes}
		P^{(M)}\circ ... \circ P^{(k_0)}\circ R_{k_0}(x)
		&= P^{(M)}\circ ...\circ P^{(k_0)}(x)
		\end{align}
		We now calculate
		\begin{align*}
		(r_{k_0}\circ R_{k_0}& - r_{k_0})(x) =\\ &\stackrel{\eqref{eq:definition-rk}}{=} \sum_{k=k_0}^{\infty}p_{k,k_0}\circ g_k \circ P^{(k-1)}\circ ... \circ P^{(k_0)}\circ R_{k_0}(x) \\
		&-  \sum_{k=1}^{\infty}p_{k,k_0}\circ g_k \circ P^{(k-1)}\circ ... \circ P^{(k_0)}(x)\\
		&\stackrel{\eqref{eq:Rk-vanishes}}{=} \sum_{k=k_0}^{N(x)}p_{k,k_0}\circ g_k \circ P^{(k-1)}\circ ... \circ P^{(k_0)}\circ R_{k_0}(x) \\
		&-  \sum_{k=k_0}^{N(x)}p_{k,k_0}\circ g_k \circ P^{(k-1)}\circ ... \circ P^{(k_0)}(x)\\
		&\stackrel{\text{\eqref{eq:definition_xk}, \eqref{eq:Rk-is-P-invariant}}}{=} \sum_{k=k_0}^{N(x)}p_{k,k_0}\circ g_k \circ R_{k_0}(x) -  \sum_{k=k_0}^{N(x)}p_{k,k_0}\circ g_k(x_k)
		\end{align*}  
		We thus find
		\begin{align}\label{eq:expression-mid}
		(r_{k_0}\circ R_{k_0} - r_{k_0})(x) &= \sum_{k=k_0}^{N(x)}p_{k,k_0}\circ (g_k\circ R_{k_0}(x) - g_k(x_{k}))
		\end{align}  
		We shall now inspect the summands on the right hand side to show that this expression equals $\widetilde{h}^{(k_0)}(x)$.
		
		Again, by the definition of $N(x)$ we have that $x_{N(x)}\in A_{i,j}^{(N(x))}$ for some $i\geq 1$ and $1\leq j\leq q_i^{N(x)}-1$, and by \eqref{eq:image-xN-one-up} we then have $T_{N(x)}(x_{N(x)})\in  A_{i,j+1}^{(N(x))}$. Using this fact together with the definition of $g_{N(x)}$ we obtain that
		\begin{align*}
		g_{N(x)}\circ T_{N(x)}(x_{N(x)}) &- g_{N(x)}(x_{N(x)}) =\\ 
		&\stackrel{\text{(\ref{eq:definition-gk})}}{=} \left(\sum_{l=1}^{j}\widetilde{h}^{(N(x))}\circ (T^{(N(x))})^{l-1}\circ P^{(N(x))}\circ T_{N(x)}(x_{N(x)})\right)  \\
		&-\left(\sum_{l=1}^{j-1}\widetilde{h}^{(N(x))}\circ (T^{(N(x))})^{l-1}\circ P^{(N(x))}(x_{N(x)})\right)\\
		&\stackrel{\text{(\ref{eq:PTx=Px1})}}{=} \left(\sum_{l=1}^{j}\widetilde{h}^{(N(x))}\circ (T^{(N(x))})^{l-1} \circ P^{(N(x))}(x_{N(x)})\right)  \\
		&- \left(\sum_{l=1}^{j-1}\widetilde{h}^{(N(x))}\circ (T^{(N(x))})^{l-1}\circ P^{(N(x))}(x_{N(x)})\right)\\
		&=\widetilde{h}^{(N(x))}\circ (T^{N(x)})^{j-1}\circ P^{(N(x))}(x_{N(x)})\\
		&\stackrel{\text{\eqref{eq:definition-P}}}{=}\widetilde{h}^{(N(x))}\circ (T^{N(x)})^{j-1}\circ (T^{N(x)})^{1-j}(x_{N(x)})\\
		&=\widetilde{h}^{(N(x))
		}(x_{N(x)})
		\end{align*}
		Combining this calculation with (\ref{eq:expression-Rx1}) we find
		\begin{align}\label{eq:expression-hN}
		\widetilde{h}^{(N(x))}(x_{N(x)}) = g_{N(x)}\circ R_{k_0}(x) - g_{N(x)}(x_{N(x)})
		\end{align}       
		
		Now fix $k_0\leq k\leq N(x)-1$. Then since by definition of $N(x)$ we have $x_k\in B^{(k)}$, we find that for some $i\geq 1$ we have $x_k \in A_{i,q_{i}^{(k)}}^{(k)}$. Also we then have $P^{(k)}(x_k)\in A_{i,1}^{^{(k)}}$ by definition of $P^{(k)}$. We now calculate	
		\begin{align*}
		g_k(x_k) &\stackrel{\text{\eqref{eq:definition-gk}}}{=} \sum_{j=1}^{q_{i}^{(k)}-1}
		\widetilde{h}^{(k)}\circ (T^{(k)})^{j-1}\circ P^{(k)}(x_k)\\
		&= \left(\sum_{j=1}^{q_{i}^{(k)}}
		\widetilde{h}^{(k)}\circ (T^{(k)})^{j-1}\circ P^{(k)}(x_k)\right) -  
		\left(\widetilde{h}^{(k)}\circ (T^{(k)})^{q_{i}^{(k)}-1}\circ P^{(k)}(x_k)\right)\\
		&\stackrel{\text{\eqref{eq:definition-hk}}}{=}
		h^{(k+1)}(P^{(k)}(x_k))-	\left(\widetilde{h}^{(k)}\circ (T^{(k)})^{q_{i}^{(k)}-1}\circ P^{(k)}(x_k)\right)\\
		&\stackrel{\text{\eqref{eq:definition-P}}}{=}
		h^{(k+1)}(P^{(k)}(x_k))-	\left(\widetilde{h}^{(k)}\circ (T^{(k)})^{q_{i}^{(k)}-1}\circ (T^{(k)})^{1-q_{i}^{(k)}}(x_k)\right)\\
		&=h^{(k+1)}(P^{(k)}(x_k))-	\widetilde{h}^{(k)}(x_k)\\
		&\stackrel{\text{\eqref{eq:definition-hk-tilde}, \eqref{eq:definition_xk}}}{=}p_{k+1,k}\circ  \widetilde{h}^{(k+1)}(x_{k+1})-	\widetilde{h}^{(k)}(x_k)
		\end{align*} 
		We note that for $k\geq k_0$ we have by definition of $g_k$ that $g_k|_{A^{(k+1)}} = 0$. Hence, as by (\ref{eq:expression-Rx1}) we have that $R_{k_0}(x) \in A^{(N(x))}\subseteq .... \subseteq A^{(k_0)}$, we find for $k=k_0,....,N(x)-1$ that 
		$g_k(R_{k_0}(x)) = 0$.
		By previous calculation we conclude for $k=k_0,...,N(x)-1$ that 
		\begin{align}\label{eq:expression-difference-gk}
		g_k\circ R_{k_0}(x) - g_k(x_k) &= -g_k(x_k) = \widetilde{h}^{(k)}(x_k) - p_{k+1,k}\circ \widetilde{h}^{(k+1)}(x_{k+1})
		\end{align}
		Finally, we obtain 
		\begin{align*}
		(r_{k_0}\circ R_{k_0} - r_{k_0})(x) &\stackrel{\text{\eqref{eq:expression-mid}}}{=} \sum_{k=k_0}^{N(x)}p_{k,k_0}\circ \left[g_k\circ R_{k_0}(x) - g_k(x_{k})\right]\\
		&\stackrel{\text{\eqref{eq:expression-hN}}}{=} \left(\sum_{k=k_0}^{N(x)-1}p_{k,k_0}\circ \left[g_k\circ R_{k_0}(x) - g_k(x_{k})\right]\right)\\ 
		&+ p_{N(x),k_0}\circ \widetilde{h}^{N(x)}(x_{N(x)})\\
		&\stackrel{\text{\eqref{eq:expression-difference-gk}}}{=} \left(\sum_{k=k_0}^{N(x)-1}p_{k,k_0}\circ \left[\widetilde{h}^{(k)}(x_k) - p_{k+1,k}\circ \widetilde{h}^{(k+1)}(x_{k+1})\right]\right)\\ 
		&+ p_{N(x),k_0}\circ \widetilde{h}^{N(x)}(x_{N(x)})\\
		&\stackrel{\text{\eqref{eq:definition-coordinate-projection}}}{=} \left(\sum_{k=k_0}^{N(x)-1}p_{k,k_0}(\widetilde{h}^{(k)}(x_k)) - p_{k+1,k_0}(\widetilde{h}^{(k+1)}(x_{k+1}))\right)\\ 
		&+ p_{N(x),k_0}(\widetilde{h}^{N(x)}(x_{N(x)}))\\
		&= p_{k_0,k_0}(\widetilde{h}^{(k_0)}(x_{k_0}))\\
		&\stackrel{\text{\eqref{eq:definition-coordinate-projection}}}{=} \widetilde{h}^{(k_0)}(x_{k_0})\\
		&\stackrel{\text{\eqref{eq:definition_xk}}}{=} \widetilde{h}^{(k_0)}(x)
		\end{align*}  
		This shows that for $k\geq 0$ the equation $\widetilde{h}^{(k_0)} = r_{k_0}\circ R_{k_0} - r_{k_0}$ holds. As $f = \widetilde{h}^{(0)}$, $g = r_0$ and $T = R_0$ this gives in particular that $f = g\circ T - g$.\\
		
		3) We shall prove ergodicity of the map $T$. Fix $k\geq 0$ and $i\geq 1$. We shall first show that the equation
		\begin{align}
		R_{k}^{q_{i}^{(k)}}|_{A_{i,1}^{(k)}} = R_{k+1}|_{A_{i,1}^{(k)}}
		\end{align} holds. Let $x\in A_{i,j}^{(k)}$ for some $j=1,..,q_{i}^{(k)}-1$ Then by definition 
		\begin{align}\label{eq:T-plus-one-ergodic-part}
		T_k(x) &= T^{(k)}(x) \in A_{i,j+1}^{(k)}
		\end{align} In particular $T_k(x)\not\in A^{(k+1)}\supseteq A^{(k+2)}\supseteq ...$. Now as for $M\geq k+1$ we have that $T_M$ is the identity on $A^{(M-1)}\setminus A^{(M)}$, we find inductively for $M\geq k$ that 
		\begin{align}\label{eq:Tk-is-identity}
		T_M\circ ... \circ T_{k}(x) = T_k(x) 
		\end{align}
		This shows us that 
		\begin{align*}
		R_{k}(x) &\stackrel{\text{\eqref{eq:definition-Rk}}}{=} \lim\limits_{M\to \infty}T_{M}\circ ... T_k(x)\\
		&\stackrel{\text{\eqref{eq:Tk-is-identity}}}{=}T_k(x)\\ 
		&\stackrel{\text{\eqref{eq:T-plus-one-ergodic-part}}}{=} T^{(k)}(x)\in A_{i,j+1}^{(k)}
		\end{align*}
		Now if $y\in A_{i,1}^{(k)}$, then it follows inductively that
		for $j=1,...,q_i^{(k)}$ we have 
		\begin{align}\label{eq:power-of-Rk}
		R_k^{j-1}(y) = (T^{(k)})^{j-1}(y) \in A_{i,j}^{(k)}
		\end{align}
		Now put $z := R_k^{q_{i}^{(k)}-1}(y)$, then since  $z\in A_{i,q_{i}^{(k)}}^{(k)}$ we have by definition of $T_k$ and $P^{(k)}$ that 
		\begin{align}\label{eq:Tk-projection}
		T_k(z) = P^{(k)}(z) = (T^{(k)})^{1-q_{i}^{(k)}}(z) \stackrel{\text{\eqref{eq:power-of-Rk}}}{=} y \in A_{i,1}^{(k)}\subseteq A^{(k+1)}
		\end{align}
		We now obtain
		\begin{align*}
		R_{k}^{q_i^{(k)}}(y) &= R_{k}(z) \\
		&\stackrel{\text{\eqref{eq:definition-Rk}}}{=} \lim\limits_{M\to \infty}T_{M}\circ ... \circ T_{k}(z) \\
		&= \left(\lim\limits_{M\to \infty}T_{M} \circ ...\circ T_{k+1}\right)|_{A^{(k+1)}} \circ T_{k}(z)\\
		&\stackrel{\text{\eqref{eq:definition-Rk}}}{=}  R_{k+1} \circ T_{k}(z)\\
		&\stackrel{\text{\eqref{eq:Tk-projection}}}{=} R_{k+1}(y)
		\end{align*}
		We conclude that indeed
		\begin{align}\label{eq:higher-powers-Rk}
		R_{k}^{q_{i}^{(k)}}|_{A_{i,1}} = R_{k+1}|_{A_{i,1}}
		\end{align}
		Also we note that for $j=1,..,q_{i}^{(k)}$ we obtained from \eqref{eq:power-of-Rk} that
		\begin{align}\label{eq:Rk-powers}
        \begin{split}
		R_k^{j-1}|_{A_{i,1}^{(k)}} &= (T^{(k)})^{j-1}|_{A_{i,1}^{(k)}},\\
		R_{k}^{j-1}(A_{i,1}^{(k)}) &= A_{i,j}^{(k)}
       \end{split}
      \end{align}

		Let $k\geq 0$ and let $F\subseteq A^{(k)}$ be a set of positive measure that is $R_{k}$-invariant. Now since
		{
		\begin{align*}
		R_{k+1}(F\cap A^{(k+1)}) &= \bigcup_{i\geq 1}R_{k+1}(F\cap A^{(k)}_{i,1})\\
		&\stackrel{\text{\eqref{eq:higher-powers-Rk}}}{=} \bigcup_{i\geq 1}R_{k}^{q_{i}^{(k)}}(F\cap A^{(k)}_{i,1})
		\subseteq F
		\end{align*}
	and since we also know that $R_{k+1}(A^{(k+1)}) \subseteq A^{(k+1)}$, by definition of the map $R_{k+1}$, we find that $R_{k+1}(F\cap A^{(k+1)}) \subseteq F\cap A^{(k+1)}$, which is to say that $F\cap A^{(k+1)}$ is $R_{k+1}$-invariant.}\\
		
		Now, we shall fix a $T$-invariant set $D\subseteq [0,1]$ of positive measure and we show that such set must have measure $\lambda(D)=1$. By what we just showed, it follows inductively for $k\geq 0$ that  $D\cap A^{(k)}$ is $R_{k}$-invariant.
		Now fix $k\geq 1$.
		Using that $D\cap A^{(k-1)}$ is $R_{k-1}$-invariant we then find for $i\geq 1$ and $j=1,....,q_{i}^{(k-1)}$ that
		\begin{align}\label{eq:sets-mapped-to1}
        \begin{split}
		R_{k-1}^{j-1}(D\cap A_{i,1}^{(k-1)}) &= R_{k-1}^{j-1}(D\cap A^{(k-1)})\cap R_{k-1}^{j-1}(A_{i,1}^{(k-1)})\\
		&\stackrel{\text{\eqref{eq:Rk-powers}}}{=}(D\cap A^{(k-1)}) \cap A_{i,j}^{(k-1)}\\
		&= D\cap A_{i,j}^{(k-1)}
        \end{split}
		\end{align}
		Now for $l\geq 1$ we find
		\begin{align*}
		\int_{D\cap A^{(k)}}Z_{l}^{(k)}d\lambda &= \sum_{i=1}^\infty \int_{D\cap A_{i,1}^{(k-1)}} Z_l^{(k)}d\lambda\\
		&\stackrel{\text{\eqref{eq:definition-Zlk}}}{=}\sum_{i=1}^\infty \int_{D\cap A_{i,1}^{(k-1)}} \sum_{j=1}^{q_i^ {(k-1)}}Z_l^{(k-1)}\circ (T^{(k-1)})^{j-1} d\lambda \\
		&\stackrel{\text{\eqref{eq:Rk-powers}}}{=}\sum_{i=1}^\infty\sum_{j=1}^{q_i^ {(k-1)}} \int_{D\cap A_{i,1}^{(k-1)}} Z_l^{(k-1)}\circ R_{k-1}^{j-1} d\lambda \\
		&\stackrel{\text{\eqref{eq:sets-mapped-to1}}}{=}\sum_{i=1}^\infty\sum_{j=1}^{q_i^ {(k-1)}} \int_{D\cap A_{i,j}^{(k-1)}} Z_l^{(k-1)} d\lambda\\
		&= \int_{D\cap A^{(k-1)}}Z_l^{(k-1)}d\lambda
		\end{align*}
		This shows for $k\geq 0$ and $l\geq 1$ that $$\int_{D\cap A^{(k)}}Z_l^{(k)}d\lambda = \int_{D\cap A^{(0)}}Z_l^{(0)} d\lambda = \lambda(D\cap D_l) - \lambda(D)\lambda(D_l).$$
		
		We note that by step (2) of the proof we have for $k\geq 0$ that $\widetilde{h}^{(k)} = r_k\circ R_k - r_k$. Now, by \eqref{eq:definition-hk-tilde} we have for $k\geq 1$ that the function $\widetilde{Z}_k^{(k)}$ can be written as an $R_k$-coboundary, since $\frac{\varepsilon_k'}{\|\widetilde{Z}_k^{(k)}\|_\infty +1}\widetilde{Z}_k^{(k)} = Y_k\circ R_k - Y_k$ where $Y_k$ is the last coordinate function of $r_k$. From this and the fact that $D\cap A^{(k)}$ is $R_k$-invariant, it follows that $\int_{D\cap A^{(k)}}\widetilde{Z}_k^{(k)}d\lambda = 0$. Now as $\|Z_{k}^{(k)} - \widetilde{Z}_k^{(k)}\|_1\leq \varepsilon_k'$ (see (\ref{eq:Z-tilde-norm})) this gives for $k\geq 1$ that  $$|\lambda(D\cap D_k) - \lambda(D)\lambda(D_k)|\leq \epsilon_k'.$$        
		We shall now show that $\lambda(D) = \lambda(D)^2$. Let $\rho>0$. By regularity of the Lebesgue measure, there exists an open $U$ s.t. $D\subseteq U$ and $\lambda(U\setminus D)<\rho$. Now $U$ can essentially be written as a countable union of disjoint open intervals, that is
		$$U = \bigcup_{i=1}^\infty (a_i,b_i)$$ 
		with $a_i,b_i\in \QQ$. Hence, there is an integer $l\geq 1$  s.t. $D_l\subseteq U$ and $\lambda(U\setminus D_l)\leq \rho$. We can moreover choose $l$ large enough so that $\varepsilon_l'\leq \rho$.
		From the bounds on $\lambda(U\setminus D)$ and $\lambda(U\setminus D_l)$, we find for the symmetric difference $D\Delta D_l$ that $\lambda(D\Delta D_l)\leq 2\rho$. 
		Hence $|\lambda(D) - \lambda(D\cap D_l)|\leq 2\rho$ and $|\lambda(D_l) - \lambda(D)|\leq 2\rho$ and we find that
		\begin{align*}
		|\lambda(D) - \lambda(D)^2|&\leq|\lambda(D) - \lambda(D\cap D_l)|\\
		&+ |\lambda(D\cap D_l) - \lambda(D)\lambda(D_l)|\\ 
		&+ |\lambda(D)\lambda(D_l) - \lambda(D)^2|\\
		&\leq 2\rho + \varepsilon_l' + 2\rho\lambda(D)\\
		&\leq 5\rho
		\end{align*}
		As $\rho>0$ was arbitrary, it now follows that $\lambda(D) = \lambda(D)^2$. This gives us $\lambda(D)=1$ and proves the ergodicity.\\
		
		4) We now prove the statement for the set $X$. As $X = A^{(1)}$ it is clear that $X$ has positive measure. Now to obtain the stated bound, we note that the function $g_0$ is such that $g_0|_{A_{i,1}^{(0)}} = 0$ for $i\geq 1$. Since $X = A^{(1)} = \bigcup_{i\geq 1}A_{i,1}^{(0)}$ this means that $g_0|_{X} = 0$.
		Analogues to the bound $\|g\|_\infty\leq (S_V+\frac{\varepsilon}{2})\|f\|_\infty$, we now obtain
		\begin{align*}
		\|g\|_{L_\infty(X;V)} &\leq \sum_{k\geq 1}\|g_k\|_\infty\\
		&\leq \varepsilon'\\
		&\leq \frac{\varepsilon}{2}\|f\|_\infty
		\end{align*}
		This gives us for $k=0,1,2,...$ that 
		\begin{align*}
		\|\sum_{j=0}^{k}f\circ T^j\|_{L_\infty(X;V)}&\leq \|g\circ T^{k+1} - g\|_{L_\infty(X;V)}\\
		&\leq \|g\circ T^{k+1}\|_\infty + \|g\|_{L_\infty(X;V)}\\
		&\leq (S_V + \varepsilon)\|f\|_\infty
		\end{align*}
		which proves the statement.
	\end{proof}
\end{Thm}

\begin{lem}\label{l_b_list}
    Let $f:[0,1]\longrightarrow\RR^d$, $f_1,\dots,f_d$ be components of  $f$. Denote by $P_i:\RR^d\longrightarrow\RR$ the projection onto $i$-th coordinate. We have
     
    (i). $f^{-1}(X_1\times\dots\times X_d)=\bigcap_{i=1}^df_i^{-1}(X_i)$ for any $X_1,\dots,X_d\subset \RR$;
    
    (ii). if $f$ be a measurable function, then $\sigma(f)\subset \sigma(f_1)\times\dots\times\sigma(f_d)$;
    
    (iii). if $f\in L_\infty([0,1];\RR^d)$, then $\sigma(f)$ is compact in $\RR^d$ and $\sigma(f_i)\subset P_i(\sigma(f)),\ i=1,\dots,d$;
    
    (iv). $f\in L_\infty([0,1];\RR^d)\Leftrightarrow f_1,\dots,f_d \in L_\infty[0,1]$;
    
    (v). if $f\in L_\infty([0,1];\RR^d)$ and the norm on $\RR^d$ be such, that $|P_i(\cdot)|\leq\|\cdot\|$, then $\|f_i\|_\infty\leq\|f\|_\infty,\ i=1,\dots,d$.
\end{lem}
\begin{proof}
First of all, let us observe that functions $f_1,\dots,f_d$ are measurable if and only if $f$ is measurable \cite[Lemma 2.12.5]{Bogachev1}.

    (i). Indeed, $t\in f^{-1}(X_1\times\dots\times X_d)\Leftrightarrow f(t)\in X_1\times\dots\times X_d\Leftrightarrow f_i(t)\in X_i,\ i=1,\dots,d\Leftrightarrow t\in\bigcap_{i=1}^df_i^{-1}(X_i)$.
    
    (ii). Let $v\in \sigma(f),\ U_i$ be neighbourhoods of $P_i(v)$ in $\RR$. Then $U:=U_1\times\dots\times U_d$ is a  neighbourhood of $v$. By (i), we have $f^{-1}(U)\subset f_i^{-1}(U_i),\ i=1,\dots,d$. Therefore $\lambda(f_i^{-1}(U_i))>0$, i.e   
    $P_i(v)\in \sigma(f_i),\ i=1,\dots,d$. Hence, 
    $$\sigma(f)\subset \sigma(f_1)\times\dots\times\sigma(f_d).$$
   
    (iii). Since $f\in L_\infty([0,1];\RR^d)$, it follows that $\sigma(f)$ is bounded in $\RR^d$. Thus, it remains only prove that $\sigma(f)$ is closed. Assume that $\sigma(f)\ni v_n\rightarrow v$. Then for every neighbourhood $U$ of the point $v$, there exists an index $n$, so that $v_n\in U$. In this case, $\lambda(f^{-1}(U))>0$. Hence, $v\in \sigma(f)$, in other words $\sigma(f)$ is a compact.    
    
Let $1\leq i\leq d,\ t\in \sigma(f_i)$. By (i), we have 
$$\lambda(f^{-1}(\RR^{i-1}\times [t-\dfrac{1}{n},t+\dfrac{1}{n}]\times\RR^{d-i}))=\lambda(f_i^{-1}([t-\dfrac{1}{n},t+\dfrac{1}{n}]))>0$$ for every $n\in\NN$. So, we have that 
$$K_n:=(\RR^{i-1}\times [t-\dfrac{1}{n},t+\dfrac{1}{n}]\times\RR^{d-i})\cap\sigma(f)$$ is a non-empty compact set in  $\RR^d$ for every $n\in\NN$. Observing that $\{K_n\}_{n=1}^\infty$ is a centered system of compacts, we infer that 
$$\sigma(f)\cap(\RR^{i-1}\times \{t\}\times\RR^{d-i})=\bigcap_{n=1}^\infty K_n\neq\emptyset.
$$ 
In particular, $t\in P_i(\sigma(f))$ and therefore $\sigma(f_i)\subset P_i(\sigma(f)),\ i=1,\dots,d$. 
    
    (iv) follows from a combination of  (ii) and (iii).
    
    (v). There exists $r\in\sigma(f_i)$, such that $\|f_i\|_\infty=|r|$. Due to (iii), we know that $r=P_i(v)$ for some $v\in\sigma(f)$. Then $\|f_i\|_\infty=|r|=|P_i(v)|\leq\|v\|\leq\sup\{\|w\|:\ w\in\sigma(f)\}=\|f\|_\infty$.
\end{proof}
\begin{proof}[Proof of Theorem \ref{theorem:complex}]
    Let $f\in L_\infty([0,1])$ be a complex-valued mean zero function, $f_1:=\Re (f),\ f_2:=\Im (f)\in L_\infty[0,1]$.    
    Then $\tilde{f}:=(f_1,f_2)\in L_\infty([0,1];\RR^2)$ (Lemma \ref{l_b_list}(iv)), (on $\RR^2$ we consider the Euclidean norm $\|\cdot\|$).
    
Theorem \ref{theorem:main-result} guarantees that there exist $\tilde{g}\in L_\infty([0,1];\RR^2)$ and an ergodic mod $0$ measure preserving invertible transformation $T$ of $[0,1]$ that 
$$
\tilde{f} = \tilde{g}\circ T - \tilde{g},\ \|\tilde{g}\|_\infty\leq(S_{{\RR^2}}+\varepsilon)\|\tilde{f}\|_\infty=\|f\|_\infty.
$$
 Let $\tilde{g}=(g_1,g_2)$, then $g:=g_1+ig_2\in L_\infty[0,1]$ (see Lemma \ref{l_b_list}(iv)), $\|g\|_\infty=\|\tilde{g}\|_\infty$ and $$f = g\circ T - g,\ \|g\|_\infty\leq(\dfrac{\sqrt{5}}{2}+\varepsilon)\|f\|_\infty,$$
    since $S_{{\RR^2}}=\frac{\sqrt{5}}{2}$ \cite[Theorem 2]{banaszczyk1987},\cite{banaszczyk1990}.
\end{proof}

Another interesting extension of Theorem  \ref{theorem:previos-result}  may be stated for an arbitrary finite collection of real valued mean zero functions.
\begin{Thm}\label{theorem:list-result}
    Let $f_1,\dots,f_n\in L_\infty[0,1]$ be mean zero real-valued functions. For any $\varepsilon>0$ there exists an ergodic mod 0 measure preserving invertible transformation  $T$ and real-valued functions $g_1,\dots,g_n\in L_\infty[0,1]$ with $\|g_i\|_\infty\leq (n+\varepsilon)\|f_i\|_{\infty}$  such that $f_i= g_i\circ T -g_i,\ i=1,\dots,n$.
\end{Thm}
\begin{proof}
    Without loss of generality, we may assume that $\|f_i\|_\infty\neq 0,\ i=1,\dots,n$. On $\RR^n$ we consider the norm 
    $$\|v\|=\max_i(|v_i|),\ {\rm where}\ v=(v_1,\dots,v_n).
    $$
    Consider the  function $$\tilde{f}=(\tilde{f_1},\dots,\tilde{f_n}):[0,1]\longrightarrow\RR^n,$$ where $\tilde{f_i}=\dfrac{f_i}{\|f_i\|_\infty},\ i=1,\dots,n$.  It follows from Lemma \ref{l_b_list}(iv),(ii) that 
    $$\tilde{f}\in L_\infty([0,1];\RR^n),\ \|\tilde{f}\|_\infty\leq 1.
    $$ It is straightforward that $\tilde{f}$ is a mean zero function.
    
    By Theorem \ref{theorem:main-result}, there exists $\tilde{g}\in L_\infty([0,1];\RR^n)$ and an ergodic mod 0 measure preserving invertible transformation  $T$ of $[0,1]$ that $$\tilde{f} = \tilde{g}\circ T - \tilde{g},\ \|\tilde{g}\|_\infty\leq n+\varepsilon,$$
    since $S_{\RR^n}\leq n$ \cite{grinberg_value_1980}.
    
    Let $\tilde{g}=(\tilde{g_1},\dots,\tilde{g_n})$, then $\tilde{g_1},\dots,\tilde{g_n}\in L_\infty[0,1]$ and $\|\tilde{g_i}\|_\infty\leq\|\tilde{g}\|_\infty,\ i=1,\dots,d$ (Lemma \ref{l_b_list}(iv),(v)). Therefore 
    $$\|\tilde{g_i}\|_\infty\leq n+\varepsilon,\ \tilde{f_i}=\tilde{g_i}\circ T-\tilde{g_i},\ i=1,\dots,n.$$
    It remains to set $g_i=\|f_i\|_\infty\tilde{g_i},\ i=1,\dots,n.$
\end{proof}


\begin{bibdiv}
\begin{biblist}
	
	\bib{Anosov}{article}{
		author={Anosov, D.~V.},
		title={The additive functional homology equation that is connected with
			an ergodic rotation of the circles},
		date={1973},
		ISSN={0373-2436},
		journal={Izv. Akad. Nauk SSSR Ser. Mat.},
		volume={37},
		pages={1259\ndash 1274},
		review={\MR{0352400}},
	}
	
	\bib{AdamsPublished}{incollection}{
		author={Adams, Terry},
		author={Rosenblatt, Joseph},
		title={Joint coboundaries},
		date={2017},
		booktitle={Dynamical systems, ergodic theory, and probability: In memory of
			{{Kolya Chernov}}},
		series={Contemp. {{Math}}.},
		volume={698},
		publisher={{Amer. Math. Soc., Providence, RI}},
		pages={5\ndash 33},
		review={\MR{3716084}},
	}
	
	\bib{Adams2}{article}{
		author={Adams, Terry},
		author={Rosenblatt, Joseph},
		title={Existence and {{Non}}-existence of {{Solutions}} to the
			{{Coboundary Equation}} for {{Measure Preserving Systems}}},
		language={en},
		date={2019-10},
		journal={arXiv:1902.09045 [math]},
		eprint={1902.09045},
	}
	
	\bib{banaszczyk1987}{article}{
		author={Banaszczyk, Wojciech},
		title={The {{Steinitz}} constant of the plane},
		date={1987},
		ISSN={0075-4102},
		journal={J. Reine Angew. Math.},
		volume={373},
		pages={218\ndash 220},
		review={\MR{870312}},
	}
	
	\bib{banaszczyk1990}{article}{
		author={Banaszczyk, Wojciech},
		title={A note on the {{Steinitz}} constant of the {{Euclidean}} plane},
		date={1990},
		ISSN={0706-1994},
		journal={C. R. Math. Rep. Acad. Sci. Canada},
		volume={12},
		number={4},
		pages={97\ndash 102},
		review={\MR{1070418}},
	}
	
	\bib{banaszczyk1990_2}{article}{
		author={Banaszczyk, Wojciech},
		title={The {{Steinitz}} theorem on rearrangement of series for nuclear
			spaces},
		date={1990},
		ISSN={0075-4102},
		journal={J. Reine Angew. Math.},
		volume={403},
		pages={187\ndash 200},
		review={\MR{1030415}},
	}
	
	\bib{bbs_full_2019}{article}{
		author={Ber, Aleksei},
		author={Borst, Matthijs},
		author={Sukochev, Fedor},
		title={Full proof of {{Kwapie\'n}}'s theorem on representing bounded
			mean zero functions on $[0,1]$},
		date={2021},
		ISSN={0039-3223},
		journal={Studia Math.},
		volume={259},
		number={3},
		pages={241\ndash 270},
		review={\MR{4269477}},
	}
	
	\bib{baranyCombinatorialQuestionsFinitedimensional1981}{article}{
		author={B{\'a}r{\'a}ny, I.},
		author={Grinberg, V.~S.},
		title={On some combinatorial questions in finite-dimensional spaces},
		date={1981},
		ISSN={0024-3795},
		journal={Linear Algebra Appl.},
		volume={41},
		pages={1\ndash 9},
		review={\MR{649713}},
	}
	
	\bib{Bogachev1}{book}{
		author={Bogachev, V.~I.},
		title={Measure theory. {{Vol}}. {{I}}},
		publisher={{Springer-Verlag, Berlin}},
		date={2007},
		ISBN={978-3-540-34513-8},
		review={\MR{2267655}},
	}

	\bib{Bogachev2}{book}{
	author={Bogachev, V.~I.},
	title={Measure theory. {{Vol}}. {{II}}},
	publisher={{Springer-Verlag, Berlin}},
	date={2007},
	ISBN={978-3-540-34513-8},
	review={\MR{2267655}},
	}
	
	\bib{Bourgain}{article}{
		author={Bourgain, Jean},
		title={Translation invariant forms on $l_p(g)$ ($1<p<\infty$)},
		date={1986},
		ISSN={0373-0956},
		journal={Ann. Inst. Fourier (Grenoble)},
		volume={36},
		number={1},
		pages={97\ndash 104},
		review={\MR{840715}},
	}
	
	\bib{Browder}{article}{
		author={Browder, Felix~E.},
		title={On the iteration of transformations in noncompact minimal
			dynamical systems},
		language={en},
		date={1958-05},
		ISSN={0002-9939},
		journal={Proc. Amer. Math. Soc.},
		volume={9},
		number={5},
		pages={773\ndash 773},
	}
	
	\bib{GTM}{book}{
		author={Conway, John~B.},
		title={A course in functional analysis},
		edition={Second},
		series={Graduate {{Texts}} in {{Mathematics}}},
		publisher={{Springer-Verlag, New York}},
		date={1990},
		volume={96},
		ISBN={978-0-387-97245-9},
		review={\MR{1070713}},
	}
	
	\bib{FK}{incollection}{
		author={Figiel, Tadeusz},
		author={Kalton, Nigel},
		title={Symmetric linear functionals on function spaces},
		date={2002},
		booktitle={Function spaces, interpolation theory and related topics
			({{Lund}}, 2000)},
		publisher={{de Gruyter, Berlin}},
		pages={311\ndash 332},
		review={\MR{1943290}},
	}
	
	\bib{grinberg_value_1980}{article}{
		author={Grinberg, V.~S.},
		author={Sevast'yanov, S.~V.},
		title={Value of the {{Steinitz}} constant},
		language={en},
		date={1980},
		ISSN={0016-2663, 1573-8485},
		journal={Functional Analysis and Its Applications},
		volume={14},
		number={2},
		pages={125\ndash 126},
	}
	
	\bib{Hardy_Wright}{book}{
		author={Hardy, G.~H.},
		author={Wright, E.~M.},
		title={An introduction to the theory of numbers},
		edition={Sixth},
		publisher={{Oxford University Press, Oxford}},
		date={2008},
		ISBN={978-0-19-921986-5},
		review={\MR{2445243}},
	}
	
	\bib{kadets_series_1997}{book}{
		author={Kadets, M.~I.},
		author={Kadets, V.~M.},
		title={Series in {{Banach}} spaces: Conditional and unconditional
			convergence},
		language={eng},
		series={Operator Theory, Advances and Applications},
		publisher={{Birkh\"auser Verlag}},
		address={{Basel ; Boston}},
		date={1997},
		number={vol. 94},
		ISBN={978-0-8176-5401-6 978-3-7643-5401-5},
	}
	
	\bib{Kolmogorov}{article}{
		author={Kolmogorov, A.~N.},
		title={On dynamical systems with an integral invariant on the torus},
		date={1953},
		journal={Doklady Akad. Nauk SSSR (N.S.)},
		volume={93},
		pages={763\ndash 766},
		review={\MR{0062892}},
	}
	
	\bib{Kwapien}{article}{
		author={Kwapie{\'n}, Stanis{\l}law},
		title={Linear {{Functionals Invariant Under Measure Preserving
					Transformations}}},
		language={en},
		date={1984},
		ISSN={0025584X, 15222616},
		journal={Math. Nachr.},
		volume={119},
		number={1},
		pages={175\ndash 179},
	}
	
	\bib{SingularTraces}{book}{
		author={Lord, Steven},
		author={Sukochev, Fedor},
		author={Zanin, Dmitriy},
		title={Singular traces},
		series={De {{Gruyter Studies}} in {{Mathematics}}},
		publisher={{De Gruyter, Berlin}},
		date={2013},
		volume={46},
		ISBN={978-3-11-026250-6 978-3-11-026255-1},
		review={\MR{3099777}},
	}
	
	\bib{LT1}{book}{
		author={Lindenstrauss, Joram},
		author={Tzafriri, Lior},
		title={Classical {{Banach}} spaces. {{II}}},
		series={Ergebnisse Der {{Mathematik}} Und Ihrer {{Grenzgebiete}}
			[{{Results}} in {{Mathematics}} and {{Related Areas}}]},
		publisher={{Springer-Verlag, Berlin-New York}},
		date={1979},
		volume={97},
		ISBN={978-3-540-08888-2},
		review={\MR{540367}},
	}
	
	\bib{Simon}{book}{
		author={Simon, Barry},
		title={Convexity: {{An Analytic Viewpoint}}},
		publisher={{Cambridge University Press}},
		address={{Cambridge}},
		date={2011},
		ISBN={978-0-511-91013-5},
	}
	
	\bib{Steinitz_1913}{article}{
		author={Steinitz, Ernst},
		title={{Bedingt konvergente Reihen und konvexe Systeme.}},
		language={de},
		date={1913-01},
		ISSN={1435-5345},
		volume={1913},
		number={143},
		pages={128\ndash 176},
	}
	
\end{biblist}
\end{bibdiv}
\end{document}